\documentclass[a4paper,11pt,twoside]{article}
\usepackage{mathrsfs}
\usepackage{amssymb}
\usepackage{amsmath}
\usepackage{mathrsfs}
\usepackage{amsthm}
\usepackage{amsfonts}
\usepackage{color}
\usepackage{latexsym}
\usepackage{indentfirst}
\usepackage{extarrows}%²Å±àÏÂÃæ¶þÐÐ
\usepackage{amscd}
\usepackage{bbm}
\usepackage{graphicx}
\usepackage[all]{xy}
\usepackage[mathscr]{eucal}
\usepackage{subfigure}
\newcommand{\K}{{\mathbb K}}

\newcommand{\R}{{\mathbb R}}

\allowdisplaybreaks
\newtheorem{theorem}{Theorem}[section]

\newtheorem{corollary}[theorem]{Corollary}

\newtheorem{definition}[theorem]{Definition}

\newtheorem{remark}[theorem]{Remark}

\newtheorem{lemma}[theorem]{Lemma}

\newtheorem{proposition}[theorem]{Proposition}
\newtheorem{claim}[theorem]{Claim}

\newtheorem{assumption}[theorem]{Assumption}

\numberwithin{equation}{section}

\begin{document}

\title{\bf\Large
Bifurcation of periodic and antiperiodic solutions\\
in non-autonomous potential-type delay systems
	\footnotetext{\hspace{-0.35cm} 2020
{\it Mathematics Subject Classification}.
Primary 34K18, 34K20.
\endgraf
{\it Key words and phrases.}
Delay or advanced differential equations,
Bifurcation, Periodic or antiperiodic solutions, Maslov index.
}}
\date{}
\author{Guangcun Lu\footnote{Beijing Normal University;
E-mail: \texttt{gclu@bnu.edu.cn}/{October 4, 2026}.}}

\maketitle

\vspace{-0.7cm}

\begin{center}
\begin{minipage}{13cm}
{\small {\bf Abstract}\quad
As a continuation of our previous work, where bifurcations of periodic and antiperiodic orbits near an equilibrium were studied for autonomous differential delay systems with one or two delays, this paper investigates bifurcations of periodic and antiperiodic solutions for five classes of non-autonomous potential delay differential systems with multiple delays. By reformulating these systems as generalized Hamiltonian systems on Euclidean spaces, we apply the Hamiltonian bifurcation theory recently developed by the author to establish alternative bifurcation results of both Fadell--Rabinowitz and Rabinowitz type.}
\end{minipage}
\end{center}

\vspace{0.2cm}
	
\tableofcontents

\vspace{0.2cm}

\section{Introduction}\label{sec:Intro0}
\setcounter{equation}{0}

%Since the work of Nussbaum \cite{Nu75} and Walther \cite{Wal83}, many authors have investigated bifurcations from periodic solutions in delay differential equations (DDEs) using various methods. See, for example,
%\cite{ArHb90, Ben24, BiHanWu03, ChowM78,ChowWa88, DiVLW, Do89, ElSM05, Er09, Fa06, Guo26, HanM05, HanBi04, HaV93, HaW04, HbQe06, Lop, LSBK23, ThWal05, Wal06, Wal}
%and the references therein.

Since the work of Nussbaum \cite{Nu75} and
Kaplan-Yorke \cite{KaYo74}, bifurcations of periodic solutions in differential delay equations (DDEs) have been extensively studied by various methods; see, e.g.,
\cite{ArHb90, Ben24, BiHanWu03,
ChowM78,ChowWa88, DiVLW, Do89, ElSM05, Er09, Fa06, Guo26, Hale, HaV93, HaW04, HanM05, HanBi04, HbQe06, Lop, LSBK23, ThWal05, Wal78, Wal83, Wal06, Wal}
 and the references therein.

%However, the bifurcation problem for systems with [????????(1.1)???1.1???????????????? (antiperiodic boundary conditions)/??????????????] remains less explored, especially in connection with [????????????????Maslov??]. The objective of this paper is to fill this gap by investigating...

The primary objective of this paper is to investigate bifurcation phenomena in five classes of differential-delay systems.
The first three classes are described by the following problems:
 \begin{equation}\label{e:Delay1}
 \left\{\begin{array}{ll}
 \dot{x}( t ) = \sum^{m-1}_{k=1}\nabla_3 V (\lambda, t , x ( t - k\tau ) ),\\
 %+ \nabla_3 V (\lambda, t , x ( t - 2 \tau ) ) +
%  \cdots + \nabla_3 V (\lambda, t , x ( t - ( m - 1 ) \tau ) ),\\
  x( t + m \tau ) = - x( t )\;\forall t,
   \end{array}\right.
  \end{equation}
  \begin{equation}\label{e:Delay2}
  \left\{\begin{array}{ll}
 \dot{x}( t ) = \sum^{m-1}_{k=1} J_n\nabla_3 G (\lambda, t , x ( t -k \tau ) ),\\
  %+ \nabla_3 G (\lambda, t , x ( t - 2 \tau ) ) +
%  \cdots + \nabla_3 G (\lambda, t , x ( t - ( m - 1 ) \tau ) )],\\
    x( t + m \tau ) =  x( t )\;\forall t,
   \end{array}\right.
  \end{equation}
  and
 \begin{equation}\label{e:Delay3}
 \left\{\begin{array}{ll}
 \ddot{x}( t ) = - \sum^{m-1}_{k=1}\nabla_3 U (\lambda, t , x ( t - k\tau ) ),\\
  %+ \nabla_3 U (\lambda, t , x ( t - 2 \tau ) ) +
%  \cdots + \nabla_3 U (\lambda, t , x ( t - ( m - 1 ) \tau ) )],  \\
    x( t + m \tau ) =  x( t )\;\forall t.
   \end{array}\right.
 \end{equation}
 where $\tau>0$ is a real, $m\ge 2$ an integer,  $\Lambda$ a topological space, and $U,V:\Lambda\times \mathbb{R}\times{\R}^{n}\to\R$
 and $G:\Lambda\times \mathbb{R}\times{\R}^{2n}\to\R$ satisfy the following:

 \begin{assumption}\label{ass:BasiAss1Delay1}
{\rm
The functions $U, V: \Lambda \times \mathbb{R} \times \mathbb{R}^n \to \mathbb{R}$ and
$G: \Lambda \times \mathbb{R} \times \mathbb{R}^{2n} \to \mathbb{R}$
are continuous, and $\tau$-periodic in $t$. Additionally, $V$ is even in $x \in \mathbb{R}^n$. For each fixed $(\lambda,t) \in \Lambda \times \mathbb{R}$, the mappings
$U(\lambda,t,\cdot), V(\lambda,t,\cdot): \mathbb{R}^n \to \mathbb{R}$ and $G(\lambda,t,\cdot): \mathbb{R}^{2n} \to \mathbb{R}$
are of class $C^2$, with all partial derivatives depending continuously on $(\lambda,t,x)$.}
\end{assumption}

The above $\nabla_3V(\lambda,t,\cdot)$ and $\nabla_3U(\lambda,t,\cdot)$ denote the Euclidean gradients of $V(\lambda,t,\cdot)$ and $U(\lambda,t,\cdot)$ in $\mathbb{R}^{n}$, respectively, while $\nabla_3G(\lambda,t,\cdot)$ denotes that of $G(\lambda,t,\cdot)$ in $\mathbb{R}^{2n}$.

%%%%%%%%%%%%%%%%%%%%%%%%%%%%%%%%%%%%%
%In the autonomous case, where $V, G, U$ in Assumption~\ref{ass:BasiAss1Delay1} are %time-independent, the local bifurcation near the trivial solution for \eqref{e:Delay1} %($m\in\{2,3\}$) and for \eqref{e:Delay2}--\eqref{e:Delay3} ($m=2$) was analyzed in \cite{Lu14}.

In the autonomous case, where $V, G, U$ in Assumption~\ref{ass:BasiAss1Delay1} are time-independent, the local bifurcation near the trivial solution for
 (\ref{e:Delay1}) ($m\in\{2,3\}$)  and for (\ref{e:Delay2})--(\ref{e:Delay3})  ($m=2$)
 was analyzed in \cite{Lu14}.
The present paper studies the bifurcation of systems \eqref{e:Delay1}--\eqref{e:Delay3}
near a family of solutions that are not necessarily trivial.
To this end, we impose the following assumptions.

\begin{assumption}\label{ass:BasiAss1Delay2}
{\rm Let $V$ satisfy  Assumption~\ref{ass:BasiAss1Delay1}, and for each $\lambda\in\Lambda$ let $x^\lambda:\R\to\mathbb{R}^{n}$
be a solution  of (\ref{e:Delay1}); moreover, $\Lambda\times \R\ni (\lambda,t)\mapsto x^\lambda(t)\in\mathbb{R}^{n}$ is also continuous.
 ({\it Note}: the map $\mathbb{R}\ni t\mapsto 0\in\mathbb{R}^n$,
  denoted by ${\bf 0}^\lambda$ without occurrence of confusions,
   satisfies (\ref{e:Delay1}) because $V(\lambda,x)$ is also even in variable $x\in\mathbb{R}^n$.)}
\end{assumption}

\begin{assumption}\label{ass:BasiAss1Delay3}
{\rm Let $G$ satisfy Assumption~\ref{ass:BasiAss1Delay1}, and for each $\lambda\in\Lambda$ let $x^\lambda:\R\to\mathbb{R}^{2n}$
be a  solution of (\ref{e:Delay2}), such that the map $\Lambda\times \R\ni (\lambda,t)\mapsto x^\lambda(t)\in\mathbb{R}^{2n}$
 is also continuous.
}
\end{assumption}

 \begin{assumption}\label{ass:BasiAss1Delay4}
{\rm Let $U$ satisfy Assumption~\ref{ass:BasiAss1Delay1}. For each $\lambda\in\Lambda$, let $x^\lambda:\R\to\mathbb{R}^{n}$
be a  solution of (\ref{e:Delay3}),  and assume that  maps
$$
\Lambda\times \R\ni (\lambda,t)\mapsto x^\lambda(t)\in\mathbb{R}^{n}\quad\hbox{and}\quad
\Lambda\times \R\ni (\lambda,t)\mapsto \dot{x}^\lambda(t)\in\mathbb{R}^{n}
$$
are also continuous.
}
\end{assumption}

\begin{definition}\label{def:Delay}
{\rm
\begin{itemize}
\item[\rm (I)] Under Assumption~\ref{ass:BasiAss1Delay2} (resp. Assumption~\ref{ass:BasiAss1Delay3}),
 $(\mu, x^\mu)$ is said to be a \textsf{bifurcation point along sequences }of (\ref{e:Delay1}) [resp. (\ref{e:Delay2})]
 with respect to the branch $\{(\lambda, x^\lambda)\,|\,\lambda\in\Lambda\}$ if
 there exists a sequence $(\lambda_k)\subset\Lambda$ converging to $\mu\in\Lambda$
 and solutions $\bar{x}^k\ne x^{\lambda_k}$  of (\ref{e:Delay1})  [resp. (\ref{e:Delay2})]
 with $\lambda=\lambda_k\in\Lambda$,  $k=1,2,\dots$, such that
 ${\bar{x}^k}|_{[0,\tau]}\to {x^{\mu}}|_{[0,\tau]}$ in $C^{0}([0,\tau];\R^{n})$
or $C^1([0,\tau];\R^{n})$.
\item[\rm (II)] Under Assumption~\ref{ass:BasiAss1Delay4} we say
 $(\mu, x^\mu)$ to be a \textsf{bifurcation point along sequences} of (\ref{e:Delay3})
 with respect to the branch $\{(\lambda, x^\lambda)\,|\,\lambda\in\Lambda\}$ if
 there exists a sequence $(\lambda_k)\subset\Lambda$ converging to $\mu\in\Lambda$
 and solutions $\bar{x}^k\ne x^{\lambda_k}$  of (\ref{e:Delay3})
 with $\lambda=\lambda_k\in\Lambda$,  $k=1,2,\dots$, such that
 ${\bar{x}^k}|_{[0,\tau]}\to {x^{\mu}}|_{[0,\tau]}$ in $C^{1}([0,\tau];\R^{n})$
or $C^2([0,\tau];\R^{n})$.
\end{itemize}
 Similarly, we may define the\textsf{ notions of  bifurcation points} of (\ref{e:Delay1}), (\ref{e:Delay2}) and (\ref{e:Delay3})
  with respect to the branch $\{(\lambda, x^\lambda)\,|\,\lambda\in\Lambda\}$, respectively.}
\end{definition}

%Liu \cite{Liu12} developed the ideas of \cite{KaYo74}
% and recast the study of periodic solutions in delay differential equations (\ref{e:Delay1})-(\ref{e:Delay3})
% into a generalized  Hamiltonian boundary problem as in (\ref{e:M-invariantDelay4}),

%\textbf{Section~\ref{sec:Remark}}. For suitable Hamiltonian functions $H(\lambda,t,x,y)$
%on $\Lambda \times \mathbb{R} \times \mathbb{R}^{n} \times \mathbb{R}^{n}$,

%We further establish bifurcation results for delay and advanced differential systems subject to an anti-periodic boundary condition. Specifically, we consider
 The fourth class of the delay system is
\begin{equation}\label{e:crm1}
\dot{x}(t) = -\nabla_4 H(\lambda, t, x(t), x(t-\tau)), \qquad x(t + 2\tau) = -x(t),
\end{equation}
as well as the advanced system
\begin{equation}\label{e:crm1Ad}
\dot{x}(t) = -\nabla_4 H(\lambda, t, x(t), x(t+\tau)), \qquad x(t + 2\tau) = -x(t).
\end{equation}
Here, $\tau > 0$ denotes a constant delay (for \eqref{e:crm1}) or advance (for \eqref{e:crm1Ad}), $\Lambda$ is a topological parameter space, and the Hamiltonian $H: \Lambda \times \mathbb{R} \times \mathbb{R}^n \times \mathbb{R}^n \to \mathbb{R}$ is continuous. We impose the following regularity and symmetry assumptions on $H$.

\begin{assumption}\label{ass:Crm1}
		\begin{enumerate}
		\item[\rm (i)] {\bf Regularity.} For every $(\lambda, t) \in \Lambda \times \mathbb{R}$, the mapping $z\equiv (x, y) \mapsto H(\lambda, t, x, y)$ is  $C^2$. Moreover,  the Euclidean gradient and the Hessian of  $H(\lambda,t, x)$ with respect to
$z\in{\R}^{2n}$, $\nabla_zH(\lambda,t, z)$ and $\nabla^2_zF(\lambda,t, z)$,
are continuous in  $(\lambda, t, z)\in\Lambda\times \R\times\mathbb{R}^{2n}$.

		\item[\rm (ii)] {\bf Symmetries.} For all $(\lambda, t, x, y) \in \Lambda \times \mathbb{R} \times \mathbb{R}^n \times \mathbb{R}^n$,
		\[
		H(\lambda, t + \tau, x, y) = H(\lambda, t, x, y) = H(\lambda, t, y, -x).
		\]
	\end{enumerate}
\end{assumption}

Equation \eqref{e:crm1Ad} was first investigated in \cite{DoNa53}. As noted by Walther in his recent survey \cite{Wal}, a comprehensive theory for such equations remains in its infancy.\\

%
% \begin{assumption}\label{ass:Distri1Delay1}
%{\rm For a topological space $\Lambda$, let continuous  functions $F: \Lambda \times \mathbb{R} \times \mathbb{R}^n \to \mathbb{R}$ and
%$G: \Lambda \times  \mathbb{R}^{n} \to \mathbb{R}$ satisfy:
%\begin{itemize}
%\item[\rm (i)] For each fixed $(\lambda,t) \in \Lambda \times \mathbb{R}$, the mapping
%$F(\lambda,t,\cdot): \mathbb{R}^n \to \mathbb{R}$ is of class $C^2$, with all partial derivatives depending continuously on $(\lambda,t,x)$.
%\item[\rm (ii)] For each fixed $\lambda\in \Lambda$, the mapping
%$G(\lambda,\cdot): \mathbb{R}^n \to \mathbb{R}$ is even and of class $C^2$, with all partial derivatives depending continuously on $(\lambda,x)$.
%\end{itemize}}
%\end{assumption}

Finally, we study the following bifurcations
of distributed-delay differential systems:
\begin{equation}\label{e:Bifu-distributedDelay1}
\left\{\begin{array}{ll}
\!\!\dot{x}(t) = -\nabla_3 F\left(\lambda, t, \displaystyle\int_0^\tau
\nabla_2 G(\lambda, x(t-s)) \,\mathrm{d}s\right), & x(t) \in \mathbb{R}^n,
\\
\!\!x(t+\tau) = -x(t) \; \forall t,
\end{array}\right.
\end{equation}
%\begin{equation}\label{e:Bifu-distributedDelay2}
%\left\{\begin{array}{ll}
%\dot{x}(t) = -\nabla_x F\left(\lambda, t, \displaystyle\int_0^\tau
%\nabla_x G(x(t-s)) \,\mathrm{d}s\right), & x(t) \in \mathbb{R}^n,
%\\
%x(t+\tau) = x(t), \; \forall t,
%\end{array}\right.
%\end{equation}
%and
%\begin{equation}\label{e:Bifu-distributedDelay3}
%\left\{\begin{array}{ll}
%\dot{x}(t) = -\nabla_x F\left(\lambda, t, \displaystyle\int_0^\tau
%\nabla_x G(x(t-s)) \,\mathrm{d}s\right), & x(t) \in \mathbb{R}^n,
%\\
%x(t+\tau) =x(t) = -x(-t), & \forall t.
%\end{array}\right.
%\end{equation}
%Under this assumption and some additional conditions on $F$,
where $\Lambda$ is a topological space, and
$F: \Lambda \times \mathbb{R} \times \mathbb{R}^n \to \mathbb{R}$ and
$G: \Lambda \times \mathbb{R}^{n} \to \mathbb{R}$ are continuous functions satisfying the following:

\begin{assumption}\label{ass:Distri1Delay1}
\begin{itemize}
\item[\rm (i)] For each fixed $(\lambda,t) \in \Lambda \times \mathbb{R}$, the function
$F(\lambda,t,\cdot): \mathbb{R}^n \to \mathbb{R}$ is of class $C^2$,
and the Euclidean gradient and the Hessian of  $F(\lambda,t, x)$ with respect to
$x\in{\R}^{n}$, $\nabla_3F(\lambda,t, x)$ and $\nabla^2_3F(\lambda,t, z)$,
are continuous in  $(\lambda, t, x)\in\Lambda\times \R\times\mathbb{R}^{n}$.

\item[\rm (ii)] For each fixed $\lambda\in \Lambda$, the function
$G(\lambda,\cdot): \mathbb{R}^n \to \mathbb{R}$ is even and of class $C^2$,
and the Euclidean gradient and the Hessian of  $G(\lambda, x)$ with respect to
$x\in{\R}^{n}$, $\nabla_2G(\lambda, x)$ and $\nabla^2_2G(\lambda, x)$,
are continuous in  $(\lambda, x)\in\Lambda\times\mathbb{R}^{n}$.
\end{itemize}
\end{assumption}

%
%Let  $F \in C^1(\mathbb{R} \times \mathbb{R}^n, \mathbb{R})$ and $G \in C^1(\mathbb{R}^n, \mathbb{R})$ satisfy
%$F(t+\tau, -x) = F(t,x) = F(-t,x)$ and
%$G(x) = G(-x)$ for all $t \in \mathbb{R}$ and $x \in \mathbb{R}^n$.
%Let $z=(x^\top,y^\top)^\top$ and $H(t,z) = 2G(x) + F(t,y)$.
% Zhong, Wang and Liu (\cite{ZhongWL25}) showed:

%
%
%We will transform the problem of solving the system (1.6) together with $(S_1)$ and $(S_2)$ to that of solving a Hamiltonian system with some symmetric periodic conditions. Let $z = \begin{pmatrix} x \\ y \end{pmatrix} \in \mathbb{R}^{2n}$, the Hamiltonian function $H(t,z)$ is defined by
%
%\emph{The existence and multiplicity of symmetrical periodic solutions
%for asymptotically linear distributed delay differential systems},

%\noindent{\textbf{Methods and Organisation of the paper}}.

\noindent\textbf{Structure of the paper.}
Kaplan and Yorke~\cite{KaYo74} showed that periodic solutions of the delay differential %equation~\eqref{e:Delay1} 
Eq.~\eqref{e:Delay1} 
can be recast as periodic solutions with a symmetric structure for an associated Hamiltonian system; see~\cite{Fe06a,LiHe1998,LiHe1999a} and references therein.
Liu~\cite{Liu12} further reformulated such problems as equivalent Hamiltonian boundary-value problems.
%Following Liu~\cite{Liu12}, we work throughout with the
%Hamiltonian boundary-value reformulation of the
%Kaplan--Yorke reduction; this is what we mean by ``the method''
%used below.

In Section~\ref{sec:delay0} we use this method to derive equivalent Hamiltonian formulations for the bifurcation problems~\eqref{e:Delay1}--\eqref{e:Delay3}.
Section~\ref{sec:delay1} then applies the theorems of~\cite{Lu11} to analyse bifurcations in~\eqref{e:Delay1} under Assumption~\ref{ass:BasiAss1Delay2}.
Theorems~\ref{th:bif-per1Delay1},~\ref{th:bif-per2Delay},~\ref{th:bif-per2Delay+} establish a framework for characterising bifurcation phenomena along the branch $\{(\lambda,x^\lambda)\mid\lambda\in\Lambda\}$ via the linearised problem
\begin{equation}\label{e:LinearDelay1.1Lambda}
\left\{\begin{array}{ll}
\!\!\dot x(t)=\displaystyle\sum_{k=1}^{m-1}\nabla_3^2 V\bigl(\lambda,t,x^\lambda(t-\tau)\bigr)x(t-\tau),\\[4pt]
\!\!x(t+m\tau)=-x(t)\quad\forall t,
\end{array}\right.
\end{equation}
where $\nabla_3^2 V$ is the Hessian of $V$ w.r.t.\ $x\in\mathbb R^n$.
If~\eqref{e:LinearDelay1.1Lambda} admits only the trivial solution, Theorem~\ref{th:bif-per1Delay1}(I) implies that $(\mu,x^\mu)$ is not a bifurcation point on that branch.
The remaining parts of the three theorems describe bifurcation behaviour in terms of Maslov-type indices of fundamental-matrix solutions of the associated linear Hamiltonian systems derived from~\eqref{e:LinearDelay1.1Lambda}; these indices are always theoretically computable.
When $V(\lambda,t,z)$ is affine in $\lambda$, the three theorems simplify to a concise form without explicit indices (Theorem~\ref{th:bif-per2Delay++}).

Parallel results for~\eqref{e:Delay2} under Assumption~\ref{ass:BasiAss1Delay3} appear in Section~\ref{sec:delay2}; the main results are Theorems~\ref{th:bif-per1Delay2},~\ref{th:bif-per2Delay2},~\ref{th:bif-per2Delay+2},~\ref{th:bif-per2Delay++G}.

Section~\ref{sec:delay3} presents only three general bifurcation results for system~\eqref{e:Delay3}: Theorems~\ref{th:bif-per1Delay3},~\ref{th:bif-per2Delay3},~\ref{th:bif-per2Delay+3}.
These do not correspond to Theorem~\ref{th:bif-per2Delay++} or~\ref{th:bif-per2Delay++G}, because the matrix $\mathcal A_m$ in~\eqref{e:Delay21} is indefinite, making $\nabla_3^2\mathcal H(\lambda,t,\mathbf0)$ in~\eqref{e:Delay32} indefinite even when all $U(\lambda,t,\cdot)$ are even and $\nabla_3^2U(\lambda,t,\mathbf0)$ has a fixed sign.

 Section~\ref{sec:delay4} discusses bifurcations of solutions for problem~\eqref{e:crm1} and~\eqref{e:crm1Ad}.

 Section~\ref{sec:delay5} investigates bifurcations of solutions for problem~\eqref{e:Bifu-distributedDelay1}.\\

%\noindent\textbf{Further researches.}
%When $V$ in system (\ref{e:Delay1}) is independent of time $t$, we can generalize our studies in \cite{Lu12, Lu13} to the system using the direct variational method
%of Guo and Yu \cite{GuoMYu05, GuoMYu11}.

\noindent{\textbf{Notation and conventions}}. Let  $[a]$  denote the largest integer no more than $a\in\R$.
All vectors in $\R^m$ will be understood as column vectors.
Let $\mathbb{K}=\mathbb{R}$ or $\mathbb{C}$, and let $\mathbb{K}^{m\times m}$ denote the set of
all $m\times m$ matrices with entries in the field $\mathbb{K}$.
For $M\in\K^{m\times m}$ let $\sigma(M)$ be the set of all eigenvalues of $M$,
let $M^\top$ be the transpose of $M$, ${\rm Ker}(M)=\{x\in\K^m\,|\,Mx=0\}$ and $\exp M=e^M=\sum^\infty_{l=0}\frac{1}{l!}M^l$.
We denote by $(\cdot,\cdot)_{\mathbb{R}^m}$ the standard Euclidean inner product on
$\mathbb{R}^m$, and by $|\cdot|$ the corresponding norm.
Let $\mathcal{L}_s(\mathbb{R}^{m})$ be the set of all real symmetric
matrixes of order $m$. The standard complex structure on $\mathbb{R}^{2n}$ is given by %(\ref{e:standcompl}).
$J_n=\left(\begin{array}{cc}
             0 & -I_n \\
             I_n & 0 \\
           \end{array}
         \right)$.
%
%({\bf Note}: In this paper,  all vectors in the Euclidean spaces are understood to be column vectors,

%When $\lambda<0$, the equation (\ref{e:Delay6Auto13C-}) is termed a \textsf{differential equation with advanced arguments} (cf. \cite{ElN73,Hale}).
%or
%For $\lambda<0$, (\ref{e:Delay6Auto13C-}) is referred to as an advanced differential equation? (cf. \cite{ElN73,Hale}).
%
%For a $C^2$ even function $W:\mathbb{R}\to\mathbb{R}$,
%if $\tau>0$ satisfies $\tau W''(0)\in(4\mathbb{Z}+3)\pi$,
% by Corollary~\ref{cor:bif-per4Delay}  we conclude that the system
%\begin{eqnarray*}
%  \dot{x}( t ) = \lambda\nabla W (x ( t - \lambda) )\quad\hbox{and}\quad
%  x( t -2\lambda) = - x( t )\;\forall t\in\mathbb{R}
%    \end{eqnarray*}
%exhibits alternative bifurcations of Fadell--Rabinowitz type with respect to the delay $\lambda$ at the trivial solution $0\in\mathbb{R}$.
%For more general and precise statements, see Theorem~\ref{th:bif-per3DelayC}, which can be derived from Theorem~\ref{th:bif-per3Delay}.
%
%Appendix~\ref{app:Matrix} provides an explicit algorithmic construction of congruence matrices between:
%(i) the matrix $A_{m,n}^{-1}$ in (\ref{e:Delay7}) and $J_{mn/2}^{-1}$,
%(ii) the matrix $J_{ n , m }$ in (\ref{e:Delay17}) and  $J_{mn}$.

\section{Reduction to Hamiltonian bifurcation problems}\label{sec:delay0}

Recall that all vectors in the Euclidean spaces will be regarded as column vectors.

%?All vectors in $\R^m$ are understood to be column vectors.\\
%We assume that all vectors in $\R^m$ are column vectors.\\
%It is understood that all vectors in $\R^m$ are column vectors.

Every skew-symmetric non-degenerate $2 N \times 2 N$ matrix $\mathcal{J}=(a_{i j})$
determines a constant symplectic form on $\mathbb{R}^{2 N}$, $\omega=\frac{1}{2} \sum_{i, j} a^{i j} d x_{i} \wedge d x_{j}$ with $(a^{i j})=\mathcal{J}^{-1}$, which can be written as
\begin{equation}\label{e:nonstandSympl}
\omega(v, w)=v^{\top} \cdot \mathcal{J}^{-1}w=(v, \mathcal{J}^{-1}w)_{\mathbb{R}^{2N}}=(\mathcal{J}^{-1}w, v)_{\mathbb{R}^{2N}}
\end{equation}
after $\sum^{2N}_{i=1}v_i\frac{\partial}{\partial x_i}$ and $\sum^{2N}_{i=1}w_i\frac{\partial}{\partial x_i}$
are identified with $v=(v_1,\dots,v_{2N})^\top\in\mathbb{R}^{2N}$ and $u=(u_1,\dots,u_{2N})^\top\in\mathbb{R}^{2N}$, respectively.
Matrixes in $\operatorname{Sp}_{\mathcal{J}}(2 N):=\{\mathcal{M}\in\mathbb{R}^{2N\times 2N}\,|\,
\mathcal{M}^{\top} \cdot \mathcal{J}^{-1} \cdot \mathcal{M}=\mathcal{J}^{-1}\}$
are called \textsf{$\mathcal{J}$-symplectic}.

\begin{assumption}\label{ass:EquivDelay}
{\rm Given $\mathcal{M}\in \operatorname{Sp}_{\mathcal{J}}(2 N)$ with $\mathcal{M}^k=I_{2N}$,  a real $\tau>0$  and
a topological space $\Lambda$, let
$K:\Lambda\times\R\times{\R}^{2N}\to\R$ be a continuous function such that
each $K(\lambda,t,\cdot):{\R}^{2N}\to\R$, $(\lambda,t)\in\Lambda\times \R$, is $C^2$,
 and all possible partial derivatives of $K$ depend continuously on
 $(\lambda, t, z)\in\Lambda\times \R\times\mathbb{R}^{2N}$, and that
 for all $(\lambda,t,z)\in\Lambda\times\R\times\mathbb{R}^{2N}$,
\begin{equation}\label{e:M-invariantDelay1}
K(\lambda, t+\tau, \mathcal{M}z)=K(\lambda, t, z).
\end{equation}
}
\end{assumption}

The second-to-last assumption is equivalent to the statement that the Euclidean gradient
$\nabla_3 K(\lambda, t, z)$ and the Euclidean Hessian matrix $\nabla^2_3 K(\lambda, t, z)$
 of $K(\lambda, t, z)$ with respect to the variable $z \in \mathbb{R}^{2N}$ are continuous in
 $(\lambda, t, z) \in \Lambda \times \mathbb{R} \times \mathbb{R}^{2N}$.
Here, $\nabla_3$ and $\nabla^2_3$ denote the partial gradient and Hessian with respect to the third argument $z$, as defined earlier.

  Since all constant symplectic forms on $\mathbb{R}^{2N}$ are linear symplectic isomorphic to
  the canonical symplectic structure $\omega_0$ on $\mathbb{R}^{2N}$,
  there exists a nonsingular $2 N \times 2 N$ matrix $A$ such that
$\omega_0(v, w)=\omega(Av, Aw)$ for all $v,w\in\mathbb{R}^{2N}$,
which is equivalent to $A^\top\mathcal{J}^{-1}A=J_N^{-1}$,
where
\begin{equation}\label{e:standcompl}
J_N=\left(\begin{array}{cc}
             0 & -I_N \\
             I_N & 0 \\
           \end{array}
         \right)
\end{equation}
with the identity matrix $I_N$ of order $N$.
Let $M=A^{-1}\mathcal{M}A$. Then it is easily checked that $M^\top J_NM=J_N$
and $M^k=I_{2N}$. Define
\begin{equation}\label{e:M-invariantDelay2}
\check{K}^A:\Lambda\times\R\times{\R}^{2N}\to\R,\;(\lambda,t, z)\mapsto K(\lambda,t, Az).
\end{equation}
Then  (\ref{e:M-invariantDelay1}) implies that for all $(\lambda, t, z)$,
$\check{K}^A(\lambda, t+\tau, Mz)=\check{K}^A(\lambda, t, z)$ and
\begin{equation}\label{e:M-invariantDelay3}
\left\{\begin{array}{l}
\nabla_3 {K}(\lambda, t, Az)=(A^{-1})^\top\nabla_3\check{K}^A(\lambda, t, z),\\
\nabla^2_3 {K}(\lambda, t, Az)=(A^{-1})^\top\nabla^2_2 \check{K}^A(\lambda, t, z)A^{-1}.
\end{array}\right.
\end{equation}
It is straightforward to verify the following proposition.

\begin{proposition}\label{prop:Equiv}
Under the above assumptions, $z:\mathbb{R}\to\mathbb{R}^{2N}$ satisfies
\begin{equation}\label{e:M-invariantDelay4}
\left\{\begin{array}{l}
\dot{z}(t)=\mathcal{J}\nabla_3K(\lambda, t, z(t)),\\
 z(t+\tau)=\mathcal{M}z(t)\;\;\forall t\in \R
\end{array}\right.
\end{equation}
if and only if
$u(t):=A^{-1}z(t)$ solves the following problem
\begin{equation}\label{e:M-invariantDelay5}
\left\{\begin{array}{l}
\dot{u}(t)={J}_N \nabla_3 \check{K}^A(\lambda, t, u(t)), \\
 u(t+\tau)={M}u(t)\;\;\forall t\in \R.
\end{array}\right.
\end{equation}
Moreover, in this situation,  (since $\mathcal{J}=AJ_N(A^{-1})^\top$) it holds that
\begin{align}\label{e:M-invariantDelay6-}
&\int^{\tau}_0\left[\frac{1}{2}(J_N\dot{u}(t), u(t))_{\mathbb{R}^{2N}}+ \check{K}(\lambda, t, u(t))\right]dt\nonumber\\
&= \int^{\tau}_0\left[\frac{1}{2}(\dot{z}(t), \mathcal{J}^{-1}z(t))_{\mathbb{R}^{2N}}+ K(\lambda, t,
z(t))\right]dt,
\end{align}
and if $\Gamma_{z}(t)$ is the fundamental matrix solution of
\begin{equation}\label{e:M-invariantDelay6}
\dot{y}(t)=\mathcal{J} \nabla^2_3{K}(\lambda, t, z(t)) y(t)
\end{equation}
 with $\Gamma_{z}(0)=I_{2 N}$, then $A^{-1}\Gamma_z(t)A$
 is the fundamental matrix solution of
\begin{equation}\label{e:M-invariantDelay7}
\dot{y}(t)={J}_N \nabla^2_3\check{K}^A(\lambda, t, u(t)) y(t)
\end{equation}
 with $A^{-1}\Gamma_z(0)A=I_{2N}$.
\end{proposition}

%The equation (\ref{e:M-invariantDelay6-}) is derived from $\mathcal{J}=AJ_N(A^{-1})^T$.
%Others are clear.

%Motivated by \cite{KaYo74},
%Liu \cite{Liu12} developed the ideas of \cite{KaYo74}
% and recast the study of periodic solutions in delay differential equations (\ref{e:Delay1})-(\ref{e:Delay3})
% into a generalized  Hamiltonian boundary problem as in (\ref{e:M-invariantDelay4}),
%that is,  he
According to Liu \cite{Liu12} we have:

\begin{proposition}\label{prop:Liu12}
Under Assumption~\ref{ass:BasiAss1Delay1}, the following conclusions hold:
  \begin{enumerate}
\item[\rm (i)] For $m\in 2\mathbb{N}$,  let $H:\Lambda\times \mathbb{R}\times({\mathbb{R}}^{n})^m\to\mathbb{R}$ be defined by
\begin{equation}\label{e:Delay4}
H\left(\lambda, t , x _ { 1 } , \ldots , x _ { m } \right) = V \left(\lambda, t , x _ { 1 } \right)
+ \cdots + V \left(\lambda, t , x _ { m } \right).
\end{equation}
Denote by $\nabla_3 H(\lambda, t, v)$  the gradient of $H\left(\lambda, t , v \right)$ with respect to the third variable $v\in({\mathbb{R}}^{n})^m$. Define the
$m n \times m n$  matrices  $A _ { m,n }$ and $T_{m,n}$  by
\begin{equation}\label{e:Delay7}
A _ { m,n} = \left( \begin{array} { c c c c } 0 & I _ { n } & \cdots & I _ { n } \\ - I _ { n } & 0 & \cdots & I _ { n } \\ \vdots & \vdots & \cdots & \vdots \\ - I _ { n } & - I _ { n } & \cdots & 0 \end{array} \right)
\quad\text{and}\quad
T_{m,n}=\left( \begin{array} { c c c } 0 & I _ { n ( m - 1 ) } \\ - I _ { n } & 0 \end{array} \right),
\end{equation}
 respectively. These matrices satisfy
\begin{align}\label{e:Delay8}
&T^m_{m,n}=-I_{mn},\quad T^{2m}_{m,n}=I_{mn}, \quad T_{m,n}A_{m,n}(T_{m,n})^\top=A_{m,n},\nonumber\\
&H(\lambda, t, T_{m,n}z)=H(\lambda, t+\tau, T_{m,n}z)=H(\lambda,t, z)\quad\forall (\lambda,t,z).
\end{align}
If $x$ is a solution of (\ref{e:Delay1}), then it is necessarily $2 m \tau$-periodic, and
\begin{equation}\label{e:Delay5}
\mathbb{R}\ni t\mapsto v(t):=(x_1(t)^\top,\dots, x_m(t)^\top)^\top\in ({\mathbb{R}}^{n})^m
\end{equation}
where $x _ { i } ( t ) = x ( t-(i-1)\tau)$ for $i=1,\dots,m$, satisfies
\begin{equation}\label{e:Delay6}
\dot{v}(t)=A_{m,n}\nabla_3 H(\lambda, t, v(t))\quad\hbox{and}\quad v(t+\tau)=T_{m,n}^{-1}v(t)\;\;\forall t\in \mathbb{R}.
\end{equation}
Conversely, if $v(t)=(x_1(t)^\top,\dots, x_m(t)^\top)^\top$, where $x _ { i } ( t )\in\mathbb{R}^n$ for $i=1,\dots,m$, satisfies (\ref{e:Delay6}),
then $x(t):=x_1(t)$ satisfies (\ref{e:Delay1}). % with $x ( t - m \tau ) = - x ( t ) $ (and thus  $2 m \tau$-periodic).

\item[\rm (ii)] For $m\in 2\mathbb{N}+1$,  let $\tilde{H}:\Lambda\times \mathbb{R}\times({\mathbb{R}}^{n})^{m-1}\to\mathbb{R}$ be defined by
\begin{align}\label{e:Delay9}
\tilde{H}\left(\lambda, t , x _ { 1 } , \ldots , x _ { m-1 } \right) &= V \left(\lambda, t , x _ { 1 } \right) + \cdots
+ V \left(\lambda, t , x _ { m -1} \right)\nonumber\\
&\quad+V \left(\lambda, t , -x_1+x_2-\cdots+ x _ { m-1} \right).
\end{align}
Denote by $\nabla_3 \tilde{H}(\lambda, t, v)$  the gradient of $\tilde{H}\left(\lambda, t , v \right)$
with respect to the third variable $v\in({\mathbb{R}}^{n})^{m-1}$.
Define the $(m-1) n \times(m-1)n$  matrix  $\tilde{B}_{m-1,n}$ by
\begin{equation}\label{e:Delay12}
\tilde { B } _ {2,n} = \left( \begin{array} { c c c c c } 0 & I _ { n }\\
- I _ { n } & I _ { n }\end{array} \right)\quad\text{and}\quad\tilde { B } _ { m - 1,n} = \left( \begin{array} { c c c c c } 0 & I _ { n } & 0 & \cdots & 0 \\ 0 & 0 & I _ { n } & \cdots & 0 \\ \vdots & \vdots & \vdots & \cdots & \vdots \\ 0 & 0 & 0 & \cdots & I _ { n } \\ - I _ { n } & I _ { n } & - I _ { n } & \cdots & I _ { n } \end{array} \right)\quad\text{for $m>3$}.
\end{equation}
These matrices satisfy
\begin{align}\label{e:Delay13}
&\left( \tilde { B }_{ m - 1,n} \right)^{ 2 m } = I_{ n ( m - 1 ) },\quad
 \tilde { B }_{ m - 1,n} A_{ m - 1,n} \left( \tilde { B }_{ m - 1,n} \right)^{ \top } = A_{ m - 1,n},
 \nonumber\\
&\tilde{H}(\lambda, t, \tilde{B}_{m-1,n}z) = \tilde{H}(\lambda, t+\tau, \tilde{B}_{m-1,n}z) = \tilde{H}(\lambda, t, z) \quad \forall (\lambda, t, z).
\end{align}
Suppose that $x$ is a solution of problem (\ref{e:Delay1}).
 Then $x$ is necessarily $2m\tau$-periodic.
Furthermore, the function $v:\mathbb{R}\to (\mathbb{\R}^{n})^{m-1}$ defined by
\begin{equation}\label{e:Delay10}
v(t)=(x_1(t)^\top,\dots, x_{m-1}(t)^\top)^\top,
\end{equation}
where $x_{i}(t)= x ( t-(i-1)\tau)$ for $i=1,\dots,m-1$,  satisfies
\begin{equation}\label{e:Delay11}
\dot{v}(t)=A_{m-1,n}\nabla_3\tilde{H}(\lambda, t, v(t))\quad\text{and}\quad v(t+\tau)=\tilde{B}_{m-1,n}^{-1}v(t)\;\;\forall t\in \mathbb{R}.
\end{equation}
Conversely, if $v(t)=(x_1(t)^\top,\dots, x_{m-1}(t)^\top)^\top$, where $x_{ i } ( t )\in\mathbb{R}^n$ for $i=1,\dots,m-1$, is a solution of (\ref{e:Delay11}),
then $x(t):=x_1(t)$ satisfies (\ref{e:Delay1}).

\item[\rm (iii)] For an integer $m\ge 2$, let $\hat{H}:\Lambda\times \mathbb{R}\times({\mathbb{R}}^{2n})^m\to\mathbb{R}$ be defined by
\begin{equation}\label{e:Delay14}
\hat{H} \left(\lambda, t , x_{ 1 } , \ldots , x_{ m } \right) = G\left(\lambda, t , x_{ 1 } \right) + \cdots + G\left(\lambda, t , x_{ m } \right).
\end{equation}
Denote by $\nabla_3 \hat{H}(\lambda, t, v)$  the gradient of $\hat{H}\left(\lambda, t , v \right)$ with respect to the third variable $v\in({\mathbb{R}}^{2n})^m$.
%Define $2mn \times 2mn$ matrices  $J_ {n, m }$ and $P_{n,m}$  by
For $k\in\mathbb{N}$, define matrices  $J_{n,m }$ of order $2mn$
 and $P_{k,m}$ of order $km$  by
\begin{equation}\label{e:Delay17}
J _ { n , m } = \left( \begin{array} { c c c c c } 0 & J _ { n } & \cdots & J _ { n } & J _ { n } \\
J _ { n } & 0 & \cdots & J _ { n } & J _ { n } \\
\vdots & \vdots & \cdots & \vdots & \vdots \\ J_{ n } & J _ { n } & \cdots & J _ { n } & 0 \end{array} \right)
\quad\text{and}\quad
P_{k,m } =  \left( \begin{array} { c c c } 0 & I _ { k( m - 1 ) } \\ I_{k} & 0 \end{array} \right),
\end{equation}
 respectively. These matrices satisfy the following properties:
\begin{align}\label{e:Delay18}
&P_{k, m }^{ - 1 } = P_{k, m }^{ \top },
\quad P_{k, m }^{ m } = I_{km },\quad
P_{2n, m } J_{ n , m } = J_{n, m}P_{2n , m }, \nonumber\\
&\hat{H}(\lambda,t, z)=\hat{H}(\lambda, t, P_{2n,m}z)=\hat{H}(\lambda, t+\tau, P_{2n,m}z)\quad\forall (\lambda,t,z).
\end{align}
If  $x$ is a  solution of (\ref{e:Delay2}), and
 $x_{ i } ( t ) = x ( t-(i-1)\tau)$ for $i=1,\dots,m$,
then the function
\begin{equation}\label{e:Delay15}
\mathbb{R}\ni t\mapsto v(t):=(x_1(t)^\top,\dots, x_m(t)^\top)^\top\in ({\mathbb{R}}^{2n})^m
\end{equation}
satisfies
\begin{equation}\label{e:Delay16}
\dot{v}(t)=J_{n,m}\nabla_3 \hat{H}(\lambda, t, v(t))\quad\text{and}\quad v(t+\tau)=P_{2n,m}^{-1}v(t)\;\;\forall t\in \mathbb{R}.
\end{equation}
Conversely, if $v(t)=(x_1(t)^\top,\dots, x_m(t)^\top)^\top$, where $x_{ i } ( t )\in\mathbb{R}^{2n}$ for $i=1,\dots,m$, satisfies (\ref{e:Delay16}),
then $x(t):=x_1(t)$ is a  solution of (\ref{e:Delay2}).

\item[\rm (iv)] For an integer $m\ge 2$, let ${\bf H}:\Lambda\times \mathbb{R}\times({\mathbb{R}}^{n})^m\to\mathbb{R}$ and  $\mathcal{H}:\Lambda\times \mathbb{R}\times({\mathbb{R}}^{n})^m
\times({\mathbb{R}}^{n})^m\to\mathbb{R}$ be defined by
\begin{eqnarray}\label{e:Delay19}
&&{\bf H}\left(\lambda, t , x_{ 1 } , \ldots , x_{ m } \right) = U\left(\lambda, t , x_{ 1 } \right) + \cdots + U\left(\lambda, t, x_{ m } \right),\\
&&\mathcal{H}(\lambda, t, w)=\mathcal{H}(\lambda, t, (z^\top, y^\top))=-\frac{1}{2}(\mathcal{A}_{m,n}y,y)-{\bf H}(\lambda,t,z),\label{e:Delay20}
\end{eqnarray}
respectively, where
\begin{equation}\label{e:Delay21}
\mathcal { A }_{ m,n} = \left( \begin{array} { c c c c } 0 & I _ { n } & \cdots & I _ { n } \\ I _ { n } & 0 & \cdots & I _ { n } \\ \vdots & \vdots & \cdots & \vdots \\ I _ { n } & I _ { n } & \cdots & 0 \end{array} \right)\in\mathbb{R}^{mn\times mn}.
\end{equation}
Denote by  $\nabla_3 \mathcal{H}(\lambda, t, w)$  the gradient of $\mathcal{H}\left(\lambda, t , w \right)$
 with respect to the third variable $w\in({\mathbb{R}}^{n})^m\times ({\mathbb{R}}^{n})^m$. Let  $P_{ n , m }$ be as in (\ref{e:Delay17}), and
\begin{equation}\label{e:Delay23}
\mathcal { P }_{ n , m } : = \left( \begin{array} { c c } P_{ n , m } & 0 \\ 0 & P_{ n , m } \end{array} \right) \in \operatorname { S p } ( 2 m n ).
%\quad\text{with}\quad
%P_{n , m }: =  \left( \begin{array} { c c c } 0 & I_ { 2n ( m - 1 ) } \\
%I_{ 2n } & 0 \end{array} \right).
\end{equation}
Then the matrices $P_{n , m }$ and $\mathcal { P }_{ n , m }$ satisfy the following properties:
\begin{align*}
&\mathcal { A }_{ m,n}P_{n , m }=P_{n , m }\mathcal { A }_{ m,n}, \quad
\mathcal { P }_{ n , m } J_{ m n } \mathcal { P }_{ n , m } ^ { \top } = J_{ m n },\quad
\mathcal { P }_{ n , m } ^ { m } = I_{ 2 mn},\\
&\mathcal{H}(\lambda,t+\tau, \mathcal { P }_{ n , m }^{-1}w)=\mathcal{H}(\lambda,t, w)
\quad\text{for all $(\lambda, t, w)$.}
\end{align*}
Suppose that  $x$ is a  solution of (\ref{e:Delay3}). Define
\begin{align*}
 x_{ i }( t )&= x ( t-(i-1)\tau),\quad i=1,\dots,m,\\
 v(t)&=(x_1(t)^\top,\dots, x_m(t)^\top)^\top,\quad u(t)=(\mathcal { A }_{m,n})^{-1}\dot{v}(t).
  % w(t)&=(v(t)^\top, u(t)^\top)^\top.
  \end{align*}
Then $w(t):=(v(t)^\top, u(t)^\top)^\top$ satisfies
\begin{equation}\label{e:Delay22}
\dot{w}(t)=J_{mn}\nabla_3\mathcal{H}(\lambda, t, w(t))\quad\hbox{and}\quad w(t+\tau)=\mathcal{P}_{n,m}^{-1}w(t)\;\;\forall t\in \mathbb{R}.
\end{equation}
Conversely, if $w(t)=(v(t)^\top, u(t)^\top)^\top$ with $v(t)=(x_1(t)^\top,\dots, x_m(t)^\top)^\top$,
 where $x_{ i } ( t )\in\mathbb{R}^{n}$
for $i=1,\cdots,m$, satisfies (\ref{e:Delay22}), then $x(t):=x_1(t)$ is a solution of
 (\ref{e:Delay3}).
  \end{enumerate}
\end{proposition}

\begin{remark}\label{rm:Liu12}
{\rm Under Assumption~\ref{ass:BasiAss1Delay1}, we have the following:
  \begin{enumerate}
\item[\rm (i)] For $m\in 2\mathbb{N}$, let $\bar{x}:\mathbb{R}\to\mathbb{R}^n$ satisfy (\ref{e:Delay1}), and
let
\begin{align*}
\bar{x}_ { i } ( t ) &= \bar{x}( t-(i-1)\tau),\quad i=1,\dots,m,\\
\bar{v}(t)&=(\bar{x}_1(t)^\top,\dots, \bar{x}_{m}(t)^\top)^\top\in ({\mathbb{R}}^{n})^{m}.
\end{align*}
Consider the linearized problem of (\ref{e:Delay1}) along the solution $\bar{x}$:
 \begin{equation}\label{e:LinearDelay1.1}
 \left\{\begin{array}{ll}
& \dot{x}( t ) =\sum^{m-1}_{i=1} \nabla_3^2 V (\lambda, t , \bar{x}( t - i\tau ) )x(t-i\tau),\\
& x( t + m \tau ) = - x( t )\;\forall t.
   \end{array}\right.
  \end{equation}
Here $\nabla^2_3V(\lambda,t, x)$ denotes the Hessian of $V(\lambda,t, x)$ with respect to
the third variable $x\in\mathbb{R}^n$.
Suppose that $x=x(t)$ is a solution of (\ref{e:LinearDelay1.1}).
Then $x(t)$ is necessarily $2m\tau$-periodic. Furthermore,
with $x_{ i } ( t ) = x ( t-(i-1)\tau)$ for $i=1,\dots,m$, the function
\begin{equation*}
\mathbb{R}\ni t\mapsto v(t):=(x_1(t)^\top,\dots, x_m(t)^\top)^\top\in ({\mathbb{R}}^{n})^m
\end{equation*}
 satisfies the linearization of problem (\ref{e:Delay6}) along $\bar{v}$:
\begin{equation}\label{e:LinearDelay1.2}
\dot{v}(t)=A_{m,n}\nabla^2_3 H(\lambda, t, \bar{v}(t))v(t)\quad\text{and}\quad v(t+\tau)=T_{m,n}^{-1}v(t)\;\;\forall t\in \mathbb{R},
\end{equation}
where $H$, $A _ { m,n }$ and $T_{m,n}$ are as in Proposition~\ref{rm:Liu12}(i),
and $\nabla^2_3 H(\lambda, t, v)$ is the Hessian of
$H\left(\lambda, t , v \right)$ with respect to the third variable $v\in({\mathbb{R}}^{n})^m$.
 Conversely, if
 $$
 v(t)=(x_1(t)^\top,\dots, x_m(t)^\top)^\top,\quad \text{with $x_{ i } ( t )\in\mathbb{R}^n$ for $i=1,\dots,m$},
 $$
 satisfies (\ref{e:LinearDelay1.2}),
then $x(t):=x_1(t)$ is a solution of (\ref{e:LinearDelay1.1}).

\item[\rm (ii)] For $m\in 2\mathbb{N}+1$,  let $\bar{x}:\mathbb{R}\to\mathbb{R}^n$ satisfy (\ref{e:Delay1}),
and let
\begin{align*}
\bar{x}_ { i } ( t ) &= \bar{x}( t-(i-1)\tau),\quad i=1,\dots,m-1,\\
\bar{v}(t)&=(\bar{x}_1(t)^\top,\dots, \bar{x}_{m-1}(t)^\top)^\top\in ({\mathbb{R}}^{n})^{m-1}.
\end{align*}
If $x$ is a solution of (\ref{e:LinearDelay1.1}), then it must be $2m\tau$-periodic, and
the function
\begin{equation*}
\mathbb{R}\ni t\mapsto v(t):=(x_1(t)^\top,\cdots, x_{m-1}(t)^\top)^\top\in ({\mathbb{R}}^{n})^{m-1},
\end{equation*}
where $x_{ i } ( t ) = x( t-(i-1)\tau)$ for $i=1,\dots,m-1$,
satisfies the linearization  of  problem (\ref{e:Delay11}) along $\bar{v}$:
\begin{equation}\label{e:LinearDelay1.3}
\dot{v}(t)=A_{m-1,n}\nabla^2_3\tilde{H}(\lambda, t, \bar{v}(t))v(t)\quad\text{and}\quad v(t+\tau)=\tilde{B}_{m-1,n}^{-1}v(t)\;\;\forall t\in \mathbb{R},
\end{equation}
where $\tilde{H}$,  $A _ {m-1,n }$ and $\tilde{B}_{m-1,n}$ are as in Proposition~\ref{rm:Liu12}(ii),
and $\nabla^2_3 \tilde{H}(\lambda, t, v)$ is the Hessian of
$\tilde{H}\left(\lambda, t , v \right)$
with respect to the third variable $v\in({\mathbb{R}}^{n})^{m-1}$.
 Conversely, if $v(t)=(x_1(t)^\top,\dots, x_{m-1}(t)^\top)^\top$ with $x_{ i } ( t )\in\mathbb{R}^n$
 ($i=1,\dots,m-1$) is a solution of (\ref{e:LinearDelay1.3}),
then $x(t):=x_1(t)$ satisfies (\ref{e:LinearDelay1.1}).

\item[\rm (iii)] For an integer $m\ge 2$ and a  solution $\bar{x}$ of (\ref{e:Delay2}),
let  $x$ be a  solution of the  linearization  of problem (\ref{e:Delay2}) along $\bar{x}$:
\begin{equation}\label{e:LinearDelay1.4}
  \left\{\begin{array}{ll}
 &\dot{x}( t ) = J_n\sum^{m-1}_{i=1}\nabla^2_3 G (\lambda, t , \bar{x}(t - i\tau ) )x ( t - i\tau ),\\
 &  x( t + m \tau ) =  x( t )\;\forall t,
   \end{array}\right.
  \end{equation}
 where $\nabla^2_3 G(\lambda, t, x)$ is the Hessian of $G\left(\lambda, t , x \right)$ with respect to
  $x\in {\mathbb{R}}^{2n}$.
 Denote by
 \begin{align*}
 x_{ i }( t )&= x ( t-(i-1)\tau),\quad i=1,\dots,m,\\
 \bar{x}_{ i } ( t ) &= \bar{x}( t-(i-1)\tau), \quad i=1,\dots,m,\\
 v(t)&=(x_1(t)^\top,\dots, x_m(t)^\top)^\top\in ({\mathbb{R}}^{2n})^m,\\
 \bar{v}(t)&=(\bar{x}_1(t)^\top,\dots, \bar{x}_m(t)^\top)^\top\in ({\mathbb{R}}^{2n})^m.
  \end{align*}
  Then $v(t)$  satisfies the linearization  of problem (\ref{e:Delay16}) along $\bar{v}$:
\begin{equation}\label{e:LinearDelay1.5}
\dot{v}(t)=J_{n,m}\nabla^2_3 \hat{H}(\lambda, t, \bar{v}(t))v(t)\quad\hbox{and}\quad v(t+\tau)=P_{2n,m}^{-1}v(t)\;\;\forall t\in \mathbb{R},
\end{equation}
where $\hat{H}$, $J_ {n, m }$ and $P_{2n,m}$ are as in Proposition~\ref{rm:Liu12}(iii),
and $\nabla^2_3 \hat{H}(\lambda, t, v)$ is the Hessian of $\hat{H}\left(\lambda, t , v \right)$ with respect to the third variable $v\in({\R}^{2n})^m$.
Conversely, if $v(t)=(x_1(t)^\top,\dots, x_m(t)^\top)^\top$ with $x_{ i } ( t )\in\mathbb{R}^{2n}$
($i=1,\dots,m$)  is a solution of (\ref{e:LinearDelay1.5}),
then $x(t):=x_1(t)$ satisfies (\ref{e:LinearDelay1.4}).

\item[\rm (iv)] For an integer $m\geq 2$ and  a  solution $\bar{x}$ of (\ref{e:Delay3}),
suppose that $x$ satisfies the  linearization  of problem (\ref{e:Delay3}) along $\bar{x}$:
\begin{equation}\label{e:LinearDelay1.6}
 \left\{\begin{array}{ll}
 &\ddot{x}( t ) = -\sum^{m-1}_{i=1}\nabla^2_3 U (\lambda, t , \bar{x}( t - i\tau ) ) x ( t - i\tau ),  \\
 &  x( t + m \tau ) =  x( t )\;\forall t,
   \end{array}\right.
 \end{equation}
 where $\nabla^2_3U(\lambda,t, x)$ denotes the Hessian of $U(\lambda,t, x)$ with respect to  $x\in\mathbb{R}^n$.  Let
 \begin{align*}
 x_{ i }( t )&= x ( t-(i-1)\tau),\quad i=1,\dots,m,\\
 \bar{x}_{ i } ( t ) &= \bar{x}( t-(i-1)\tau), \quad i=1,\dots,m,\\
 v(t)&=(x_1(t)^\top,\dots, x_m(t)^\top)^\top,\quad u(t)=(\mathcal { A }_{m,n})^{-1}\dot{v}(t),\\
 \bar{v}(t)&=(\bar{x}_1(t)^\top,\dots, \bar{x}_m(t)^\top)^\top,\quad
 \bar{u}(t)=(\mathcal { A }_{m,n})^{-1}\dot{\bar{v}}(t),\\
  w(t)&=(v(t)^\top, u(t)^\top)^\top,\quad \bar{w}(t)=(\bar{v}(t)^\top, \bar{u}(t)^\top)^\top.
  \end{align*}
   Then $w(t)$  satisfies the  linearization  of problem (\ref{e:Delay22}) along $\bar{w}$:
\begin{equation}\label{e:LinearDelay1.7}
\dot{w}(t)=J_{mn}\nabla^2_3\mathcal{H}(\lambda, t, \bar{w}(t))w(t)\quad\hbox{and}\quad w(t+\tau)=\mathcal{P}_{n,m}^{-1}w(t)\;\;\forall t\in \mathbb{R},
\end{equation}
where $\mathcal{H}$, $J_ {mn}$ and $\mathcal{P}_{n,m}$ are as in Proposition~\ref{rm:Liu12}(iv),
and $\nabla^2_3 \mathcal{H}(\lambda, t, w)$ is the Hessian of $\mathcal{H}\left(\lambda, t , w \right)$
 with respect to the third variable $w\in({\R}^{n})^m\times ({\R}^{n})^m$.
Conversely, if $w(t)=(v(t)^\top, u(t)^\top)^\top$, where $v(t)=(x_1(t)^\top,\dots, x_m(t)^\top)^\top$
with $x_{ i } ( t )\in\mathbb{R}^{n}$ ($i=1,\dots,m$), satisfy (\ref{e:LinearDelay1.7}),
then $x(t):=x_1(t)$ is a  solution of (\ref{e:LinearDelay1.6}).
  \end{enumerate}
  }
\end{remark}

Let $\mathcal{P}_\tau(2n) =\{\gamma\in C([0,\tau],{\rm Sp}(2n))\,|\, \gamma(0)=I_{2n}\}$
and ${\cal P}^\ast_\tau(2n)=\{\gamma\in\mathcal{P}_\tau(2n)\,|\, \gamma(\tau)\in{\rm Sp}(2n)^\ast\}$.
As an extension of the Conley-Zehnder index $i_{\rm CZ}$ to paths in $\mathcal{P}_\tau(2n)\setminus\mathcal{P}^\ast_\tau(2n)$,
  Long \cite{Long97, Long02}  defined the \textsf{Maslov-type index} of $\gamma\in\mathcal{P}_\tau(2n)$ to be a pair of integers $(i_\tau(\gamma),\nu_\tau(\gamma))$, where $\nu_\tau(\gamma):=\dim{\rm Ker}(\gamma(\tau)-I_{2n})$,
  ${i}_\tau(\gamma)=i_{\rm CZ}(\gamma)$ if $\nu_\tau(\gamma)=0$, and
\begin{eqnarray}\label{e:long}
 {i}_\tau(\gamma)=\inf\{i_{\rm CZ}(\beta)\,|\, \beta\in{\cal
P}^\ast_\tau(2n)\;\hbox{is sufficiently $C^0$ close to $\gamma$
in}\;{\cal P}_\tau(2n)\}
\end{eqnarray}
if $\nu_\tau(\gamma)>0$.

Dong \cite{Do06} and Liu \cite{Liu06} each extended
the Maslov-type index $(i_\tau(\gamma),\nu_\tau(\gamma))$ of $\gamma\in\mathcal{P}_\tau(2n)$
 to the case relative to a given symplectic matrix $M\in{\rm Sp}(2n,\mathbb{R})$,
but via  different methods. We denoted their respective extensions by
\begin{equation}\label{e:DongLiuIndex}
(i_{\tau, M}(\gamma),\nu_{\tau,M}(\gamma))\quad\text{and}\quad (i_{\tau}^M(\gamma),\nu_{\tau}^M(\gamma)).
\end{equation}
(The former was written as  $(i_M(\gamma),\nu_M(\gamma))$ in \cite{Do06}).
 They are more suitable and convenient for dual variational methods and saddle point reduction ones, respectively.
If $M=I_{2n}$, by \cite{Do06} and \cite[Definition~2.5 and Definition~2.6]{Liu06} there holds
\begin{equation}\label{e:dongLiuIndex+}
(i_{\tau,M}(\gamma),\nu_{\tau,M}(\gamma))=(i_{\tau}^M(\gamma),\nu_{\tau}^M(\gamma))=(i_\tau(\gamma),\nu_\tau(\gamma)).
\end{equation}
Let $\xi$ be any element in $\mathcal{P}_\tau(2n)$ satisfying $\xi(\tau)=M^{-1}$,
and for a real $\lambda$ let $[\lambda]$ denote the largest integer less than or equal to
it as agreed at the end of Section~\ref{sec:Intro0}.
 Dong  defined
 \begin{equation}\label{e:dongIndex}
  \begin{aligned}
  &\nu_{\tau,M}(\gamma)=\dim{\rm Ker}(\gamma(\tau)-M)\quad\hbox{and}\\
&i_{\tau,M}(\gamma)=[i_\tau((\gamma M^{-1})\ast\xi)-\Delta({\xi})/\pi]
   \end{aligned}
 \end{equation}
(\cite[Definitions~2.1, 2.2]{Do06}), and  Liu (\cite[Definition~2.7 and Remark~2.8]{LiuTang15}) defined
\begin{equation}\label{e:LiuTang-index}
\begin{aligned}
&\nu_{\tau}^M(\gamma)=\dim{\rm Ker}(\gamma(\tau)-M)\quad\hbox{and}\\
&i^{M}_\tau(\gamma)=i_\tau((M^{-1}\gamma)\ast\xi)-i_\tau({\xi})
\end{aligned}
 \end{equation}
 for $M\ne I_{2n}$,  and $(i_{\tau}^M(\gamma),\nu_{\tau}^M(\gamma))=(i_\tau(\gamma),\nu_\tau(\gamma))$
  for $M= I_{2n}$.
 (Note that $i_\tau((M^{-1}\gamma)\ast\xi)$ can be replaced by $i_\tau((\gamma M^{-1})\ast\xi))$,
 see the proof of \cite[Proposition~A.1]{Lu11}.
It was shown in \cite[Remark~2.8]{LiuTang15} that when $M=I_{2n}$
the right side of the second equality in (\ref{e:LiuTang-index}) is $i_\tau(\gamma)+n$.
Therefore, $i_{\tau,M}(\gamma)$ and $i^M_{\tau}(\gamma)$ are generally not equal.)
We derived the following exact relation between
$i_{\tau, M}$ and $i^{M}_\tau$ in \cite{Lu11}.

\begin{proposition}[\hbox{\cite[Proposition~A.1]{Lu11}}]\label{prop:twoDef}
 For any $(M,\gamma)\in ({\rm Sp}(2n,\R)\setminus\{I_{2n}\})\times\mathcal{P}_\tau(2n)$
 and  any $\xi\in\mathcal{P}_\tau(2n)$ satisfying $\xi(\tau)=M^{-1}$
 it holds that
  \begin{equation}\label{e:indexrelation}
 i_{\tau,M}(\gamma)=[i^{M}_\tau(\gamma)+i_\tau(\xi)-\Delta({\xi})/\pi]=i^{M}_\tau(\gamma)+ [i_\tau(\xi)-\Delta({\xi})/\pi].
\end{equation}
 {\rm (}The term $[i_\tau(\xi)-\Delta({\xi})/\pi]$
only depends on $M$.{\rm )}
   \end{proposition}

Let $\mathcal { J }$, $\mathcal { M }$ and $A$ be as given above (\ref{e:M-invariantDelay2}).
For a symmetric matrix-valued function $B\in C\left( [ 0 , \tau ], \mathcal{L}_{s}(\mathbb{R}^{2N})\right)$,
  let $\gamma_B$ be the fundamental matrix solution of
  $\dot { y } = \mathcal { J } B ( t ) y$ with $\gamma_B(0)=I_{2N}$.
  Each $\gamma_B(t)$ is $\mathcal { J }$-symplectic (\cite{Liu12}).
It is clear that  $A^{-1}\gamma_BA$
 is the fundamental  matrix solution of $\dot { y } = J_N (A^\top B ( t )A) y$ with
 $A^{-1}\gamma_B(0)A=I_{2N}$.
 If $\hat{A}$ is another nonsingular $2 N \times 2 N$ matrix such that
$\omega_0(v, w)=\omega(\hat{A}v, \hat{A}w)$ for all $v,w\in\mathbb{R}^{2N}$,
and $\hat{M}:=\hat{A}^{-1}\mathcal{M}\hat{A}$, then
$\omega_0(A^{-1}\hat{A}v, A^{-1}\hat{A}w)=\omega(\hat{A}v, \hat{A}w)=\omega_0(v, w)$ implies that
$Q:=A^{-1}\hat{A}$ is a standard symplectic matrix, and $\hat{M}=\hat{A}^{-1}{A}M(\hat{A}^{-1}A)^{-1}=Q^{-1}MQ$.
From the latter and \cite[Lemma~A.11]{Lu14} it follows that
\begin{eqnarray*}
&&i_{\tau}^M(A^{-1}\gamma_BA)=  i_{\tau}^{\hat{M}}(\hat{A}^{-1}\gamma_B\hat{A})\quad\hbox{and}\\
&&\dim{\rm Ker}(A^{-1}\gamma_B(\tau)A-M)=\dim{\rm Ker}(\hat{A}^{-1}\gamma_B(\tau)\hat{A}-\hat{M}).
\end{eqnarray*}
These show that
\begin{eqnarray}\label{e:NM-index}
i^{\mathcal{J}}_{\mathcal{M}}(\gamma_B):=i^M_\tau(A^{-1}\gamma_BA)\quad\hbox{and}\quad
\nu^{\mathcal{J}}_{\mathcal{M}}(\gamma_B):=\nu^M_\tau(A^{-1}\gamma_BA)
\end{eqnarray}
are well-defined;  they are collectively referred to as the
 the \textsf{$(\mathcal{J}, \mathcal{M})$-index} of $\gamma_B$ in \cite{Liu12} and  \cite{ZhouLZZ22}.

%Actually, the second equality is clear because $\nu^{\mathcal{J}}_{\mathcal{M}}(\gamma_B)$ is equal to
%the dimension of the space of solutions of  $\dot { y } = \mathcal { J } B ( t ) y$ with $y(\tau)=\mathcal{M}y(0)$.
%
%Liu \cite{Liu12}
%introduced the $(\mathcal{J}, \mathcal{M})$-index of $\gamma_B$,
%and Liu et al \cite{ZhouLZZ22}
%equivalently defined  the index as
%\begin{eqnarray}\label{e:NM-index}
%i^{\mathcal{J}}_{\mathcal{M}}(\gamma_B)=i^M_\tau(A^{-1}\gamma_BA)\quad\hbox{and}\quad
%\nu^{\mathcal{J}}_{\mathcal{M}}(\gamma_B)=\nu^M_\tau(A^{-1}\gamma_BA)
%\end{eqnarray}
%via the Maslov-type index $(i_{\tau}^M,\nu_{\tau}^M)$ in (\ref{e:DongLiuIndex}), which was defined by
%Liu \cite{Liu06}.

%
%However, if we replace $i_{\tau}^M$ with the Maslov-type index $i_{\tau, M}$
%in (\ref{e:dongIndex}),
%by the second part of Lemma~\ref{lem:Dong06+} we can only guarantee that
%\begin{eqnarray*}
%i_{\tau,M}(A^{-1}\gamma_BA) = i_{\tau,\hat{M}}(\hat{A}^{-1}\gamma_B\hat{A})
%\end{eqnarray*}
%holds in the case $Q:=A^{-1}\hat{A}$ is an orthogonal symplectic matrix
%(although we always have the equality $\nu_{\tau,M}(A^{-1}\gamma_BA)=  \nu_{\tau,\hat{M}}(\hat{A}^{-1}\gamma_B\hat{A})$).
%Consequently, defining $i_{\mathcal{M}, \mathcal{J}}(\gamma_B)=i_{\tau, M}(A^{-1}\gamma_BA)$
% as in (\ref{e:NM-index}) is ill-posed.
% %For this reason, it becomes important to provide an explicit/algorithmic construction of $A$,
%%  which is also done in Appendix~\ref{app:Matrix}.
%

\section{ Bifurcations for the delay system (\ref{e:Delay1})}\label{sec:delay1}

 Under Assumption~\ref{ass:BasiAss1Delay2}, let $x^\lambda_ { i } ( t ) = x^\lambda( t-(i-1)\tau)$, $i=1,\dots,m$.
Then we have
\begin{equation}\label{e:Delay24}
{\small
\nabla^2_3 H\bigl(\lambda, t, x^\lambda_1(t), \ldots, x^\lambda_m(t)\bigr) =
\begin{pmatrix}
\nabla^2_3 V\bigl(\lambda, t, x^\lambda_1(t)\bigr) & 0 & \cdots & 0 \\
0 & \nabla^2_3 V\bigl(\lambda, t, x^\lambda_2(t)\bigr) & \cdots & 0 \\
\vdots & \vdots & \ddots & \vdots \\
0 & 0 & \cdots & \nabla^2_3 V\bigl(\lambda, t, x^\lambda_m(t)\bigr)
\end{pmatrix}
}
\end{equation}
for \( H \) (as in Proposition~\ref{prop:Liu12}(i)) when \( m \in 2\mathbb{N} \),
where \( \nabla^2_3 H(\lambda, t, v) \) is the Hessian of \( H(\lambda, t, v) \)
with respect to the third variable \( v \in (\mathbb{R}^n)^m \).
Furthermore,
\begin{align*}
&\nabla^2_3\tilde{H}\bigl(\lambda, t, x^\lambda_1(t), \ldots, x^\lambda_{m-1}(t)\bigr)\\
&=
\begin{pmatrix}
\nabla^2_3 V\bigl(\lambda, t, x^\lambda_1(t)\bigr) & 0 & \cdots & 0 \\
0 & \nabla^2_3 V\bigl(\lambda, t, x^\lambda_2(t)\bigr) & \cdots & 0 \\
\vdots & \vdots & \ddots & \vdots \\
0 & 0 & \cdots & \nabla^2_3 V\bigl(\lambda, t, x^\lambda_{m-1}(t)\bigr)
\end{pmatrix}
\\
&\qquad\qquad+
\Bigl(-I_{n},\;I_{n},\;-I_{n},\;\cdots,\;I_{n}\Bigr)^\top
\cdot\nabla^2_3V\biggl(\lambda,t,\sum^{m-1}_{i=1}(-1)^i x^\lambda_i(t)\biggr)
\cdot\Bigl(-I_{n},\;I_{n},\;-I_{n},\;\cdots,\;I_{n}\Bigr),
\end{align*}
\begin{equation}\label{e:Delay25}
\end{equation}
for \( \tilde{H} \) in Proposition~\ref{prop:Liu12}(ii) by \cite[pages 28--29]{ZhouLZZ22} when \( m\in 2\mathbb{N}+1 \),
where \( \nabla^2_3 \tilde{H}(\lambda, t, v) \) is the Hessian of
\( \tilde{H}(\lambda, t, v) \) with respect to the third variable \( v \in (\mathbb{R}^{n})^{m-1} \).
Let $\gamma_{n, m,x}^\lambda$ be the fundamental matrix solution of the linear system
\begin{equation}\label{e:Delay26}
\left.\begin{array}{ll}
&\dot{Z}(t)=A_{m,n}\nabla^2_3H\left(\lambda, t , x^\lambda_ { 1 }(t) , \ldots , x^\lambda_ { m }(t)\right)Z(t)\\
&\hbox{[resp. $\dot{Z}(t)=A_{m-1,n}\nabla^2_3\tilde{H}\left(\lambda, t , x^\lambda_ { 1 }(t) , \ldots , x^\lambda_ { m-1 }(t)\right)                                    Z(t)$]}
\end{array}\right\}
\end{equation}
with $\gamma^\lambda_{n,m,x}(0)=I_{nm}$ (resp. $\gamma^\lambda_{n,m,x}(0)=I_{n(m-1)}$) for $m\in 2\mathbb{N}$ (resp.
$m\in 2\mathbb{N}+1$).

When $m\in 2\mathbb{N}$, applying the reasoning above
(\ref{e:M-invariantDelay2}) to $\mathcal{J}=A_{m,n}$ and $\mathcal{M}=T_{m,n}^{-1}$
yields a nonsingular real $nm\times nm$ matrix $\Upsilon(m,n)$ such that
$A_{m,n}^{-1}={\Upsilon}(m, n)^\top J_{mn/2}^{-1}{\Upsilon}(m, n)$.
%%%%%%%%%%%%%%%%%%%%%%%%%%%%%%%%%%%%%%%%%%%%%%%%%%%%%%%%%%%%%%%%%%%%%
%%%The matrix $\Upsilon(m,n)$  given by Theorem~\ref{th:A_{m,n}} satisfies
%%%$A_{m,n}^{-1}={\Upsilon}(m, n)^\top J_{mn/2}^{-1}{\Upsilon}(m, n)$.
%%%%%%%%%%%%%%%%%%%%%%%%%%%%%%%%%%%%%%%%%%%%%%%%%%%%%%%%%%%%%%%%%%%
Define
\begin{align}\label{e:Delay26.1}
M_{m,n}&={\Upsilon}(m, n)T_{m,n}^{-1}{\Upsilon}(m, n)^{-1},\\
\check{H}(\lambda,t, z)&=H(\lambda,t, {\Upsilon}(m, n)^{-1}z)\quad\forall
(\lambda,t, z)\in \Lambda\times\R\times({\R}^{n})^m.\label{e:Delay26.2}
\end{align}
Then $M_{m,n}$ is a symplectic matrix and satisfies $M_{m,n}^{2m}=I_{mn}$.

For $m\in 2\mathbb{N}+1$ let $\Upsilon(m-1, n)$ be as above, and define
\begin{align}
M_{m,n}&=\Upsilon(m-1, n)\,\tilde{B}_{m-1,n}^{-1}\,\Upsilon(m-1, n)^{-1},
\label{e:Delay26.3}\\
\check{H}(\lambda,t, z)&=
\tilde{H}\!\bigl(\lambda,t,\,
\Upsilon(m-1, n)^{-1}z\bigr)
\quad\forall(\lambda,t, z)\in \Lambda\times\mathbb{R}\times(\mathbb{R}^{n})^{m-1}.
\label{e:Delay26.4}
\end{align}
Then the $M_{m,n}$ defined in \eqref{e:Delay26.3} is again symplectic and satisfies
$M_{m,n}^{2m}=I_{(m-1)n}$.
%(or $I_{(m-1)\,n}$ to be completely unambiguous).

In both cases above, the identity $\check{H}(\lambda,t+\tau, Mz)=\check{H}(\lambda,t, z)$
holds for all $(\lambda,t,z)$.  By Proposition~\ref{prop:Equiv}, we then obtain the following:

\begin{proposition}\label{prop:Equiv1}
For $m\in 2\mathbb{N}$ {\rm (}resp. $m\in 2\mathbb{N}+1${\rm )},
$v:\mathbb{R}\to ({\R}^{n})^m$ {\rm [}resp. $v:\mathbb{R}\to({\R}^{n})^{m-1}${\rm ]}
 solves (\ref{e:Delay6}) {\rm [}resp. (\ref{e:Delay11}){\rm ]} if and only if
$u(t):={\Upsilon}(m, n)v(t)$ {\rm [}resp.  $u(t):={\Upsilon}(m-1, n)v(t)${\rm ]} satisfies the following problem
\begin{equation}\label{e:M-invariantDelay5++}
\left\{\begin{array}{ll}
&\dot{u}(t)={J}_{mn/2} \nabla_3 \check{H}(\lambda, t, u(t))\\
& \hbox{{\rm [}resp. $\dot{u}(t)={J}_{(m-1)n/2} \nabla_3 \check{H}(\lambda, t, u(t))${\rm ]}},\\
& u(t+\tau)={M}_{m,n}u(t)\;\;\forall t\in \R
\end{array}\right.
\end{equation}
with $M_{m,n}={\Upsilon}(m, n)T_{m,n}^{-1}{\Upsilon}(m, n)^{-1}$ {\rm [}resp. $M_{m,n}={\Upsilon}(m-1, n)\tilde{B}_{m-1,n}^{-1}{\Upsilon}(m-1, n)^{-1}${\rm ]},
and in this situation there holds
\begin{align*}
&\int^{\tau}_0\left[\frac{1}{2}(J_{mn/2}\dot{u}(t), u(t))_{\mathbb{R}^{mn}}+ \check{H}(\lambda, t, u(t))\right]dt\nonumber\\
&= \int^{\tau}_0\left[\frac{1}{2}(\dot{v}(t), A_{m,n}^{-1}v(t))_{\mathbb{R}^{mn}}+ H(\lambda, t,
v(t))\right]dt.\\
{\rm [}\hbox{resp.}
&\int^{\tau}_0\left[\frac{1}{2}(J_{(m-1)n/2}\dot{u}(t), u(t))_{\mathbb{R}^{(m-1)n}}+ \check{H}(\lambda, t, u(t))\right]dt\nonumber\\
&= \int^{\tau}_0\left[\frac{1}{2}(\dot{v}(t), A_{m-1,n}^{-1}v(t))_{\mathbb{R}^{(m-1)n}}+ H(\lambda, t,
v(t))\right]dt. {\rm ]}
\end{align*}
\end{proposition}

For each $\lambda\in\Lambda$ and $m\in 2\mathbb{N}$ {\rm (}resp. $m\in 2\mathbb{N}+1${\rm )},
let
%\begin{align}\label{e:M-invariantDelay6++}
%&v^\lambda(t)=(x^\lambda_1(t)^\top,\dots, x^\lambda_m(t)^\top)^\top\in ({\R}^{n})^m\quad\text{and}\quad
%u^\lambda(t):=\Upsilon(m,n)v^\lambda(t)\\
%&{\rm [resp.}\;
%v^\lambda(t)=(x^\lambda_1(t)^\top,\dots, x^\lambda_{m-1}(t)^\top)^\top\in ({\R}^{n})^{m-1}\quad\text{and}\quad
%u^\lambda(t):=\Upsilon(m-1,n)v^\lambda(t)
%{\rm ]},\nonumber
%\end{align}
\begin{equation}\label{e:M-invariantDelay6++}
\left.\begin{array}{ll}
&v^\lambda(t)=(x^\lambda_1(t)^\top,\dots, x^\lambda_m(t)^\top)^\top\in ({\R}^{n})^m\quad\text{and}\quad
u^\lambda(t):=\Upsilon(m,n)v^\lambda(t)\\
&{\rm [resp.}\;
v^\lambda(t)=(x^\lambda_1(t)^\top,\dots, x^\lambda_{m-1}(t)^\top)^\top\in ({\R}^{n})^{m-1}\quad\text{and}\\
&\hspace{70mm} u^\lambda(t):=\Upsilon(m-1,n)v^\lambda(t)
{\rm ]},
\end{array}\right\}.
\end{equation}
where $x^\lambda_ { i } ( t ) = x^\lambda( t-(i-1)\tau)$ for $i=1,\dots,m$.
Then, by Proposition~\ref{prop:Liu12}\,(i)-(ii),  we obtain that
  $v^\lambda$ solves (\ref{e:Delay6}) [resp. (\ref{e:Delay11})], and therefore $u^\lambda$
satisfies (\ref{e:M-invariantDelay5++}) by Proposition~\ref{prop:Equiv1}.

Linearizing the first two lines of (\ref{e:M-invariantDelay5++}) at $u^\lambda$,
we have the linear system
\begin{equation}\label{e:M-invariantDelay7++}
\left.\begin{array}{ll}
&\dot{y}(t)={J}_{mn/2} \nabla^2_3\check{H}(\lambda, t, u^\lambda(t)) y(t)\\
&\hbox{[resp. $\dot{y}(t)={J}_{(m-1)n/2} \nabla^2_3\check{H}(\lambda, t, u^\lambda(t)) y(t)$].}
\end{array}\right\}
\end{equation}
Let $\Gamma_{n, m,x}^\lambda(t)$ be the fundamental matrix solution of it
with $\Gamma_{n, m,x}^\lambda(0)=I_{mn}$. Then
 $$
 \Gamma_{n, m,x}^\lambda(t)=\Upsilon(m,n)\gamma_{n, m,x}^\lambda(t)\Upsilon(m,n)^{-1}
 \;\;{\rm [resp.\;}\Gamma_{n, m,x}^\lambda(t)=\Upsilon(m-1,n)\gamma_{n, m,x}^\lambda(t)\Upsilon(m-1,n)^{-1}\;{\rm ]}
 $$
for $m\in 2\mathbb{N}$ {\rm (}resp. $m\in 2\mathbb{N}+1${\rm )}, where $\gamma_{n, m,x}^\lambda$ is
the fundamental matrix solution  of (\ref{e:Delay26}) as above.

\begin{remark}\label{rm:Equiv1}
{\rm In practice, we lack a general method for choosing the nonsingular  matrix $\Upsilon(m,n)$
to ensure that the symplectic matrices matrixes  $M_{m,n}$ defined in (\ref{e:Delay26.1}) and (\ref{e:Delay26.3}) are also orthogonal.
When $m=2$, by (\ref{e:Delay7})  we have
$A_{2,n} = \begin{pmatrix} 0 & I_n \\
-I_n & 0 \end{pmatrix}=T_{2,n}=-J_n$.
%A direct computation yields the desired congruence relation
%\[
%\begin{pmatrix} 0 & I_n \\ I_n & 0 \end{pmatrix}^\top A_{2,n}^{-1} \begin{pmatrix} 0 & I_n \\ I_n & 0 \end{pmatrix}
%= \begin{pmatrix} 0 & I_n \\ I_n & 0 \end{pmatrix} J_n \begin{pmatrix} 0 & I_n \\ I_n & 0 \end{pmatrix}
%= -J_n = J_n^{-1}.
%\]
Hence, we can take $\Upsilon(2,n) = \begin{pmatrix} 0 & I_n \\ I_n & 0 \end{pmatrix}^{-1} = \begin{pmatrix} 0 & I_n \\ I_n & 0 \end{pmatrix}$. In this case,
\[
M_{2,n} = \Upsilon(2, n)T_{2,n}^{-1}\Upsilon(2, n)^{-1} = -\Upsilon(2, n)J_n^{-1}\Upsilon(2, n)^\top = -J_n
\]
is a symplectic orthogonal matrix and satisfies $(M_{2,n})^4=I_{2n}$.
However, by (\ref{e:Delay12}), $(\tilde{B}_{2,n})^{-1} = \begin{pmatrix} I_n & -I_n \\ I_n & 0 \end{pmatrix}$, and
\[
M_{3,n} = \Upsilon(2, n)\tilde{B}_{2,n}^{-1}\Upsilon(2, n)^{-1} = \begin{pmatrix} 0 & I_n \\ -I_n & I_n \end{pmatrix}
\]
is not a symplectic orthogonal matrix.
%Since the results from \cite{Lu11} apply only to symplectic orthogonal matrices, we therefore restrict our analysis from now on to the case $m=2$.
}
\end{remark}

%According to  \cite[Definition~2.4]{Liu12} and \cite[Definition~2.2]{ZhouLZZ22},
% the Maslov-type index $(i_{\tau}^M, \nu_{\tau}^M)$ in (\ref{e:LiuTang-index}) %yields well-defined indexes

%By (\ref{e:NM-index}), for $m\in 2\mathbb{N}$ {\rm (}resp. $m\in 2\mathbb{N}+1${\rm )}, we have $(A_{m,n}, T_{m,n}^{-1})$-index (resp.
%$(A_{m-1,n}, \tilde{B}_{m-1,n}^{-1})$-index
%) of $\gamma_{n, m,x}^\lambda$

 By (\ref{e:NM-index}),  we have the $(A_{m,n}, T_{m,n}^{-1})$-index
 of $\gamma_{n, m,x}^\lambda$  defined by
\begin{equation}\label{e:Delay27}
\left\{\begin{array}{ll}
&i_{T_{m,n}^{-1}}^{A_{m,n}}(\gamma_{n, m,x}^\lambda):=i_{\tau}^{M_{m,n}}(\Gamma_{n, m,x}^\lambda)=i_{\tau}^{M_{m,n}}(\Upsilon(m,n)\gamma_{n, m,x}^\lambda\Upsilon(m,n)^{-1}),\\
&\nu_{T_{m,n}^{-1}}^{A_{m,n}}(\gamma_{n, m,x}^\lambda):=\nu_{\tau}^{M_{m,n}}(\Gamma_{n, m,x}^\lambda)=\nu_{\tau}^{M_{m,n}}(\Upsilon(m,n)\gamma_{n, m,x}^\lambda\Upsilon(m,n)^{-1})
\end{array}\right.
\end{equation}
with $M_{m,n}={\Upsilon}(m, n)T_{m,n}^{-1}{\Upsilon}(m, n)^{-1}$ for $m\in 2\mathbb{N}$,
 and the $(A_{m-1,n}, \tilde{B}_{m-1,n}^{-1})$-index
 of $\gamma_{n, m,x}^\lambda$ defined by
\begin{equation}\label{e:Delay28}
\left\{\begin{array}{ll}
&i_{\tilde{B}_{m-1,n}^{-1}}^{A_{m-1,n}}(\gamma_{n, m,x}^\lambda):=i_{\tau}^{M_{m,n}}(\Gamma_{n, m,x}^\lambda)=i_{\tau}^{M_{m,n}}(\Upsilon(m-1,n)\gamma_{n, m,x}^\lambda\Upsilon(m-1,n)^{-1}),\\
&\nu_{\tilde{B}_{m-1,n}^{-1}}^{A_{m-1,n}}(\gamma_{n, m,x}^\lambda):=\nu_{\tau}^{M_{m,n}}(\Gamma_{n, m,x}^\lambda)=\nu_{\tau}^{M_{m,n}}(\Upsilon(m-1,n)\gamma_{n, m,x}^\lambda\Upsilon(m-1,n)^{-1})
\end{array}\right.
\end{equation}
with $M_{m,n}={\Upsilon}(m-1, n)\tilde{B}_{m-1,n}^{-1}{\Upsilon}(m-1, n)^{-1}$
for $m\in 2\mathbb{N}+1$.

For convenience, in what follows we will always write
\begin{equation}\label{e:Delay277*}
i_{T_{m,n}^{-1}}^{A_{m,n}}(x^\lambda):=
i_{T_{m,n}^{-1}}^{A_{m,n}}(\gamma_{n, m,x}^\lambda)
\quad\text{and}\quad
\nu_{T_{m,n}^{-1}}^{A_{m,n}}(x^\lambda):=\nu_{T_{m,n}^{-1}}^{A_{m,n}}(\gamma_{n, m,x}^\lambda)
\end{equation}
for $m\in 2\mathbb{N}$, and
\begin{equation}\label{e:Delay288*}
i_{\tilde{B}_{m-1,n}^{-1}}^{A_{m-1,n}}(x^\lambda):=
i_{\tilde{B}_{m-1,n}^{-1}}^{A_{m-1,n}}(\gamma_{n, m,x}^\lambda)
\quad\text{and}\quad
\nu_{\tilde{B}_{m-1,n}^{-1}}^{A_{m-1,n}}(x^\lambda):=
\nu_{\tilde{B}_{m-1,n}^{-1}}^{A_{m-1,n}}(\gamma_{n, m,x}^\lambda)
\end{equation}
for $m\in 2\mathbb{N}+1$.

Note that for $m\in 2\mathbb{N}$ (resp. $m\in 2\mathbb{N}+1$), $\nu_{T_{m,n}^{-1}}^{A_{m,n}}(x^\lambda)$
(resp. $\nu_{\tilde{B}_{m-1,n}^{-1}}^{A_{m-1,n}}(x^\lambda)$)
is equal to the dimension of the solution space of  the linear problem
\begin{equation}\label{e:Delay28+}
\left.\begin{array}{ll}
&\dot{Z}(t)=A_{m,n}\nabla^2_3H\left(\lambda, t , x^\lambda_ { 1 }(t) , \ldots , x^\lambda_ { m }(t)\right)Z(t)\\
&Z(t+\tau)=T_{m,n}^{-1}Z(t)\;\;\forall t\in \R
\end{array}\right\}
\end{equation}
\begin{equation}\label{e:Delay28++}
\left({\rm resp}. \left.\begin{array}{ll}
&\dot{Z}(t)=A_{m-1,n}\nabla^2_3\tilde{H}\left(\lambda, t , x^\lambda_ { 1 }(t) , \ldots , x^\lambda_ { m-1 }(t)\right)                                    Z(t)\\
&Z(t+\tau)=\tilde{B}_{m-1,n}^{-1}Z(t)\;\;\forall t\in \R,
\end{array}\right\}\right).
\end{equation}
Equivalently,
$\nu_{T_{m,n}^{-1}}^{A_{m,n}}(x^\lambda)$
(resp. $\nu_{\tilde{B}_{m-1,n}^{-1}}^{A_{m-1,n}}(x^\lambda)$)
  equals the dimension of the solution space of the linear delay problem (\ref{e:LinearDelay1.1Lambda})
  for $m\in 2\mathbb{N}$ (resp. $m\in 2\mathbb{N}+1$) by Remark~\ref{rm:Liu12}(i)-(ii).

Applying Theorem~1.5 of \cite{Lu11} to the problem (\ref{e:M-invariantDelay5++}) with the choices $(H, u_\lambda, M)=(\check{H}, u^\lambda, M_{m,n})$,  and with
$u^\lambda(t)$ given by $\Upsilon(m,n)v^\lambda(t)$
as in (\ref{e:M-invariantDelay6++}), we immediately obtain:

\begin{proposition}\label{prop:bif-per1Delay1}
 \begin{enumerate}
\item[\rm (I)]{\rm (\textsf{Necessary condition}):}
If $(\mu, u^\mu)$ is a bifurcation point along sequences of (\ref{e:M-invariantDelay5++})
 with respect to the branch $\{(\lambda, u^\lambda)\,|\,\lambda\in\Lambda\}$,
 then $\nu_{\tau, M_{m,n}}(\Gamma_{n,m,x}^{\mu})\ne 0$.

\item[\rm (II)]{\rm (\textsf{Sufficient condition}):}
Let $\Lambda$ be first countable.
Suppose  that for some $\mu\in\Lambda$,
there exist two sequences  $(\lambda_k^-)$ and
$(\lambda_k^+)$ in $\Lambda$  such that $\lambda_k^-\to\mu$ and $\lambda_k^+\to\mu$.
 For each \(k\in\mathbb{N}\), let
\begin{align*}
&[i_{\tau, M_{m,n}}(\Gamma_{n,m,x}^{\lambda^-_k}), i_{\tau, M_{m,n}}(\Gamma_{n,m,x}^{\lambda^-_k})+\nu_{\tau, M_{m,n}}(\Gamma_{n,m,x}^{\lambda^-_k})]\\
&\cap[i_{\tau, M_{m,n}}(\Gamma_{n,m,x}^{\lambda^+_k}), i_{\tau, M_{m,n}}(\Gamma_{n,m,x}^{\lambda^+_k})+\nu_{\tau, M_{m,n}}(\Gamma_{n,m,x}^{\lambda^+_k})]
=\emptyset,
\end{align*}
with either \(\nu_{\tau, M_{m,n}}(\Gamma_{n,m,x}^{\lambda^+_k})=0\) or \(\nu_{\tau, M_{m,n}}(\Gamma_{n,m,x}^{\lambda^-_k})=0\).
Define $\hat{\Lambda}:=\{\mu,\lambda^+_k, \lambda^-_k\,|\,k\in\mathbb{N}\}$.
  Then  $(\mu, u^\mu)$ is a bifurcation point  of problem (\ref{e:M-invariantDelay5++})
  with respect to the branch $\{(\lambda, u^\lambda)\,|\,\lambda\in\hat\Lambda\}$
 (and hence also with respect to the larger branch $\{(\lambda, u^\lambda) \mid \lambda \in \Lambda\}$).

\item[\rm (III)]{\rm (\textsf{Existence for bifurcations}):}
 Let $\Lambda$ be path-connected. Suppose that there exist two  points $\lambda^+, \lambda^-\in\Lambda$ such that
 \begin{align*}
&[i_{\tau, M_{m,n}}(\Gamma_{n,m,x}^{\lambda^-}), i_{\tau, M_{m,n}}(\Gamma_{n,m,x}^{\lambda^-})+\nu_{\tau, M_{m,n}}(\Gamma_{n,m,x}^{\lambda^-})]\\
&\cap[i_{\tau, M_{m,n}}(\Gamma_{n,m,x}^{\lambda^+}), i_{\tau, M_{m,n}}(\Gamma_{n,m,x}^{\lambda^+})+\nu_{\tau, M_{m,n}}(\Gamma_{n,m,x}^{\lambda^+})]
=\emptyset,
\end{align*}
 with $\nu_{\tau, M_{m,n}}(\Gamma_{n,m,x}^{\lambda^+})=0$ or $\nu_{\tau, M_{m,n}}(\Gamma_{n,m,x}^{\lambda^-})=0$.
 Then for any path $\alpha: [0,1] \to \Lambda$ connecting $\lambda^+$ to $\lambda^-$,
there exist $t_k \to \bar{t}$ in $[0,1]$ and solutions $u^k \neq u^{\alpha(t_k)}$ of \eqref{e:M-invariantDelay5++} with $\lambda = \alpha(t_k)$ such that $u^k \to u^{\alpha(\bar{t})}$ in $C^1_{\mathrm{loc}}(\mathbb{R}, \mathbb{R}^{2n})$ as $k\to\infty$.
 Moreover, $\alpha(\bar{t}) \neq \lambda^+$ if $\nu_{\tau, M_{m,n}}(\Gamma_{n,m,x}^{\lambda^+}) = 0$,
and $\alpha(\bar{t}) \neq \lambda^-$ if $\nu_{\tau, M_{m,n}}(\Gamma_{n,m,x}^{\lambda^-}) = 0$.
  \end{enumerate}
\end{proposition}

Let $(i_{\tau, M}, \nu_{\tau, M})$ be  the Maslov-type index  in (\ref{e:dongIndex}).
Together with (\ref{e:LiuTang-index}) and Proposition~\ref{prop:twoDef}, this yields the following relations.

For $m\in 2\mathbb{N}$, we have
\begin{equation}\label{e:Delay27+Dong}
\left\{\begin{array}{ll}
&i_{T_{m,n}^{-1}}^{A_{m,n}}(x^\lambda)=i_{\tau}^{M_{m,n}}(\Gamma_{n, m,x}^\lambda)=i_{\tau, M_{m,n}}(\Gamma_{n, m,x}^\lambda)+ N_{m,n},\\
&\nu_{T_{m,n}^{-1}}^{A_{m,n}}(x^\lambda)=\nu_{\tau}^{M_{m,n}}(\Gamma_{n, m,x}^\lambda)=\nu_{\tau, M_{m,n}}(\Gamma_{n, m,x}^\lambda)
\end{array}\right.
\end{equation}
where $N_{m,n}$ is an integer depending only on $M_{m,n}={\Upsilon}(m, n)T_{m,n}^{-1}{\Upsilon}(m, n)^{-1}$.
For $m\in 2\mathbb{N}+1$, we have
\begin{equation}\label{e:Delay28+Dong}
\left\{\begin{array}{ll}
&i_{\tilde{B}_{m-1,n}^{-1}}^{A_{m-1,n}}(x^\lambda)=i_{\tau}^{M_{m,n}}(\Gamma_{n, m,x}^\lambda)=i_{\tau, M_{m,n}}(\Gamma_{n, m,x}^\lambda)+ N_{m,n},\\
&\nu_{\tilde{B}_{m-1,n}^{-1}}^{A_{m-1,n}}(x^\lambda)=\nu_{\tau}^{M_{m,n}}(\Gamma_{n, m,x}^\lambda)=\nu_{\tau, M_{m,n}}(\Gamma_{n, m,x}^\lambda)
\end{array}\right.
\end{equation}
where $N_{m,n}$ is an integer depending only on $M_{m,n}={\Upsilon}(m-1, n)\tilde{B}_{m-1,n}^{-1}{\Upsilon}(m-1, n)^{-1}$.

%By (\ref{e:Delay27})-(\ref{e:Delay28}),
By Proposition~\ref{prop:Liu12}(i)-(ii), and Proposition~\ref{prop:Equiv1},
we may derive (I), (II), and (III) of the following theorem from (I), (II), and (III) of Proposition~\ref{prop:bif-per1Delay1}, respectively.

\begin{theorem}\label{th:bif-per1Delay1}
Under Assumption~\ref{ass:BasiAss1Delay2}, for $m\in 2\mathbb{N}$,
%Let $i_{\tau}(x^\lambda)$ and $\nu_{\tau}(x^\lambda)$
%be defined by (\ref{e:Delay27}).
 the following conclusions hold.
 \begin{enumerate}
\item[\rm (I)]{\rm (\textsf{Necessary condition}):}
 If the pair $(\mu, x^\mu)$ is a bifurcation point along sequences of (\ref{e:Delay1})
 with respect to the branch $\{(\lambda, x^\lambda)\,|\,\lambda\in\Lambda\}$,
 then $\nu_{T_{m,n}^{-1}}^{A_{m,n}}(x^\mu)\ne 0$.

\item[\rm (II)]{\rm (\textsf{Sufficient condition}):}
Let  $\Lambda$ be first countable.
Suppose that for some $\mu\in\Lambda$,
there exist two sequences $(\lambda_k^-)$ and
$(\lambda_k^+)$ in  $\Lambda$ such that $\lambda_k^\pm\to\mu$,  and such that
for each $k\in\mathbb{N}$,
$$
[i_{T_{m,n}^{-1}}^{A_{m,n}}(x^{\lambda_k^-}), i_{T_{m,n}^{-1}}^{A_{m,n}}(x^{\lambda_k^-})+\nu_{T_{m,n}^{-1}}^{A_{m,n}}(x^{\lambda_k^-})]
\cap[i_{T_{m,n}^{-1}}^{A_{m,n}}(x^{\lambda_k^+}), i_{T_{m,n}^{-1}}^{A_{m,n}}(x^{\lambda_k^+})+
\nu_{T_{m,n}^{-1}}^{A_{m,n}}(x^{\lambda_k^+})]=\emptyset,
$$
 and moreover, either $\nu_{T_{m,n}^{-1}}^{A_{m,n}}(x^{\lambda_k^+})=0$ or $\nu_{T_{m,n}^{-1}}^{A_{m,n}}(x^{\lambda_k^-})=0$.
Let $\hat{\Lambda}:=\{\mu,\lambda^+_k, \lambda^-_k\,|\,k\in\mathbb{N}\}$.
Then $(\mu, x^\mu)$ is a bifurcation point of  problem (\ref{e:Delay1}) with respect to the branch
$\{(\lambda, x^\lambda) \mid \lambda \in \hat\Lambda\}$; hence, it is also a bifurcation point with respect to the larger branch
$\{(\lambda, x^\lambda) \mid \lambda \in \Lambda\}$.

\item[\rm (III)]{\rm (\textsf{Existence for bifurcations}):}
 Let  $\Lambda$ be path-connected. Suppose that there exist two  points $\lambda^+, \lambda^-\in\Lambda$ such that
  $$
[i_{T_{m,n}^{-1}}^{A_{m,n}}(x^{\lambda^-}), i_{T_{m,n}^{-1}}^{A_{m,n}}(x^{\lambda^-})+\nu_{T_{m,n}^{-1}}^{A_{m,n}}(x^{\lambda^-})]
\cap[i_{T_{m,n}^{-1}}^{A_{m,n}}(x^{\lambda^+}), i_{T_{m,n}^{-1}}^{A_{m,n}}(x^{\lambda^+})+
\nu_{T_{m,n}^{-1}}^{A_{m,n}}(x^{\lambda^+})]=\emptyset,
$$
 and either $\nu_{T_{m,n}^{-1}}^{A_{m,n}}(x^{\lambda^+})=0$ or $\nu_{T_{m,n}^{-1}}^{A_{m,n}}(x^{\lambda^-})=0$.
  Then for any path $\alpha: [0,1] \to \Lambda$ connecting $\lambda^+$ to $\lambda^-$,
there exist $t_k \to \bar{t}$ in $[0,1]$ and solutions $x^k \neq x^{\alpha(t_k)}$ of
(\ref{e:Delay1}) with $\lambda = \alpha(t_k)$ such that $x^k \to x^{\alpha(\bar{t})}$ in $C^1_{\mathrm{loc}}(\mathbb{R}, \mathbb{R}^{n})$ as $k\to\infty$.
 Moreover, $\alpha(\bar{t}) \neq \lambda^+$ if $\nu_{T_{m,n}^{-1}}^{A_{m,n}}(x^{\lambda^+})=0$,
 and $\alpha(\bar{t}) \neq \lambda^-$ if  $\nu_{T_{m,n}^{-1}}^{A_{m,n}}(x^{\lambda^-})=0$.
\end{enumerate}
Correspondingly, for  $m\in 2\mathbb{N}+1$,
 the statements (I)-(III) above continue to hold if we replace $i_{T_{m,n}^{-1}}^{A_{m,n}}$ and $\nu_{T_{m,n}^{-1}}^{A_{m,n}}$  by
$i_{\tilde{B}_{m-1,n}^{-1}}^{A_{m-1,n}}$ and $\nu_{\tilde{B}_{m-1,n}^{-1}}^{A_{m-1,n}}$, respectively.
\end{theorem}

\begin{theorem}[\textsf{Alternative bifurcations of Rabinowitz type}]\label{th:bif-per2Delay}
 Let  Assumption~\ref{ass:BasiAss1Delay2} hold with $\Lambda$ being a real interval.
 For $m\in 2\mathbb{N}$, let $i_{T_{m,n}^{-1}}^{A_{m,n}}(x^\lambda)$ and
 $\nu_{T_{m,n}^{-1}}^{A_{m,n}}(x^\lambda)$ be defined by (\ref{e:Delay277*}).
Suppose that $\mu$ is an interior point of $\Lambda$ such that
 $\nu_{T_{m,n}^{-1}}^{A_{m,n}}(x^\mu)\ne 0$,  and that $\nu_{T_{m,n}^{-1}}^{A_{m,n}}(x^\lambda)=0$
 for all $\lambda\in\Lambda\setminus\{\mu\}$ sufficiently close to $\mu$. Moreover, suppose that
  $i_{T_{m,n}^{-1}}^{A_{m,n}}(x^\lambda)$ takes, respectively, values $i_{T_{m,n}^{-1}}^{A_{m,n}}(x^\mu)$ and
  $i_{T_{m,n}^{-1}}^{A_{m,n}}(x^\mu)+ \nu_{T_{m,n}^{-1}}^{A_{m,n}}(x^\mu)$
 as $\lambda\in\Lambda$ varies in
 two deleted half neighborhoods  of $\mu$.
 Then one of the following alternatives occurs:% one of the following assertions holds:
\begin{enumerate}
\item[\rm (i)] For $\lambda\ne\mu$,
 problem (\ref{e:Delay1}) possesses  a sequence of solutions $\{x^{\mu,j}\}^\infty_{j=1}$,
  with $x^{\mu,j}\ne x^\mu$, such that  $x^{\mu,j}|_{[0,\tau]}\to x^\mu|_{[0,\tau]}$ in $C^1([0,\tau];\R^{n})$ as $j\to\infty$.

\item[\rm (ii)]  For every $\lambda\in\Lambda\setminus\{\mu\}$ sufficiently close to $\mu$,
problem (\ref{e:Delay1}) admits a  solution $\bar{x}^\lambda\ne x^\lambda$ satisfying
  $\bar{x}^\lambda|_{[0,\tau]}\to x^\lambda|_{[0,\tau]}$ in $C^1([0,\tau];\R^{n})$  as $\lambda\to \mu$.

\item[\rm (iii)] For a given neighborhood $\mathcal{W}$ of $x^\mu|_{[0,\tau]}$ in $C^1([0,\tau];\R^{n})$,
there is a one-sided  neighborhood $\Lambda^0$ of $\mu$ such that
for any $\lambda\in\Lambda^0\setminus\{\mu\}$, problem (\ref{e:Delay1}) (with parameter $\lambda$)
admits at least two distinct solutions, $\bar{x}^\lambda\ne x^\lambda$ and $\hat{x}^\lambda\ne x^\lambda$, satisfying $\bar{x}^\lambda|_{[0,\tau]}, \hat{x}^\lambda|_{[0,\tau]}\in \mathcal{W}$.
Moreover, if $\nu_{T_{m,n}^{-1}}^{A_{m,n}}(x^\mu)>1$
and if, for parameter $\lambda\in\Lambda^0\setminus\{\mu\}$, problem (\ref{e:Delay1}) (with parameter  $\lambda$) has only finitely many  solutions whose restrictions to $[0,\tau]$ belong to $\mathcal{W}$, then the above solutions $\bar{x}^\lambda$ and $\hat{x}^\lambda$  can, in addition,
be chosen to satisfy the following condition for $H$ in (\ref{e:Delay4}):
\begin{align}\label{e:Delay26++}
&\int^{\tau}_0\left[\frac{1}{2}(\dot{\bar{v}}^\lambda(t), A_{m,n}^{-1}\bar{v}^\lambda(t))_{\mathbb{R}^{mn}}+ H(\lambda, t, \bar{v}^\lambda(t))\right]dt\nonumber\\
&\ne \int^{\tau}_0\left[\frac{1}{2}(\dot{\hat{v}}^\lambda(t), A_{m,n}^{-1}\hat{v}^\lambda(t))_{\mathbb{R}^{mn}}+ H(\lambda, t,
\hat{v}^\lambda(t))\right]dt
\end{align}
 where $\bar{v}^\lambda(t):=(\bar{x}^\lambda_1(t)^\top, \cdots, \bar{x}^\lambda_m(t)^\top)^\top\in ({\R}^{n})^m$
with $\bar{x}^\lambda_ { i } ( t ) = \bar{x}^\lambda( t-(i-1)\tau)$ ($i=1,
\cdots,m$),
and $\hat{v}^\lambda(t):=(\hat{x}^\lambda_1(t)^\top, \cdots, \hat{x}^\lambda_m(t)^\top)^\top\in ({\R}^{n})^m$
with $\hat{x}^\lambda_ { i } ( t ) =
\hat{x}^\lambda( t-(i-1)\tau)$ ($i=1,\cdots,m$). % {\bf and}
\end{enumerate}
Correspondingly, for $m\in 2\mathbb{N}+1$, the following replacements should be made to the statements above:
\begin{enumerate}
\item[$\bullet$]  $i_{T_{m,n}^{-1}}^{A_{m,n}}$ and $\nu_{T_{m,n}^{-1}}^{A_{m,n}}$ are replaced by
$i_{\tilde{B}_{m-1,n}^{-1}}^{A_{m-1,n}}$ and $\nu_{\tilde{B}_{m-1,n}^{-1}}^{A_{m-1,n}}$ in (\ref{e:Delay288*}), respectively;
\item[$\bullet$] the part ``$H$ in (\ref{e:Delay4}):....($i=1,\cdots,m$).'' in (iii)
is replaced by ``$\tilde{H}$ in (\ref{e:Delay9}),
\begin{eqnarray}\label{e:Delay27++}
&&\int^{\tau}_0\left[\frac{1}{2}(\dot{\bar{v}}^\lambda(t), A_{m-1,n}^{-1}\bar{v}^\lambda(t))_{\mathbb{R}^{(m-1)n}}+\tilde{H}(\lambda, t,
\bar{v}^\lambda(t))\right]dt\nonumber\\
&&\ne \int^{\tau}_0\left[\frac{1}{2}(\dot{\hat{v}}^\lambda(t), A_{m-1,n}^{-1}\hat{v}^\lambda(t))_{\mathbb{R}^{(m-1)n}}+ \tilde{H}(\lambda, t,
\hat{v}^\lambda(t))\right]dt
\end{eqnarray}
where $\bar{v}^\lambda(t):=(\bar{x}^\lambda_1(t)^\top,\cdots, \bar{x}^\lambda_{m-1}(t)^\top)^\top\in ({\R}^{n})^{m-1}$
with $\bar{x}^\lambda_ { i } ( t ) = \bar{x}^\lambda( t-(i-1)\tau)$ for $i=1,\dots,m-1$,
and $\hat{v}^\lambda(t):=(\hat{x}^\lambda_1(t)^\top,\cdots, \hat{x}^\lambda_{m-1}(t)^\top)^\top\in ({\R}^{n})^{m-1}$
with $\hat{x}^\lambda_ { i } ( t ) = \hat{x}^\lambda( t-(i-1)\tau)$ for $i=1,\dots,m-1$.''
\end{enumerate}
 \end{theorem}

Since  $H$ in (\ref{e:Delay26++}) is even,
the inequality in (\ref{e:Delay26++})  shows
 that $\bar{v}^\lambda\ne -\hat{v}^\lambda$ and
hence $\bar{x}^\lambda\ne -\hat{x}^\lambda$.

\begin{remark}\label{rm:bif-per2Delay}
{\rm When $m=2$, Equation (\ref{e:Delay26++}) can be stated more precisely as follows:
\begin{align}\label{e:Delay26+++}
&(\bar{x}^\lambda(0), \bar{x}^\lambda(\tau))_{\mathbb{R}^n}-
 \int^{\tau}_0(\dot{\bar{x}}^\lambda(t), \bar{x}^\lambda(t-\tau))_{\mathbb{R}^n}dt+\int^{2\tau}_0V(\lambda,t,
 \bar{x}^\lambda(t))dt\nonumber\\
&\ne (\hat{x}^\lambda(0), \hat{x}^\lambda(\tau))_{\mathbb{R}^n}-
 \int^{\tau}_0(\dot{\hat{x}}^\lambda(t), \hat{x}^\lambda(t-\tau))_{\mathbb{R}^n}dt+\int^{2\tau}_0V(\lambda,t,
 \hat{x}^\lambda(t))dt.
\end{align}
(Clearly, this  also implies $\bar{x}^\lambda\ne -\hat{x}^\lambda$.) In fact, since $A_{2,n}^{-1}=J_n$, then
$(\dot{\bar{v}}^\lambda(t), A_{2,n}^{-1}\bar{v}^\lambda(t))_{\mathbb{R}^{2n}}=
-(\dot{\bar{x}}_1^\lambda(t), \bar{x}_2^\lambda(t))_{\mathbb{R}^{n}}+
(\dot{\bar{x}}_2^\lambda(t), \bar{x}_1^\lambda(t))_{\mathbb{R}^{n}}$.
Note that $\bar{x}^{\lambda}(t+2\tau)=- \bar{x}^\lambda( t)=-\bar{x}^\lambda_ {1} ( t )$ and $\bar{x}^\lambda_ { 2 } ( t ) = \bar{x}^\lambda( t-\tau)$
 for all $t\in\mathbb{R}$.
Using integration by parts, we obtain
\begin{align*}
\int^\tau_0(\dot{\bar{x}}_2^\lambda(t), \bar{x}_1^\lambda(t))_{\mathbb{R}^{n}}dt
&=({\bar{x}}_2^\lambda(t), \bar{x}_1^\lambda(t))_{\mathbb{R}^{n}}\big|^\tau_0-
\int^\tau_0(\bar{x}_2^\lambda(t), \dot{\bar{x}}_1^\lambda(t))_{\mathbb{R}^{n}}dt\\
&=2({\bar{x}}^\lambda(0), \bar{x}^\lambda(\tau))_{\mathbb{R}^{n}}
-\int^\tau_0(\dot{\bar{x}}_1^\lambda(t), \bar{x}_2^\lambda(t))_{\mathbb{R}^{n}}dt.
\end{align*}
It follows that
\begin{align*}
\int^\tau_0(\dot{\bar{v}}^\lambda(t), A_{2,n}^{-1}\bar{v}^\lambda(t))_{\mathbb{R}^{2n}}dt&=
-\int^\tau_0(\dot{\bar{x}}_1^\lambda(t), \bar{x}_2^\lambda(t))_{\mathbb{R}^{n}}dt+
\int^\tau_0(\dot{\bar{x}}_2^\lambda(t), \bar{x}_1^\lambda(t))_{\mathbb{R}^{n}}dt\\
&=2({\bar{x}}^\lambda(0), \bar{x}^\lambda(\tau))_{\mathbb{R}^{n}}
-2\int^\tau_0(\dot{\bar{x}}^\lambda(t), \bar{x}^\lambda(t-\tau))_{\mathbb{R}^{n}}dt.
\end{align*}
Moreover, since  $V(\lambda,t, x)$ is even  in $x$ and $\tau$-periodic in $t$, there holds
\begin{align*}
\int^\tau_0V(\lambda, t, \bar{x}^\lambda(t-\tau))dt&=
\int^\tau_0V(\lambda, t, -\bar{x}^\lambda(t+\tau))dt\\
&=\int^\tau_0V(\lambda, t, \bar{x}^\lambda(t+\tau))dt\\
&=\int^\tau_0V(\lambda, t+\tau, \bar{x}^\lambda(t+\tau))dt\\
&=\int^{2\tau}_\tau V(\lambda, t, \bar{x}^\lambda(t))dt.
\end{align*}
Hence, it follows from
$H(\lambda, t, \bar{v}^\lambda(t))=V(\lambda, t, \bar{x}^\lambda(t))+
V(\lambda, t, \bar{x}^\lambda(t-\tau))$ that
$$
\int^\tau_0 H(\lambda, t, \bar{v}^\lambda(t))dt=\int^{2\tau}_0V(\lambda,t,
 \bar{x}^\lambda(t))dt.
$$
These show that (\ref{e:Delay26++}) and  (\ref{e:Delay26+++}) are equivalent.}
\end{remark}

\begin{proof}[\bf Proof of Theorem~\ref{th:bif-per2Delay}]
%From By the assumptions and (\ref{e:Delay27})-(\ref{e:Delay28}), we have:
% $\nu_{\tau, M_{2,n}}(\Gamma_{n,2,x}^\mu)\ne 0$,  $\nu_{\tau, M_{2,n}}(\Gamma_{n,2,x}^\lambda)=0$
% for all $\lambda\in\Lambda\setminus\{\mu\}$ sufficiently close to $\mu$, and
%  $i_{\tau, M_{2,n}}(\Gamma_{n,2,x}^\lambda)$ takes, respectively, values
%  $i_{\tau, M_{2,n}}(\Gamma_{n,2,x}^\mu)$ and
%  $i_{\tau, M_{2,n}}(\Gamma_{n,2,x}^\mu)+ \nu_{\tau, M_{2,n}}(\Gamma_{n,2,x}^\mu)$
% as $\lambda\in\Lambda$ varies in  two deleted half neighborhoods  of $\mu$.
We only prove the case $m\in 2\mathbb{N}$ because a similar proof can be completed for $m\in 2\mathbb{N}+1$.
From the assumptions, and (\ref{e:Delay27}), (\ref{e:Delay277*}) and (\ref{e:Delay27+Dong}) it follows that:
$$
\nu_{\tau, M_{m,n}}(\Gamma_{n,m,x}^\mu)\ne 0,\quad\text{whereas\quad
  $\nu_{\tau, M_{m,n}}(\Gamma_{n,m,x}^\lambda)=0$\;
 for all $\lambda\in\Lambda\setminus\{\mu\}$ sufficiently close to $\mu$, }
$$
and that the index
  $i_{\tau, M_{m,n}}(\Gamma_{n,m,x}^\lambda)$ takes values
  $i_{\tau, M_{m,n}}(\Gamma_{n,m,x}^\mu)$ and
  $i_{\tau, M_{m,n}}(\Gamma_{n,m,x}^\mu)+ \nu_{\tau, M_{m,n}}(\Gamma_{n,m,x}^\mu)$
 on the  two deleted half-neighborhoods  of $\mu$,  respectively.

In the first part of Theorem~1.8 in \cite{Lu11}, taking $(H, u_\lambda, M)$ as
 $(\check{H}, u^\lambda, M_{m,n})$ in problem (\ref{e:M-invariantDelay5++}) (where $u^\lambda$ is as above (\ref{e:M-invariantDelay7++})),
 we may obtain three assertions (i)-(iii).
  Then, as shown in the proof of Theorem~\ref{th:bif-per1Delay1}, we deduce that they
  correspond respectively to (i), (ii), and (iii) in Theorem~\ref{th:bif-per2Delay}.
   \end{proof}

\begin{theorem}[\textsf{Alternative bifurcations of  Fadell-Rabinowitz type}]\label{th:bif-per2Delay+}
Suppose that Assumption~\ref{ass:BasiAss1Delay1} holds with $\Lambda$ being a real interval and
with $x^\lambda\equiv {\bf 0}^\lambda$ (the zero solution of (\ref{e:Delay1}) with parameter  $\lambda$).
 For $m\in 2\mathbb{N}$,
let $i_{T_{m,n}^{-1}}^{A_{m,n}}({\bf 0}^\lambda)$ and
 $\nu_{T_{m,n}^{-1}}^{A_{m,n}}({\bf 0}^\lambda)$ be defined by (\ref{e:Delay277*}).
Suppose that $\mu$ is an interior point of $\Lambda$ such that
 $\nu_{T_{m,n}^{-1}}^{A_{m,n}}({\bf 0}^\mu)\ne 0$,  and that $\nu_{T_{m,n}^{-1}}^{A_{m,n}}({\bf 0}^\lambda)=0$
 for all $\lambda\in\Lambda\setminus\{\mu\}$ sufficiently close to $\mu$. Moreover, suppose that
  $i_{T_{m,n}^{-1}}^{A_{m,n}}({\bf 0}^\lambda)$ takes, respectively, values $i_{T_{m,n}^{-1}}^{A_{m,n}}({\bf 0}^\mu)$ and
  $i_{T_{m,n}^{-1}}^{A_{m,n}}({\bf 0}^\mu)+ \nu_{T_{m,n}^{-1}}^{A_{m,n}}({\bf 0}^\mu)$
 as $\lambda\in\Lambda$ varies in
 two deleted half neighborhoods  of $\mu$.
 %%%%%%%%%%%%%%%%%%%%%%%%%%%%%%%%%%%%%%%%%%%%%%%%%%%%%%%%%%%%%%%%%%%%%%%%%%%%%%%%%%%%%%%%%%%%%%%%%%%%
%For $m=2$, let  Assumption~\ref{ass:BasiAss1Delay1}  with $\Lambda$ being a real interval be  satisfied,
%and let ${\bf 0}^\lambda$ be the zero solution of (\ref{e:Delay1}) with parameter  $\lambda$.
%Suppose  that for some interior point $\mu$ of $\Lambda$,
% $\nu_{\tau}({\bf 0}^\mu)\ne 0$ and  $\nu_{\tau}({\bf 0}^\lambda)=0$
% for each $\lambda\in\Lambda\setminus\{\mu\}$ sufficiently close to $\mu$, and that
%  $i_{\tau}({\bf 0}^\lambda)$ takes values $i_{\tau}({\bf 0}^\mu)$ and
%  $i_{\tau}({\bf 0}^\mu)+ \nu_{\tau}({\bf 0}^\mu)$
%  on the  two deleted half-neighborhoods  of $\mu$,  respectively.
 Then at least one of the following assertions holds:
\begin{enumerate}
\item[\rm (i)]
The problem (\ref{e:Delay1}) with $\lambda=\mu$ admits a sequence of solutions, $x^{\mu,j}\ne {\bf 0}^\mu$ ($j=1,2,\cdots$)
such that  $x^{\mu,j}|_{[0,\tau]}$ converges to the zero  in $C^1([0,\tau];\R^{n})$.
\item[\rm (iv)] There exist left and right  neighborhoods $\Lambda^-$ and $\Lambda^+$ of $\mu$ within $\Lambda$
and nonnegative integers $n^+$ and $n^-$ satisfying $n^++n^-\ge \nu_{T_{m,n}^{-1}}^{A_{m,n}}({\bf 0}^\mu)$,
such that for $\lambda\in\Lambda^-\setminus\{\mu\}$ (resp. $\lambda\in\Lambda^+\setminus\{\mu\}$),
 problem (\ref{e:Delay1}) (with parameter $\lambda$) possesses at least $n^-$ (resp. $n^+$) distinct pairs of nontrivial solutions,
$\{v^{\lambda,i}, -v^{\lambda,i}\}$ ($i = 1, \dots, n^-$ resp. $n^+$),
whose restrictions to $[0,\tau]$  converge to zero in $C^1([0,\tau];\R^{n})$  as $\lambda\to\mu$.
\end{enumerate}
Correspondingly,  for $m\in 2\mathbb{N}+1$, the following replacements should be made to the statements above:
 $i_{T_{m,n}^{-1}}^{A_{m,n}}$ and $\nu_{T_{m,n}^{-1}}^{A_{m,n}}$ are replaced by
$i_{\tilde{B}_{m-1,n}^{-1}}^{A_{m-1,n}}$ and $\nu_{\tilde{B}_{m-1,n}^{-1}}^{A_{m-1,n}}$ in (\ref{e:Delay288*}), respectively;
\end{theorem}
\begin{proof}[\bf Proof]
As in the proof of Theorem~\ref{th:bif-per2Delay},
we only prove the case $m\in 2\mathbb{N}$.
From the assumptions and from (\ref{e:Delay27}), (\ref{e:Delay277*}), and (\ref{e:Delay27+Dong}), we derive that:
$\nu_{\tau, M_{2,n}}(\Gamma_{n,2,{\bf 0}}^\mu)\ne 0$,  $\nu_{\tau, M_{2,n}}(\Gamma_{n,2,{\bf 0}}^\lambda)=0$
 for all $\lambda\in\Lambda\setminus\{\mu\}$ sufficiently close to $\mu$, and
  $i_{\tau, M_{2,n}}(\Gamma_{n,2,{\bf 0}}^\lambda)$ takes, respectively, values
  $i_{\tau, M_{2,n}}(\Gamma_{n,2,{\bf 0}}^\mu)$ and
  $i_{\tau, M_{2,n}}(\Gamma_{n,2,{\bf 0}}^\mu)+ \nu_{\tau, M_{2,n}}(\Gamma_{n,2,{\bf 0}}^\mu)$
 as $\lambda\in\Lambda$ varies in  two deleted half neighborhoods  of $\mu$.
Then applying the second part of Theorem~1.8 in \cite{Lu11}
to problem (\ref{e:M-invariantDelay5++}) with
 $(H, u_\lambda, M)=(\check{H}, {\bf 0}, M_{2,n})$,
% $(H, u_\lambda, M)=(\check{H}, u^\lambda, M_{2,n})$
% (where $u^\lambda=\Upsilon(m,n)v^\lambda$
%is given in (\ref{e:M-invariantDelay6++})), 
we obtain the desired conclusions as in the previous proof.
\end{proof}

When $V(\lambda, t, x)$ is affine in $\lambda \in \mathbb{R}$,
the previous three theorems take a concise form without indexes.

\begin{theorem}\label{th:bif-per2Delay++}
 Let the potential $V(\lambda, t, x)$ in the problem (\ref{e:Delay1}) be of the form:
$$
V(\lambda, t, x) = V_0(t, x) + \lambda V_1(t, x), \quad \lambda \in \mathbb{R},
$$
where $V_0, V_1: \mathbb{R} \times \mathbb{R}^n \to \mathbb{R}$ are $C^1$-smooth functions satisfying additional conditions:
\begin{enumerate}
\item[\rm (i)] $V_0(t+\tau, x)=V_0(t,x)$ and ${V}_1(t+\tau, x)={V}_1(t,x)$ for all $(t,x)\in\R\times\mathbb{R}^{n}$.
\item[\rm (ii)] All functions $V_0(t, \cdot), V_1(t, \cdot): \mathbb{R}^n \to \mathbb{R}$ are even and of class $C^2$,
so that $\nabla_2 V_0(t, \mathbf{0}) = \nabla_2 V_1(t, \mathbf{0}) = 0$ for all $t$ (with $\nabla_2$ denoting the gradient with respect to $x\in \mathbb{R}^n$).
Furthermore, for $i=0,1$, the gradient $\nabla_2 V_i(t, x)$ and the Hessian $\nabla^2_2 V_i(t, x)$
 with respect to $x$ depend continuously on $(t, x) \in \mathbb{R} \times \mathbb{R}^n$.
%{\rm (Therefore $\nabla_2V_0(t,{\bf 0})=\nabla_2{V}_1(t,{\bf 0})=0\;\forall t$. Hereafter,
% denote, respectively, )}
%\end{enumerate}
%Let $\bar{v}:\R\to\R^{2n}$ satisfy
%\begin{enumerate}
%\item[\rm (c)] $\dot{v}(t)=J\nabla_zH_0(t, v(t))$ and $v(t+\tau)=Mv(t)\;\forall t$.
\item[\rm (iii)] The Hessian $\nabla^2_2{V}_1(t, {\bf 0})$ is either positive definite for all $t$
     or negative definite for all $t$.
%$\nabla^2_2{V}_1(t, {\bf 0})$ is definite for all $t$, being either positive definite or negative definite.
\end{enumerate}
Denoted the zero solution of the problem
\begin{eqnarray}\label{e:LinearDelay2+}
 \left\{\begin{array}{ll}
 &\dot{x}( t ) =\sum^{m-1}_{i=1} \nabla_2 V_0 (t , x(t-i\tau))+
 \lambda\sum^{m-1}_{i=1}\nabla_2 V_1 (t , x ( t - i\tau )),\\
 & x( t - m \tau ) = - x( t )\;\forall t
   \end{array}\right.
  \end{eqnarray}
% with parameter value $\lambda$
  by ${\bf 0}^\lambda$, and
 the dimension of the solution space of the linear delay problem
 \begin{eqnarray}\label{e:LinearDelay2}
 \left\{\begin{array}{ll}
 &\dot{x}( t ) =\sum^{m-1}_{i=1} \nabla_2^2 V_0 (t , {\bf 0})x(t-i\tau) +
 \lambda\sum^{m-1}_{i=1}\nabla^2_2 V_1 (t , {\bf 0})x ( t - i\tau ),\\
 & x( t - m \tau ) = - x( t )\;\forall t
   \end{array}\right.
  \end{eqnarray}
by $\nu_{\tau}({\bf 0}^\lambda)$. %the claim below (\ref{e:Delay28++}).
Then
\begin{enumerate}
\item[\rm (A)] $\Sigma:=\{\lambda\in\R\,|\, \nu_{\tau}({\bf 0}^\lambda)>0\}$
is a discrete set in $\R$.
\item[\rm (B)] $(\mu, {\bf 0}^\mu)$ with $\mu\in\R$ is a bifurcation point  for the problem (\ref{e:LinearDelay2+})
if and only if $\nu_{\tau}({\bf 0}^\mu)>0$.
\item[\rm (C)] For each $\mu\in\Sigma$,   one of the following assertions holds:
\begin{enumerate}
\item[\rm (C.1)]
The problem (\ref{e:LinearDelay2+})
 with $\lambda=\mu$ has a sequence of solutions, $x^{\mu,j}\ne {\bf 0}^\mu$ ($j=1,2,\cdots$)
such that  $x^{\mu,j}|_{[0,\tau]}$ converges to  zero  in $C^1([0,\tau];\R^{n})$.
\item[\rm (C.2)] There exist left and right  neighborhoods $\Lambda^-$ and $\Lambda^+$ of $\mu$ in $\mathbb{R}$
and nonnegative integers $n^+$ and $n^-$ satisfying $n^++n^-\ge \nu_{\tau}({\bf 0}^\mu)$,
such that for $\lambda\in\Lambda^-\setminus\{\mu\}$ (resp. $\lambda\in\Lambda^+\setminus\{\mu\}$),
the problem (\ref{e:LinearDelay2+}) with parameter value $\lambda$  has at least $n^-$ (resp. $n^+$) distinct pairs of nontrivial solutions,
$\{v^{\lambda,i}, -v^{\lambda,i}\}$ ($i = 1, \dots, n^-$ resp. $n^+$),
whose restrictions to $[0,\tau]$  converge to zero in $C^1([0,\tau];\R^{n})$  as $\lambda\to\mu$.
\end{enumerate}
%the conclusions of Theorem~\ref{th:bif-per2Delay+}
%with $V(\lambda,t, z)=V_0(t,z)+\lambda{V}_1(t,z)$ and $\Lambda=\mathbb{R}$ hold true.
%
% and a small enough $\rho>0$, (\ref{e:case1}) and (\ref{e:case2})
%hold, and therefore the conclusions in the first part of Theorem~\ref{th:bif-per2}
%holds, and the second one of Theorem~\ref{th:bif-per2} is also true provided that
%$\bar{v}=0$ and all $H_0(t,\cdot), \hat{H}(t,\cdot)$ are even.
\end{enumerate}
Moreover, under the above assumptions (i)-(ii), suppose that $m\in 2\mathbb{N}$ and that the function $\bar{x}:\mathbb{R}\to\mathbb{R}^n$ satisfies:
$$
  \dot{\bar{x}}( t ) =\sum^{m-1}_{i=1} \nabla_2 V_0 (t , \bar{x}(t-i\tau)),\quad
 \bar{x}( t - m \tau ) = -\bar{x}( t ),\quad
 \sum^{m-1}_{i=1} \nabla_2 V_1 (t , \bar{x}(t-i\tau))=0
$$
for all $t\in\mathbb{R}$. Furthermore,  the matrix 
$\nabla^2_2 V_1 (t , \bar{x}(t))$
 is either positive definite for all $t$  or negative definite for all $t$.
Denoted the dimension of the solution space of the linear delay problem
 \begin{eqnarray}\label{e:LinearDelay2+C}
 \left\{\begin{array}{ll}
 &\dot{x}( t ) =\sum^{m-1}_{i=1} \nabla_2^2 V_0 (t , \bar{x}(t-i\tau))x(t-i\tau) +
 \lambda\sum^{m-1}_{i=1}\nabla^2_2 V_1 (t ,\bar{x}(t-i\tau))x ( t - i\tau ),\\
 &  x( t - m \tau ) = - x( t )\;\forall t
   \end{array}\right.
  \end{eqnarray}
by $\nu_{\lambda, \tau}(\bar{x})$. Then
\begin{enumerate}
\item[\rm (I)] $\Sigma(\bar{x}):=\{\lambda\in\R\,|\, \nu_{\lambda,\tau}(\bar{x})>0\}$
is a discrete set in $\R$.
\item[\rm (II)] $(\mu, \bar{x})$ with $\mu\in\R$ is a bifurcation point  for the problem (\ref{e:LinearDelay2+})
%\begin{eqnarray}\label{e:LinearDelay2++C}
% \dot{x}( t ) =\sum^{m-1}_{i=1} \big(\nabla_2 V_0 (t , x(t-i\tau))+ \lambda\nabla_2 V_1 (t , x(t-i\tau))\big)\quad\hbox{and}\quad
% x( t - m \tau ) = - x( t )\;\forall t
%   \end{eqnarray}
if and only if $\nu_{\mu,\tau}(\bar{x})>0$.
\item[\rm (III)] For each $\mu\in\Sigma(\bar{x})$,   one of the following assertions holds:
\begin{enumerate}
\item[\rm (III.1)]
The problem (\ref{e:LinearDelay2+}) with $\lambda=\mu$ has a sequence of solutions, $x^{j}\ne \bar{x}$ ($j=1,2,\cdots$)
such that  $x^{j}|_{[0,\tau]}\to \bar{x}|_{[0,\tau]}$ in $C^1([0,\tau];\R^{n})$.

\item[\rm (III.2)]  For every $\lambda\in\Lambda\setminus\{\mu\}$ sufficiently close to $\mu$, there is a  solution $\bar{x}^\lambda\ne \bar{x}$ of the problem
(\ref{e:LinearDelay2+}) with parameter value $\lambda$
such that  $\bar{x}^\lambda|_{[0,\tau]}-\bar{x}|_{[0,\tau]}\to 0$ in $C^1([0,\tau];\R^{n})$  as $\lambda\to \mu$.

\item[\rm (III.3)] For a given neighborhood $\mathcal{W}$ of $\bar{x}|_{[0,\tau]}$ in $C^1([0,\tau];\R^{n})$,
there is a one-sided  neighborhood $\Lambda^0$ of $\mu$ within $\mathbb{R}$ such that
for any $\lambda\in\Lambda^0\setminus\{\mu\}$, the problem (\ref{e:LinearDelay2+})  with parameter value $\lambda$
has at least two distinct solutions $\bar{x}^\lambda\ne \bar{x}$ and $\hat{x}^\lambda\ne \bar{x}$ satisfying
$\bar{x}^\lambda|_{[0,\tau]}\in \mathcal{W}$ and $\hat{x}^\lambda|_{[0,\tau]}\in \mathcal{W}$.
Moreover, if $\nu_{\mu,\tau}(\bar{x})>1$ and the problem (\ref{e:LinearDelay2+}) with parameter value $\lambda$
has only finitely many  solutions whose restrictions to $[0,\tau]$ belong to $\mathcal{W}$, then the above
$\bar{x}^\lambda$ and $\hat{x}^\lambda$  can also be required to satisfy:
\begin{eqnarray*}
&&\int^{\tau}_0\left[\frac{1}{2}(\dot{\bar{v}}^\lambda(t), A_{m,n}^{-1}\bar{v}^\lambda(t))_{\mathbb{R}^{mn}}+
 H_0(t, \bar{v}^\lambda(t))+ \lambda H_1(t, \bar{v}^\lambda(t))\right]dt\nonumber\\
&&\ne \int^{\tau}_0\left[\frac{1}{2}(\dot{\hat{v}}^\lambda(t), A_{m,n}^{-1}\hat{v}^\lambda(t))_{\mathbb{R}^{mn}}+
H_0(t,\hat{v}^\lambda(t))+ \lambda H_1(t,
\hat{v}^\lambda(t))\right]dt
\end{eqnarray*}
for  $H_0$ and $H_1$ in (\ref{e:H_0H_1}), where $\bar{v}^\lambda(t):=(\bar{x}^\lambda_1(t)^T,\cdots, \bar{x}^\lambda_m(t)^T)^T\in ({\R}^{n})^m$
with $\bar{x}^\lambda_ { i } ( t ) = \bar{x}^\lambda( t-(i-1)\tau)$ for $i=1,\cdots,m$,
and $\hat{v}^\lambda(t):=(\hat{x}^\lambda_1(t)^T,\cdots, \hat{x}^\lambda_m(t)^T)^T\in ({\R}^{n})^m$
with $\hat{x}^\lambda_ { i } ( t ) = \hat{x}^\lambda( t-(i-1)\tau)$ for $i=1,\cdots,m$.
\end{enumerate}
  \end{enumerate}
\end{theorem}

\begin{proof}[\bf Proof]
{\it Step 1}(proving the first part for $m\in 2\mathbb{N}$). Let us define  $H_0, {H}_1: \mathbb{R}\times({\R}^{n})^m\to\R$  by
\begin{equation}\label{e:H_0H_1}
H_0 \left(t , x _ { 1 } , \ldots , x _ { m } \right) = \sum^m_{i=1}V_0 \left(t , x _ {i } \right)
\quad\hbox{and}\quad {H}_1 \left(t , x _ { 1 } , \ldots , x _ { m } \right) = \sum^m_{i=1}{V}_1\left(t , x _ {i } \right).
\end{equation}
By (\ref{e:Delay24}), we have either $\nabla^2_2{H}_1(t, {\bf 0})>0$ for all $t$ or $\nabla^2_2{H}_1(t, {\bf 0})<0$ for all $t$.
Let
\begin{eqnarray*}
\check{H}_i(t, z)=H_i(t, {\Upsilon}(m, n)^{-1}z)\quad\forall
(t, z)\in \R\times({\R}^{n})^m,\;i=0,1.
\end{eqnarray*}
Then $\nabla^2_2 {H}_i(t, {\Upsilon}(m, n)^{-1}z)=({\Upsilon}(m, n))^T\nabla^2_2 \check{H}_i(t, z){\Upsilon}(m, n)$
for $i=0,1$. Therefore %We have
$$\nabla^2_2{H}_1(t, {\bf 0})>0\;\forall t\Longleftrightarrow \nabla^2_2\check{H}_1(t, {\bf 0})>0\;\forall t,\quad
\nabla^2_2{H}_1(t, {\bf 0})<0\;\forall t\Longleftrightarrow \nabla^2_2\check{H}_1(t, {\bf 0})<0\;\forall t.
$$
Define  $\check{H}: \mathbb{R}\times\mathbb{R}\times({\R}^{n})^m\to\R$  by $\check{H}(\lambda, t, x)=\check{H}_0(t,x)+\lambda \check{H}_1(t,x)$.
Then the zero map ${\bf 0}:\mathbb{R}\to ({\mathbb{R}}^{n})^m$ satisfies the following problem
\begin{equation}\label{e:LinearDelay2++}
\dot{u}(t)={J}_{mn/2} \nabla_3 \check{H}(\lambda, t, u(t))\quad\hbox{and}\quad
 u(t+\tau)={M}_{m,n}u(t)\;\;\forall t\in \mathbb{R}
\end{equation}
with $M_{m,n}={\Upsilon}(m, n)T_{m,n}^{-1}{\Upsilon}(m, n)^{-1}$, and the dimension of the solution space of the linear problem
\begin{equation}\label{e:LinearDelay2+++}
\dot{u}(t)={J}_{mn/2} \nabla_3^2 \check{H}(\lambda, t, {\bf 0})u(t)\quad\hbox{and}\quad
 u(t+\tau)={M}_{m,n}u(t)\;\;\forall t\in \mathbb{R}
\end{equation}
equals to that of
\begin{equation}\label{e:LinearDelay2++++}
\dot{v}(t)=A_{m,n}\nabla^2_2 H_0(t, {\bf 0})v(t)+ \lambda A_{m,n}\nabla^2_2 H_1(t, {\bf 0})v(t)\quad\hbox{and}\quad v(t+\tau)=T_{m,n}^{-1}v(t)\;\;\forall t\in \mathbb{R},
\end{equation}
and thus equals $\nu_{\tau}({\bf 0}^\lambda)$ (the dimension  of the solution space of the problem (\ref{e:LinearDelay2})
by the claim below (\ref{e:Delay28++})).

Applying the final part of Corollary~1.9 in \cite{Lu11} to $(H,M)=(\check{H}, M_{m,n})$ we obtain the conclusions (A) and
  \begin{enumerate}
\item[\rm (B*)] $(\mu, {\bf 0})$ with $\mu\in\R$ is a bifurcation point  for the problem (\ref{e:LinearDelay2++})
if and only if $\nu_{\tau}({\bf 0}^\mu)>0$;
\item[\rm (C*)] for each $\mu\in\Sigma$,   one of the following assertions holds:
  \begin{enumerate}
\item[\rm (C*.1)] the problem (\ref{e:LinearDelay2++})
 with $\lambda=\mu$ has a sequence of solutions, $u^{\mu,j}\ne {\bf 0}$ ($j=1,2,\cdots$)
such that  $u^{\mu,j}|_{[0,\tau]}$ converges to  zero  in $C^1([0,\tau];(\R^{n})^m)$.
\item[\rm (C*.2)] There exist left and right  neighborhoods $\Lambda^-$ and $\Lambda^+$ of $\mu$ in $\mathbb{R}$
and nonnegative integers $n^+$ and $n^-$ satisfying $n^++n^-\ge \nu_{\tau}({\bf 0}^\mu)$,
such that for $\lambda\in\Lambda^-\setminus\{\mu\}$ (resp. $\lambda\in\Lambda^+\setminus\{\mu\}$),
the problem (\ref{e:LinearDelay2++}) with parameter value $\lambda$  has at least $n^-$ (resp. $n^+$) distinct pairs of nontrivial solutions,
$\{u^{\lambda,i}, -u^{\lambda,i}\}$ ($i=1,\cdots,n^-$ resp. $n^+$),
such that their restrictions to $[0,\tau]$  converge to zero in $C^1([0,\tau];(\R^{n})^m)$  as $\lambda\to\mu$.
\end{enumerate}
  \end{enumerate}
By Proposition~\ref{prop:Equiv1}, conclusions (B*), (C*.1), and (C*.2) are respectively equivalent to the following three claims:
\begin{description}
\item[\rm (B**)] $(\mu, {\bf 0})$ with $\mu\in\R$ is a bifurcation point  for the problem
\begin{equation}\label{e:LinearDelay2+++++}
\dot{v}(t)=A_{m,n}\nabla_2 H_0(t, v(t))+ \lambda A_{m,n}\nabla_2 H_1(t, v(t))\quad\hbox{and}\quad v(t+\tau)=T_{m,n}^{-1}v(t)\;\;\forall t\in \mathbb{R}
\end{equation}
if and only if $\nu_{\tau}({\bf 0}^\mu)>0$;
\item[\rm (C**.1)] the problem (\ref{e:LinearDelay2+++++})
 with $\lambda=\mu$ has a sequence of solutions, $v^{\mu,j}\ne {\bf 0}$, $j=1,2,\cdots$,
such that  $v^{\mu,j}|_{[0,\tau]}$ converges to the zero  in $C^1([0,\tau];(\R^{n})^m)$;
\item[\rm (C**.2)] there exist left and right  neighborhoods $\Lambda^-$ and $\Lambda^+$ of $\mu$ in $\mathbb{R}$
and integers $n^+, n^-\ge 0$, such that $n^++n^-\ge \nu_{\tau}({\bf 0}^\mu)$,
and for $\lambda\in\Lambda^-\setminus\{\mu\}$ (resp. $\lambda\in\Lambda^+\setminus\{\mu\}$),
the problem (\ref{e:LinearDelay2+++++}) with parameter value $\lambda$  has at least $n^-$ (resp. $n^+$) distinct pairs of nontrivial solutions,
$\{v^{\lambda,i}, -v^{\lambda,i}\}$ ($i = 1, \dots, n^-$ resp. $n^+$),
whose restrictions to $[0,\tau]$  converge to zero in $C^1([0,\tau];(\R^{n})^m)$  as $\lambda\to\mu$.
\end{description}
By Proposition~\ref{prop:Liu12}(i), the three claims are respectively equivalent to (B), (C.1), and (C.2).
This completes the proof for this case.

%
%The desired conclusions follow from \cite[Corollary~1.12]{Lu11} immediately.

{\it Step 2}(proving the first part for $m\in 2\mathbb{N}+1$). Define $H_0, {H}_1: \mathbb{R}\times({\R}^{n})^{m-1}\to\R$  by
\begin{eqnarray*}
&&{H}_0\left(t , x _ { 1 } , \ldots , x _ { m-1 } \right) = V_0 \left(t , x _ { 1 } \right) + \cdots + V_0\left(t , x _ { m -1} \right)+V_0 \left(t , -x_1+x_2-\cdots+ x _ { m-1} \right),\\
&&{H}_1\left(t , x _ { 1 } , \ldots , x _ { m-1 } \right) = {V}_1 \left(t , x _ { 1 } \right) + \cdots + {V}_1 \left(t , x _ { m -1} \right)+
{V}_1 \left(t , -x_1+x_2-\cdots+ x _ { m-1} \right).
\end{eqnarray*}
It follows from (\ref{e:Delay25}) that
\begin{eqnarray*}
% {\small
\nabla^2_2{H}_1\left(t , {\bf 0}\right) &=&
\left( \begin{array} { c c c c }
\nabla^2_2{V}_1 \left(t, 0 \right) & 0 & \cdots & 0 \\ 0 &
\nabla^2_2{V}_1\left(t, 0\right) & \cdots & 0 \\ \vdots & \vdots & \vdots & \vdots \\ 0 & 0 & \cdots &
\nabla^2_2{V}_1 \left(t, 0\right) \end{array} \right)\\
%\end{eqnarray*}
%\begin{equation*}
&&+\left(-I_{n}\;I_{n}\;-I_{n}\;\cdots\;I_{n}\right)^T \cdot \nabla^2_2{V}_1
\left(t, 0\right) \cdot \left(-I_{n}\;I_{n}\;-I_{n}\;\cdots\;I_{n}\right)
\end{eqnarray*}
If $\nabla^2_2{V}_1(t, {\bf 0})>0\;\forall t$ (resp. $\nabla^2_2{V}_1(t, {\bf 0})<0\;\forall t$),
then the first term is positive (resp. negative) definite, and the second term is
semi-positive (resp. semi-negative) definite,
and thus $\nabla^2_2{H}_1(t, {\bf 0})>0\;\forall t$ (resp.  $\nabla^2_2{H}_1(t, {\bf 0})<0\;\forall t$).
Next, let us define $M_{m,n}={\Upsilon}(m-1, n)\tilde{B}_{m-1,n}^{-1}{\Upsilon}(m-1, n)^{-1}$ and
\begin{eqnarray*}
\check{H}_i(\lambda,t, z)=\tilde{H}_i(\lambda,t, {\Upsilon}(m-1, n)^{-1}z)\quad\forall
(\lambda,t, z)\in \Lambda\times\R\times({\R}^{n})^{m-1},\quad i=0,1.
\end{eqnarray*}
The remainder of the proof for this case can be completed using arguments analogous to those in Step 1,
with several appropriate modifications.

{\it Step 3}(proving the ``Moreover" part). %(Proof of ``Moreover'')
For ${H}_i$ and $\check{H}_i$ in Step 1, since
$$
\nabla_2H_i\left(t , x_ { 1 }, \ldots , x_ { m }\right) =\left(\nabla_2V_i(t, x_1)^T,\cdots, \nabla_2V_i(t, x_m)^T\right)^T
$$
and
$$
{\small
\nabla^2_2H_i\left(t , x_ { 1 }, \ldots , x_ { m }\right) =
\left( \begin{array} { c c c c }
\nabla^2_2V_i\left(t, x_1 \right) & 0 & \cdots & 0 \\ 0 &
\nabla^2_2V_i\left(t, x_{ 2 } \right) & \cdots & 0 \\ \vdots & \vdots & \vdots & \vdots \\ 0 & 0 & \cdots &
\nabla^2_2V_i\left(t, x_ { m }\right) \end{array} \right)}
$$
for $i=0,1$, by the assumptions and  Proposition~\ref{prop:Liu12}(i),
$\mathbb{R}\ni t\mapsto \bar{v}(t):=(\bar{x}_1(t)^T,\cdots, \bar{x}_m(t)^T)^T\in ({\R}^{n})^m$,
where $\bar{x}_ { i } ( t ) = \bar{x}( t-(i-1)\tau)$ for $i=1,\cdots,m$, satisfies %the following problem
\begin{eqnarray}\label{e:LinearDelay2+A}
&&  \dot{\bar{v}}(t)=A_{m,n}\nabla_2 H_0(t, \bar{v}(t))\quad\hbox{and}\quad \bar{v}(t+\tau)=T_{m,n}^{-1}\bar{v}(t)\;\;\forall t\in \R,\\
 &&\nabla_2 H_1 (t , \bar{v}(t))=0\;\forall t. \label{e:LinearDelay2+B}
   \end{eqnarray}
   We claim that
   $\nabla^2_2 H_1 (t , \bar{v}(t))$  is either positive definite for all $t$
  or negative definite for all $t$.
In fact, suppose that $\nabla^2_2 V_1 (t , \bar{x}(t))$
 is positive definite for all $t\in\mathbb{R}$.  
Since ${V}_1(t, x)$ is $\tau$-periodic in $t$, 
$\nabla^2_2{V}_1(t, x)$ is also $\tau$-periodic in $t$. 
By this periodicity, we have
 $\nabla^2_2 V_1 (t , \bar{x}(t-i\tau))=
\nabla^2_2 V_1 (t-i\tau, \bar{x}(t-i\tau))$, 
which is positive definite for all $t\in\mathbb{R}$
and $i=1,\cdots,m-1$. It follows  that
$\nabla^2_2 H_1 (t , \bar{v}(t))$
is positive definite for all $t\in\mathbb{R}$.
The case where $\nabla^2_2 V_1 (t , \bar{x}(t))$
 is negative definite for all $t\in\mathbb{R}$ can be proved similarly.
As in Step 1, let
$\check{H}_i(t, z)=H_i(t, {\Upsilon}(m, n)^{-1}z)$ for all
$(t, z)\in \R\times({\R}^{n})^m,\;i=0,1$.
Then $\bar{u}(t):={\Upsilon}(m, n)\bar{v}(t)$  satisfies
\begin{eqnarray}\label{e:LinearDelay2+AA}
&&  \dot{\bar{u}}(t)={J}_{mn/2}\nabla_2 \check{H}_0(t, \bar{u}(t))\quad\hbox{and}\quad \bar{u}(t+\tau)=
{M}_{m,n}\bar{u}(t)\;\;\forall t\in \R,\\
 &&\nabla_2 \check{H}_1 (t , \bar{u}(t))=0\;\forall t \label{e:LinearDelay2+BB}
   \end{eqnarray}
  and $\nabla^2_2 \check{H}_1 (t , \bar{u}(t))$  is either positive definite for all $t$
  or negative definite for all $t$.
Define  $H, \check{H}: \mathbb{R}\times\mathbb{R}\times({\R}^{n})^m\to\R$  by ${H}(\lambda, t, x)={H}_0(t,x)+\lambda {H}_1(t,x)$ and $\check{H}(\lambda, t, x)=\check{H}_0(t,x)+\lambda \check{H}_1(t,x)$. Clearly,
(\ref{e:LinearDelay2+A})-(\ref{e:LinearDelay2+B}) and (\ref{e:LinearDelay2+AA})-(\ref{e:LinearDelay2+BB})
imply
\begin{eqnarray}\label{e:LinearDelay2+AAA}
&&  \dot{\bar{v}}(t)=A_{m,n}\nabla_3 H(\lambda, t, \bar{v}(t))\quad\hbox{and}\quad \bar{v}(t+\tau)=T_{m,n}^{-1}\bar{v}(t)\;\;\forall t\in \R,\\
&&  \dot{\bar{u}}(t)={J}_{mn/2}\nabla_3 \check{H}(\lambda, t, \bar{u}(t))\quad\hbox{and}\quad \bar{u}(t+\tau)=
{M}_{m,n}\bar{u}(t)\;\;\forall t\in \R,
  \label{e:LinearDelay2+BBB}
   \end{eqnarray}
respectively. Since the dimension of the solution space of the problem (\ref{e:LinearDelay2+C})
is $\nu_{\lambda,\tau}(\bar{x})$,  the same holds for the solution spaces of the following systems:
\begin{eqnarray}\label{e:LinearDelay2+AAAA}
\dot{{v}}(t)=A_{m,n}\nabla^2_3 H(\lambda, t, \bar{v}(t))v(t)\quad\hbox{and}\quad {v}(t+\tau)=T_{m,n}^{-1}{v}(t)\;\;\forall t\in \R
   \end{eqnarray}
and
\begin{eqnarray}\label{e:LinearDelay2+BBBB}
\dot{{u}}(t)={J}_{mn/2}\nabla^2_3 \check{H}(\lambda, t, \bar{u}(t))u(t)\quad\hbox{and}\quad {u}(t+\tau)=
{M}_{m,n}{u}(t)\;\;\forall t\in \R.
   \end{eqnarray}
Applying  \cite[Corollary~1.9]{Lu11} to $(H, M)=(\check{H}, M_{m,n})$ we obtain the conclusions (I) and
\begin{enumerate}
\item[\rm (II*)] $(\mu, \bar{u})$ with $\mu\in\R$ is a bifurcation point  for
\begin{eqnarray}\label{e:LinearDelay2+AAA*}
 \dot{{u}}(t)={J}_{mn/2}\nabla_3 \check{H}(\lambda, t, {u}(t))\quad\hbox{and}\quad {u}(t+\tau)=
{M}_{m,n}{u}(t)\;\;\forall t\in \R
   \end{eqnarray}
with $\check{H}(\lambda, t, x)=\check{H}_0(t,x)+\lambda \check{H}_1(t,x)$ as above
if and only if $\nu_{\lambda,\tau}(\bar{x})>0$.
\item[\rm (III*)] For each $\mu\in\Sigma(\bar{x})$,   one of the following assertions holds:
\begin{enumerate}
\item[\rm (III.1*)]
The problem (\ref{e:LinearDelay2+AAA*}) with $\lambda=\mu$ has a sequence of solutions, $u^{j}\ne \bar{u}$ ($j=1,2,\cdots$)
such that  $u^{j}|_{[0,\tau]}\to \bar{u}|_{[0,\tau]}$ in $C^1([0,\tau];\R^{n})$.

\item[\rm (III.2*)]  For every $\lambda\in\Lambda\setminus\{\mu\}$ sufficiently close to $\mu$, there is a  solution $\bar{u}^\lambda\ne \bar{u}$ of the problem (\ref{e:LinearDelay2+AAA*}) with parameter value $\lambda$
such that  $\bar{u}^\lambda|_{[0,\tau]}-\bar{u}|_{[0,\tau]}\to 0$ in $C^1([0,\tau];\R^{n})$  as $\lambda\to \mu$.

\item[\rm (III.3*)] For a given neighborhood $\mathcal{W}^\ast$ of $\bar{u}|_{[0,\tau]}$ in $C^1([0,\tau];(\R^{n})^m)$,
there is a one-sided  neighborhood $\Lambda^0$ of $\mu$ such that
for any $\lambda\in\Lambda^0\setminus\{\mu\}$, the problem (\ref{e:LinearDelay2+AAA*})  with parameter value $\lambda$
has at least two distinct solutions $\bar{u}^\lambda\ne \bar{u}$ and $\hat{u}^\lambda\ne \bar{u}$ satisfying
$\bar{u}^\lambda|_{[0,\tau]}\in \mathcal{W}^\ast$ and $\hat{u}^\lambda|_{[0,\tau]}\in \mathcal{W}^\ast$.
Moreover, if $\nu_{\mu,\tau}(\bar{x})>1$ and the problem (\ref{e:LinearDelay2+AAA*}) with parameter value $\lambda$
has only finitely many  solutions whose restrictions to $[0,\tau]$ belong to $\mathcal{W}^\ast$, then the above
$\bar{u}^\lambda$ and $\hat{u}^\lambda$  can also be required to satisfy:
\begin{eqnarray*}
&&\int^{\tau}_0\left[\frac{1}{2}( J_{mn/2}\dot{\bar{u}}^\lambda(t),\bar{u}^\lambda(t))_{\mathbb{R}^{mn}}+
 \check{H}_0(t, \bar{u}^\lambda(t))+ \lambda \check{H}_1(t, \bar{u}^\lambda(t))\right]dt\nonumber\\
&&\ne \int^{\tau}_0\left[\frac{1}{2}(J_{mn/2}\dot{\hat{u}}^\lambda(t), \hat{u}^\lambda(t))_{\mathbb{R}^{mn}}+
\check{H}_0(t,\hat{u}^\lambda(t))+ \lambda \check{H}_1(t,
\hat{u}^\lambda(t))\right]dt
\end{eqnarray*}
for  $H_0$ and $H_1$ in (\ref{e:H_0H_1}).
\end{enumerate}
  \end{enumerate}
As in Step 1, it is not hard to translate these into the expected claims in Theorem~\ref{th:bif-per2Delay++}.
\end{proof}

%[\textsf{Alternative bifurcations of Fadell-Rabinowitz type and of Rabinowitz type}]

\section{ Bifurcations of the delay system (\ref{e:Delay2})}\label{sec:delay2}

Let $x^\lambda_ { i } ( t ) = x^\lambda( t-(i-1)\tau)$, $i=1,\cdots,m$. Then for each $\lambda\in\Lambda$,
\begin{equation}\label{e:Delay29@}
v^\lambda(t)=(x^\lambda_1(t)^\top,\cdots, x^\lambda_m(t)^\top)^\top\in ({\mathbb{R}}^{2n})^m
\end{equation}
 satisfies (\ref{e:Delay16}).
Under Assumption~\ref{ass:BasiAss1Delay3},  there holds
\begin{equation}\label{e:Delay29}
{%\small
\nabla^2_3\hat{H}\left(\lambda, t , v^\lambda(t)\right) =
\left( \begin{array} { c c c c }
\nabla^2_3G\left(\lambda,t, x^\lambda_1(t) \right) & 0 & \cdots & 0 \\ 0 &
\nabla^2_3G\left(\lambda,t, x^\lambda_{ 2 }(t) \right) & \cdots & 0 \\ \vdots & \vdots & \vdots & \vdots \\ 0 & 0 & \cdots &
\nabla^2_3G\left(\lambda,t, x^\lambda_ { m }(t) \right) \end{array} \right)}
\end{equation}
for $\hat{H}$ in Proposition~\ref{prop:Liu12}(iii).
Let $\hat{\gamma}_{n, m,x}^\lambda$ be the fundamental matrix solution of the linear system
\begin{equation}\label{e:Delay30}
\dot{Z}(t)=J_{n,m}\nabla^2_3\hat{H}\left(\lambda, t , v^\lambda(t)\right)Z(t)
\end{equation}
with $\hat{\gamma}^\lambda_{n,m,x}(0)=I_{2nm}$.

By the arguments above (\ref{e:M-invariantDelay2}),
(with $J_N = J_{mn}$ and $\mathcal{J} = J_{n,m}$),
there exists a congruence transformation between $J_{n,m}$ and $J_{mn}$, given by an invertible real
matrix $\Gamma(m, 2n)=(A^{-1})^\top$ of order $2mn$ satisfying
$\Gamma(m, 2n)^\top J_{n,m} \Gamma(m, 2n) = J_{mn}$.
%(As done in Section~\ref{app:Matrix-I}
%a precise congruence between  the matrix $J_{ n , m }$ and $J_{mn}$ can be constructed in a programmable way.
%Hereafter, \textsf{we will choose it without special mention}.)
When $m=2$,  we can explicitly find a matrix
 \begin{equation}\label{e:LinearDelay2G+}
\Gamma(2,2n)=\left( \begin{array} { c c c } -J_n & 0 \\ 0 &I _ { 2n }\end{array} \right),
\end{equation}
 which satisfies
$$
 \Gamma(2, 2n)^\top
 J_{n,2}\Gamma(2, 2n)=J_{2n},\quad\text{with}\;
 J_{n,2}:=\left( \begin{array} { c c c } 0& J_n \\ J_n &0\end{array} \right).
 $$
Moreover, the matrix $M_{2,n}$ defined by
\begin{equation}\label{e:LinearDelay2G++}
M_{2,n}:=\Gamma(2, 2n)^\top P_{n,2}^{-1}(\Gamma(2, 2n)^\top)^{-1}%=
%\left( \begin{array} { c c c } J_n & 0 \\ 0 &I _ { 2n }\end{array} \right)
%\left( \begin{array} { c c c } 0 & I _ { 2n} \\ I _ { 2n } & 0 \end{array} \right)
%\left( \begin{array} { c c c } -J_n & 0 \\ 0 &I _ { 2n }\end{array} \right)\\
=\left( \begin{array} { c c c } 0 & J_n \\ -J_n & 0 \end{array} \right)%\in{\rm Sp}(4n)^0
\end{equation}
is an orthogonal symplectic matrix with $(M_{2,n})^2=I_{4n}$.
%Applying Theorem 1.5 in \cite{Lu11} to $(H, v_\lambda, M)=(\check{H}, u^\lambda, M_{2,n})$
%in the problem (\ref{e:M-invariantDelay5+++})
%with $m=2$ (where $u^\lambda=\Gamma(2,n)v^\lambda$ and $v^\lambda$ is given by (\ref{e:Delay29@})),
%  as argued in the proof of Theorem~\ref{th:bif-per1Delay1}
% we obtain the following result.

By Proposition~\ref{prop:Equiv},
$v:\mathbb{R}\to ({\mathbb{R}}^{2n})^m$ solves  the problem (\ref{e:Delay16})
 if and only if
$u(t):=\Gamma(m, 2n)^\top v(t)$ satisfies the following problem
\begin{equation}\label{e:M-invariantDelay5+++}
\begin{cases}
\dot{u}(t)={J}_{mn} \nabla_3 \check{H}(\lambda, t, u(t))\\
u(t+\tau)={M}_{m,n}u(t)\;\;\forall t\in\mathbb{R},
\end{cases}
\end{equation}
where
\begin{align}\label{e:Delay30.1}
M_{m,n}&=\Gamma(m, 2n)^\top P_{2n,m}^{-1}(\Gamma(m, 2n)^\top)^{-1},\\
\check{H}(\lambda,t, z)&=\hat{H}(\lambda,t, (\Gamma(m, 2n)^\top)^{-1}z)\quad\forall
(\lambda,t, z)\in \Lambda\times\mathbb{R}\times({\mathbb{R}}^{2n})^m,\label{e:Delay30.2}
\end{align}
and in this situation there holds
$$
\int^{\tau}_0\left[\frac{1}{2}(J_{mn}\dot{u}(t), u(t))_{\mathbb{R}^{2mn}}+ \check{H}(\lambda, t, u(t))\right]dt
= \int^{\tau}_0\left[\frac{1}{2}(\dot{v}(t), J_{n,m}^{-1}v(t))_{\mathbb{R}^{2mn}}+
\hat{H}(\lambda, t, v(t))\right]dt.
$$

Note that each $u^\lambda(t):=\Gamma(m, 2n)^\top v^\lambda(t)$ satisfies the problem (\ref{e:M-invariantDelay5+++}) by Proposition~\ref{prop:Equiv}.
Denote by $\Xi_{n, m,x}^\lambda(t)$ the fundamental matrix solution of  the linear system
\begin{equation*}%\label{e:M-invariantDelay7++}
\dot{y}(t)={J}_{mn} \nabla^2_3\check{H}(\lambda, t, u^\lambda(t)) y(t)
\end{equation*}
with $\Xi_{n, m,x}^\lambda(0)=I_{2mn}$. Then
$\Xi_{n, m,x}^\lambda(t)=\Gamma(m,2n)^\top\hat\gamma_{n, m,x}^\lambda(t)(\Gamma(m,2n)^\top)^{-1}$  by Proposition~\ref{prop:Equiv}.

%According to  \cite[Definition~2.4]{Liu12} and \cite[Definition~2.2]{ZhouLZZ22},
% the Maslov-type index $(i_{\tau}^M, \nu_{\tau}^M)$ in (\ref{e:LiuTang-index}) %yields well-defined indexes
% allows us to define the following indices:
% %As before, using  the Maslov-type index $(i_{\tau, M}, \nu_{\tau, M})$ given in (\ref{e:dongIndex}),  we define

 By (\ref{e:NM-index}), we have the $(J_{n,m}, P_{2n,m})$-index of
 $\hat{\gamma}_{n, m,x}^\lambda$
 defined by
\begin{equation}\label{e:Delay31}
\begin{cases}
i_{P_{2n,m}^{-1}}^{J_{n,m}}(\hat{\gamma}_{n, m,x}^\lambda):=i_{\tau}^{M_{m,n}}(\Xi_{n, m,x}^\lambda)=i_{\tau}^{M_{m,n}}(
\Gamma(m,2n)^\top\hat\gamma_{n, m,x}^\lambda(t)(\Gamma(m,2n)^\top)^{-1}),\\
\nu_{P_{2n,m}^{-1}}^{J_{n,m}}(\hat{\gamma}_{n, m,x}^\lambda):=\nu_{\tau}^{M_{m,n}}(\Xi_{n, m,x}^\lambda)=
\nu_{\tau}^{M_{m,n}}(\Gamma(m,2n)^\top\hat\gamma_{n, m,x}^\lambda(t)(\Gamma(m,2n)^\top)^{-1})
\end{cases}
\end{equation}
with $M_{m,n}$ as defined in (\ref{e:Delay30.1}).
For convenience, in what follows we will always write
\begin{equation}\label{e:Delay31*}
i_{P_{2n,m}^{-1}}^{J_{n,m}}(x^\lambda):=i_{P_{2n,m}^{-1}}^{J_{n,m}}(\hat{\gamma}_{n, m,x}^\lambda)\quad\text{and}\quad
\nu_{P_{2n,m}^{-1}}^{J_{n,m}}(x^\lambda):=
\nu_{P_{2n,m}^{-1}}^{J_{n,m}}(\hat{\gamma}_{n, m,x}^\lambda).
\end{equation}
Clearly,
Remark~\ref{rm:Liu12}(iii) implies that $\nu_{P_{2n,m}^{-1}}^{J_{n,m}}(x^\lambda)$ equals the dimension of the solution space
 of the following linearization for the problem (\ref{e:Delay2}) along ${x}^\lambda$,
 \begin{equation}\label{e:LinearDelay2G}
  \begin{cases}
 \dot{x}( t ) = \sum^{m-1}_{i=1}J_n\nabla^2_3 G (\lambda, t , x^\lambda( t - i\tau ) )x ( t - i\tau ) ,\\
    x( t + m \tau ) =  x( t )\;\forall t.
   \end{cases}
  \end{equation}

We now apply Theorem 1.5 in \cite{Lu11} to the problem (\ref{e:M-invariantDelay5+++}). In this application, we set
$(H, v_\lambda, M)=(\check{H}, u^\lambda, M_{m,n})$, with
$u^\lambda=\Gamma(m,2n)v^\lambda$ and $v^\lambda$ as in  (\ref{e:Delay29@})).
Following the same line of argument as in the proof of Theorem~\ref{th:bif-per1Delay1},
 we arrive at the following result.

\begin{theorem}\label{th:bif-per1Delay2}
Let Assumption~\ref{ass:BasiAss1Delay3} be satisfied, and
let $i_{P_{2n,m}^{-1}}^{J_{n,m}}(x^{\lambda})$
 and $\nu_{P_{2n,m}^{-1}}^{J_{n,m}}(x^{\lambda})$
be defined by (\ref{e:Delay31}) and (\ref{e:Delay31*}).  Then the following is true.
\begin{enumerate}
\item[\rm (I)]{\rm (\textsf{Necessary condition}):}
 If $(\mu, x^\mu)$ is a bifurcation point along sequences of the problem (\ref{e:Delay2})
 with respect to the branch $\{(\lambda, x^\lambda)\,|\,\lambda\in\Lambda\}$,
 then $\nu_{P_{2n,m}^{-1}}^{J_{n,m}}(x^{\mu})\ne 0$.

\item[\rm (II)]{\rm (\textsf{Sufficient condition}):}
Let $\Lambda$ be first countable.
Suppose that for some $\mu\in\Lambda$,
there exist two sequences $(\lambda_k^-)$ and
$(\lambda_k^+)$ in  $\Lambda$ such that $\lambda_k^\pm\to\mu$,  and such that
for each $k\in\mathbb{N}$,
$$
[i_{P_{2n,m}^{-1}}^{J_{n,m}}(x^{\lambda_k^-}), i_{P_{2n,m}^{-1}}^{J_{n,m}}(x^{\lambda_k^-})+\nu_{P_{2n,m}^{-1}}^{J_{n,m}}(x^{\lambda_k^-})]
\cap[i_{P_{2n,m}^{-1}}^{J_{n,m}}(x^{\lambda_k^+}), i_{P_{2n,m}^{-1}}^{J_{n,m}}(x^{\lambda_k^+})+
\nu_{P_{2n,m}^{-1}}^{J_{n,m}}(x^{\lambda_k^+})]=\emptyset,
$$
 and moreover, either $\nu_{P_{2n,m}^{-1}}^{J_{n,m}}(x^{\lambda_k^+})=0$ or $\nu_{P_{2n,m}^{-1}}^{J_{n,m}}(x^{\lambda_k^-})=0$.
Let $\hat{\Lambda}:=\{\mu,\lambda^+_k, \lambda^-_k\,|\,k\in\mathbb{N}\}$.
  Then  $(\mu, x^\mu)$ is a bifurcation point  of the problem (\ref{e:Delay2})
  with respect to the branch $\{(\lambda, x^\lambda)\,|\,\lambda\in\hat\Lambda\}$
  (and hence also with respect to $\{(\lambda, x^\lambda)\,|\,\lambda\in\Lambda\}$).

\item[\rm (III)]{\rm (\textsf{Existence for bifurcations}):}
Let $\Lambda$ be a path-connected space. Suppose that there exist two points $\lambda^+, \lambda^- \in \Lambda$
 such that the intervals $[i_{P_{2n,m}^{-1}}^{J_{n,m}}(x^{\lambda^-}), i_{P_{2n,m}^{-1}}^{J_{n,m}}(x^{\lambda^-}) +
\nu_{P_{2n,m}^{-1}}^{J_{n,m}}(x^{\lambda^-})]$ and $[i_{P_{2n,m}^{-1}}^{J_{n,m}}(x^{\lambda^+}), i_{P_{2n,m}^{-1}}^{J_{n,m}}(x^{\lambda^+}) +
\nu_{P_{2n,m}^{-1}}^{J_{n,m}}(x^{\lambda^+})]$  are disjoint, i.e.,
 $$
[i_{P_{2n,m}^{-1}}^{J_{n,m}}(x^{\lambda^-}), i_{P_{2n,m}^{-1}}^{J_{n,m}}(x^{\lambda^-}) +
\nu_{P_{2n,m}^{-1}}^{J_{n,m}}(x^{\lambda^-})] \cap [i_{P_{2n,m}^{-1}}^{J_{n,m}}(x^{\lambda^+}), i_{P_{2n,m}^{-1}}^{J_{n,m}}(x^{\lambda^+}) +
\nu_{P_{2n,m}^{-1}}^{J_{n,m}}(x^{\lambda^+})] = \emptyset,
$$
and at least one of $\nu_{P_{2n,m}^{-1}}^{J_{n,m}}(x^{\lambda^+})$ or $\nu_{P_{2n,m}^{-1}}^{J_{n,m}}(x^{\lambda^-})$ is zero.
Then, for any continuous path $\alpha : [0,1] \to \Lambda$ connecting $\lambda^+$ to $\lambda^-$, there exists a sequence
$(t_k) \subset [0,1]$ converging to some $\bar{t} \in [0,1]$, and corresponding solutions $x^k \neq x^{\alpha(t_k)}$
 (of the problem (\ref{e:Delay2}) with $\lambda = \alpha(t_k)$), such that the sequence $(x^k)$ converges to $x^{\alpha(\bar{t})}$
  in the $C^1$ norm on any compact interval $I \subset \mathbb{R}$ as $k \to \infty$.
Moreover, if $\nu_{P_{2n,m}^{-1}}^{J_{n,m}}(x^{\lambda^+}) = 0$, then $\alpha(\bar{t}) \neq \lambda^+$; and if
$\nu_{P_{2n,m}^{-1}}^{J_{n,m}}(x^{\lambda^-}) = 0$, then $\alpha(\bar{t}) \neq \lambda^-$.
  \end{enumerate}
 \end{theorem}

%Applying to the first part of Theorem~1.8 in \cite{Lu11} to  $(H, v_\lambda, M)=(\check{H}, {\bf 0}, M_{2,n})$ in the problem (\ref{e:M-invariantDelay5+++})
%may yield the following theorem.

Applying the first part of Theorem~1.8 in \cite{Lu11} to the problem (\ref{e:M-invariantDelay5+++}),
with the substitutions  %$(H, v_\lambda, M)=(\check{H}, \mathbf{0}, M_{m,n})$,
$(H, v_\lambda, M)=(\check{H}, u^\lambda, M_{m,n})$ with
$u^\lambda=\Gamma(m,2n)v^\lambda$ and $v^\lambda$ as in  (\ref{e:Delay29@}),
 we may obtain the following theorem.

\begin{theorem}[\textsf{Alternative bifurcations of Rabinowitz type}]\label{th:bif-per2Delay2}
Let Assumption~\ref{ass:BasiAss1Delay3} hold, where $\Lambda$ is a real interval,
and let $i_{P_{2n,m}^{-1}}^{J_{n,m}}(x^\lambda)$ and 
$\nu_{P_{2n,m}^{-1}}^{J_{n,m}}(x^\lambda)$
be defined by (\ref{e:Delay31}) and (\ref{e:Delay31*}).
Suppose that  $\mu$ is an interior point of $\Lambda$ such that
$\nu_{P_{2n,m}^{-1}}^{J_{n,m}}(x^\mu) \ne 0$ and 
$\nu_{P_{2n,m}^{-1}}^{J_{n,m}}(x^\lambda) = 0$ for all $\lambda \in \Lambda \setminus \{\mu\}$
sufficiently close to $\mu$,
and  $i_{P_{2n,m}^{-1}}^{J_{n,m}}(x^\lambda)$ takes the values $i_{P_{2n,m}^{-1}}^{J_{n,m}}(x^\mu)$
and $i_{P_{2n,m}^{-1}}^{J_{n,m}}(x^\mu) + \nu_{P_{2n,m}^{-1}}^{J_{n,m}}(x^\mu)$
as $\lambda$ varies in the two deleted half-neighborhoods of $\mu$.
  Then  one of the following assertions holds:
\begin{enumerate}
\item[\rm (i)]
The problem (\ref{e:Delay2}) with $\lambda=\mu$ has a sequence of solutions, $x^{\mu,j}\ne x^\mu$ ($j=1,2,\cdots$)
such that  $x^{\mu,j}|_{[0,\tau]}\to x^\mu|_{[0,\tau]}$ in $C^1([0,\tau];\R^{2n})$.

\item[\rm (ii)]  For every $\lambda\in\Lambda\setminus\{\mu\}$ sufficiently close to $\mu$,
there is a  solution $\bar{x}^\lambda\ne x^\lambda$ of
the problem (\ref{e:Delay2}) with parameter value $\lambda$ such that  $\bar{x}^\lambda|_{[0,\tau]}-x^\lambda|_{[0,\tau]}\to 0$ in $C^1([0,\tau];\R^{2n})$  as $\lambda\to \mu$.

\item[\rm (iii)]
For a given neighborhood $\mathcal{W}$ of $x^\mu|_{[0,\tau]}$ in $C^1([0,\tau];\R^{2n})$,
there is a one-sided  neighborhood $\Lambda^0$ of $\mu$ such that
for any $\lambda\in\Lambda^0\setminus\{\mu\}$, the problem (\ref{e:Delay2}) with parameter value $\lambda$
has at least two distinct solutions $\bar{x}^\lambda\ne x^\lambda$ and $\hat{x}^\lambda\ne x^\lambda$ satisfying
$\bar{x}^\lambda|_{[0,\tau]}\in \mathcal{W}$ and $\hat{x}^\lambda|_{[0,\tau]}\in \mathcal{W}$.
Moreover, if $\nu_{P_{2n,m}^{-1}}^{J_{n,m}}(x^\mu)>1$ and the problem (\ref{e:Delay2}) with parameter value $\lambda$
has only finitely many  solutions whose restrictions to $[0,\tau]$ belong to $\mathcal{W}$, then the above
$\bar{x}^\lambda$ and $\hat{x}^\lambda$  can also be required to satisfy:
%\begin{eqnarray*}
%&&\int^{\tau}_0\left[\frac{1}{2}(\dot{\bar{v}}^\lambda(t), J_{n,m}^{-1}\bar{v}^\lambda(t))_{\mathbb{R}^{2mn}}+ \hat{H}(\lambda, t,
%\bar{v}^\lambda(t))\right]dt\\
%&&\ne \int^{\tau}_0\left[\frac{1}{2}(\dot{\hat{v}}^\lambda(t), J_{n,m}^{-1}\hat{v}^\lambda(t))_{\mathbb{R}^{2mn}}+ \hat{H}(\lambda, t,
%\hat{v}^\lambda(t))\right]dt
%\end{eqnarray*}
\begin{align}\label{e:LinearDelay2G+++}
&\int^{\tau}_0\left[\frac{1}{2}(\dot{\bar{v}}^\lambda(t), J_{n,m}^{-1}\bar{v}^\lambda(t))_{\mathbb{R}^{2mn}}+ \hat{H}(\lambda, t,
\bar{v}^\lambda(t))\right]dt\nonumber\\
&\ne \int^{\tau}_0\left[\frac{1}{2}(\dot{\hat{v}}^\lambda(t), J_{n,m}^{-1}\hat{v}^\lambda(t))_{\mathbb{R}^{2mn}}+ \hat{H}(\lambda, t,
\hat{v}^\lambda(t))\right]dt
\end{align}
for $\hat{H}$ in (\ref{e:Delay14}), where $\bar{v}^\lambda(t):=(\bar{x}^\lambda_1(t)^\top,\cdots, \bar{x}^\lambda_m(t)^\top)^\top\in ({\R}^{2n})^m$
with $\bar{x}^\lambda_ { i } ( t ) = \bar{x}^\lambda( t-(i-1)\tau)$ ($i=1,\cdots,m$),
and $\hat{v}^\lambda(t):=(\hat{x}^\lambda_1(t)^\top,\cdots, \hat{x}^\lambda_m(t)^\top)^\top\in ({\R}^{2n})^m$
with $\hat{x}^\lambda_ { i } ( t ) = \hat{x}^\lambda( t-(i-1)\tau)$ ($i=1,\cdots,m$).
\end{enumerate}
 \end{theorem}

\begin{remark}\label{rm:bif-per2Delay2}
{\rm When $m=2$, because  $J_{n,2}^{-1}=\left( \begin{array} { c c c } 0& J_n \\ J_n &0\end{array} \right)^{-1}=-J_{n,2}$, we have
$$
(\dot{\bar{v}}^\lambda(t), J_{n,2}^{-1}\bar{v}^\lambda(t))_{\mathbb{R}^{4n}}=
(\dot{\bar{x}}_1^\lambda(t), J_{n}\bar{x}_2^\lambda(t))_{\mathbb{R}^{2n}}+
(\dot{\bar{x}}_2^\lambda(t), J_{n}\bar{x}_1^\lambda(t))_{\mathbb{R}^{2n}}.
$$
Since $\bar{x}^\lambda$ is $2\tau$-periodic, applying integration by parts yields
$$
\int^{\tau}_0(\dot{\bar{v}}^\lambda(t), J_{n,2}^{-1}\bar{v}^\lambda(t))_{\mathbb{R}^{4n}}=
2({\bar{x}}^\lambda(0), J_n\bar{x}^\lambda(\tau))_{\mathbb{R}^{2n}}+
2\int^{\tau}_0({\bar{x}}^\lambda(t), J_n\bar{x}^\lambda(t-\tau))_{\mathbb{R}^{2n}}dt.
$$
Note that $G(\lambda, t, x)$ is $\tau$-periodic in $t$ and $\bar{x}^\lambda(t)$ is $2\tau$-periodic.
We deduce
$$
\int^\tau_0\hat{H}(\lambda, t,\bar{v}^\lambda(t))dt=\int^\tau_{-\tau}G(\lambda, t,
\bar{x}^\lambda(t))dt
$$
since  $\hat{H}(\lambda, t,\bar{v}^\lambda(t))=G(\lambda, t,
\bar{x}^\lambda(t))+ G(\lambda, t, \bar{x}^\lambda(t-\tau))$.
 It follows that
(\ref{e:LinearDelay2G+++}) is equivalent to
\begin{align}\label{e:LinearDelay2G++++}
&(\bar{x}^\lambda(0), J_n\bar{x}^\lambda(\tau))_{\mathbb{R}^{2n}}+
 \int^{\tau}_0(\dot{\bar{x}}^\lambda(t), J_n\bar{x}^\lambda(t-\tau))_{\mathbb{R}^{2n}}dt+\int^{\tau}_{-\tau}G(\lambda,t,
 \bar{x}^\lambda(t))dt\nonumber\\
&\ne (\hat{x}^\lambda(0), J_n\hat{x}^\lambda(\tau))_{\mathbb{R}^{2n}}+
 \int^{\tau}_0(\dot{\hat{x}}^\lambda(t), J_n\hat{x}^\lambda(t-\tau))_{\mathbb{R}^{2n}}dt+\int^{\tau}_{-\tau}G(\lambda,t,
 \hat{x}^\lambda(t))dt.
\end{align}
}
\end{remark}

Similarly,  if all $G(\lambda,t,\cdot)$ are even, then  problem 
(\ref{e:Delay2}) with parameter $\lambda$ has the zero solution ${\mathbf{0}}^\lambda$.
Applying the second part of Theorem~1.8 in \cite{Lu11} to  problem (\ref{e:M-invariantDelay5+++}) via the identifications
$(H,  M)=(\check{H},  M_{m,n})$ yields the following theorem.

\begin{theorem}[\textsf{Alternative bifurcations of  Fadell-Rabinowitz type}]\label{th:bif-per2Delay+2}
Let Assumption~\ref{ass:BasiAss1Delay3} hold, where $\Lambda$ is a real interval.
Suppose also that all $G(\lambda,t,\cdot)$ are even; this implies that
problem (\ref{e:Delay2}) with parameter $\lambda$ has the zero solution ${\mathbf{0}}^\lambda$.  
Let $i_{P_{2n,m}^{-1}}^{J_{n,m}}({\mathbf{0}}^\lambda)$ and $\nu_{P_{2n,m}^{-1}}^{J_{n,m}}({\mathbf{0}}^\lambda)$
be defined by (\ref{e:Delay31}) and (\ref{e:Delay31*}).
Suppose that for an interior point $\mu$ of $\Lambda$,
$\nu_{P_{2n,m}^{-1}}^{J_{n,m}}({\mathbf{0}}^\mu) \ne 0$ and $\nu_{P_{2n,m}^{-1}}^{J_{n,m}}({\mathbf{0}}^\lambda) = 0$ for all $\lambda \in \Lambda \setminus \{\mu\}$
 sufficiently close to $\mu$. Furthermore,
 for $\lambda$  in the two  punctured half-neighborhoods of $\mu$,
  $i_{P_{2n,m}^{-1}}^{J_{n,m}}({\mathbf{0}}^\lambda)$ assumes the values $i_{P_{2n,m}^{-1}}^{J_{n,m}}({\mathbf{0}}^\mu)$
and $i_{P_{2n,m}^{-1}}^{J_{n,m}}({\mathbf{0}}^\mu) + \nu_{P_{2n,m}^{-1}}^{J_{n,m}}({\mathbf{0}}^\mu)$, respectively.
 Then  one of the following assertions holds:
\begin{enumerate}
\item[\rm (A)]
The problem (\ref{e:Delay2}) with $\lambda=\mu$ has a sequence of solutions, $x^{\mu,j}\ne {\mathbf{0}}^\mu$ ($j=1,2,\cdots$)
such that  $x^{\mu,j}|_{[0,\tau]}$ converges to the zero  in $C^1([0,\tau];\mathbb{R}^{2n})$.
\item[\rm (B)] There exist left and right  neighborhoods $\Lambda^-$ and $\Lambda^+$ of $\mu$ within $\Lambda$
and nonnegative integers $n^+$ and $n^-$ satisfying $n^++n^-\ge \nu_{P_{2n,m}^{-1}}^{J_{n,m}}({\mathbf{0}}^\mu)$,
such that for $\lambda\in\Lambda^-\setminus\{\mu\}$ (resp. $\lambda\in\Lambda^+\setminus\{\mu\}$),
the problem (\ref{e:Delay2}) with parameter value $\lambda$  has at least $n^-$ (resp. $n^+$) distinct pairs of nontrivial solutions,
$\{v^{\lambda,i}, -v^{\lambda,i}\}$ ($i = 1, \dots, n^-$ resp. $n^+$),
whose restrictions to $[0,\tau]$  converge to zero in $C^1([0,\tau];\mathbb{R}^{2n})$  as $\lambda\to\mu$.
\end{enumerate}
 \end{theorem}

\begin{theorem}\label{th:bif-per2Delay++G}
Let the potential $G(\lambda, t, x)$ in the problem (\ref{e:Delay2}) be of the form:
$$
G(\lambda, t, x) = G_0(t, x) + \lambda G_1(t, x), \quad \lambda \in \mathbb{R},
$$
where $G_0, G_1: \mathbb{R} \times \mathbb{R}^{2n} \to \mathbb{R}$ are $C^1$-smooth functions satisfying additional conditions:
\begin{enumerate}
\item[\rm (i)] $G_0(t+\tau, x)=G_0(t,x)$ and ${G}_1(t+\tau, x)={G}_1(t, x)$ for all $(t, x)\in\Lambda\times\R\times\mathbb{R}^{2n}$.
\item[\rm (ii)]
 All functions $G_0(t, \cdot), G_1(t, \cdot): \mathbb{R}^{2n} \to \mathbb{R}$ are  of class $C^2$,
 for $i=0,1$, the gradient $\nabla_2 G_i(t, x)$ and the Hessian $\nabla^2_2 G_i(t, x)$ of $G_i(t, x)$
 with respect to $x$ depend continuously on $(t, x) \in \mathbb{R} \times \mathbb{R}^{2n}$.
\end{enumerate}
For the $C^1$ function $\bar{x}:\mathbb{R}\to\mathbb{R}^{2n}$ there holds for all $t\in\mathbb{R}$:
$$
 \dot{\bar{x}}( t ) =\sum^{m-1}_{i=1} J_n\nabla_2 G_0 (t , \bar{x}(t-i\tau)), \quad
 \bar{x}( t + m \tau ) =  \bar{x}( t )\quad\hbox{and}\quad
 \sum^{m-1}_{i=1} \nabla_2 G_1 (t , \bar{x}(t-i\tau))=0.
 $$
 Furthermore,  $\nabla^2_2 G_1 (t , \bar{x}(t))$  is either positive definite for all $t$
  or negative definite for all $t$.
  Denoted the dimension of the solution space of the linear delay problem
 \begin{eqnarray}\label{e:LinearDelay2+CG}
 \left\{\begin{array}{ll}
 \dot{x}( t ) =\sum^{m-1}_{i=1}  J_n\nabla_2^2G_0 (t , \bar{x}(t-i\tau))x(t-i\tau) +
 \lambda\sum^{m-1}_{i=1} J_n\nabla^2_2G_1 (t ,\bar{x}(t-i\tau))x ( t - i\tau ),\\
 \hspace{20mm}  x( t + m \tau ) =  x( t )\;\forall t
   \end{array}\right.
  \end{eqnarray}
by $\nu_{\lambda, \tau}(\bar{x})$. Then
\begin{enumerate}
\item[\rm (I)] $\Sigma(\bar{x}):=\{\lambda\in\R\,|\, \nu_{\lambda,\tau}(\bar{x})>0\}$
is a discrete set in $\R$.
\item[\rm (II)] $(\mu, \bar{x})$ with $\mu\in\R$ is a bifurcation point  for
\begin{eqnarray}\label{e:LinearDelay2+G}
 \left\{\begin{array}{ll}
 \dot{x}( t ) =\sum^{m-1}_{i=1} J_n\nabla_2 G_0 (t , x(t-i\tau))+
 \lambda\sum^{m-1}_{i=1}J_n\nabla_2 G_1 (t , x ( t - i\tau )),\\
 \hspace{20mm}  x( t + m \tau ) =  x( t )\;\forall t
   \end{array}\right.
  \end{eqnarray}
if and only if $\nu_{\mu,\tau}(\bar{x})>0$.
\item[\rm (III)] For each $\mu\in\Sigma(\bar{x})$,   one of the following assertions holds:
\begin{enumerate}
\item[\rm (III.1)]
The problem (\ref{e:LinearDelay2+G}) with $\lambda=\mu$ has a sequence of solutions, $x^{j}\ne \bar{x}$ ($j=1,2,\cdots$)
such that  $x^{j}|_{[0,\tau]}\to \bar{x}|_{[0,\tau]}$ in $C^1([0,\tau];\R^{2n})$.

\item[\rm (III.2)]  For every $\lambda\in\Lambda\setminus\{\mu\}$ sufficiently close to $\mu$,
 there is a  solution $\bar{x}^\lambda\ne \bar{x}$ of
the problem (\ref{e:LinearDelay2+G}) with parameter value $\lambda$
such that  $\bar{x}^\lambda|_{[0,\tau]}-\bar{x}|_{[0,\tau]}\to 0$ in $C^1([0,\tau];\R^{2n})$  as $\lambda\to \mu$.

\item[\rm (III.3)] For a given neighborhood $\mathcal{W}$ of $\bar{x}|_{[0,\tau]}$ in $C^1([0,\tau];\R^{2n})$,
there is a one-sided  neighborhood $\Lambda^0$ of $\mu$ such that
for any $\lambda\in\Lambda^0\setminus\{\mu\}$, the problem (\ref{e:LinearDelay2+G})  with parameter value $\lambda$
has at least two distinct solutions $\bar{x}^\lambda\ne \bar{x}$ and $\hat{x}^\lambda\ne \bar{x}$ satisfying
$\bar{x}^\lambda|_{[0,\tau]}\in \mathcal{W}$ and $\hat{x}^\lambda|_{[0,\tau]}\in \mathcal{W}$.
Moreover, if $\nu_{\mu,\tau}(\bar{x})>1$ and the problem (\ref{e:LinearDelay2+G}) with parameter value $\lambda$
has only finitely many  solutions whose restrictions to $[0,\tau]$ belong to $\mathcal{W}$, then the above
$\bar{x}^\lambda$ and $\hat{x}^\lambda$  can also be required to satisfy:
\begin{eqnarray*}
&&\int^{\tau}_0\left[\frac{1}{2}(\dot{\bar{v}}^\lambda(t), J_{m,n}^{-1}\bar{v}^\lambda(t))_{\mathbb{R}^{2mn}}+
 \hat{H}_0(t, \bar{v}^\lambda(t))+ \lambda \hat{H}_1(t, \bar{v}^\lambda(t))\right]dt\nonumber\\
&&\ne \int^{\tau}_0\left[\frac{1}{2}(\dot{\hat{v}}^\lambda(t), J_{m,n}^{-1}\hat{v}^\lambda(t))_{\mathbb{R}^{2mn}}+
\hat{H}_0(t,\hat{v}^\lambda(t))+ \lambda \hat{H}_1(t,
\hat{v}^\lambda(t))\right]dt
\end{eqnarray*}
for  $\hat{H}_0$ and $\hat{H}_1$ in (\ref{e:H_0H_1G}), where $\bar{v}^\lambda(t):=(\bar{x}^\lambda_1(t)^\top,\cdots, \bar{x}^\lambda_m(t)^\top)^\top\in ({\R}^{2n})^m$
with $\bar{x}^\lambda_ { i } ( t ) = \bar{x}^\lambda( t-(i-1)\tau)$, $i=1,\cdots,m$,
and $\hat{v}^\lambda(t):=(\hat{x}^\lambda_1(t)^\top,\cdots, \hat{x}^\lambda_m(t)^\top)^\top\in ({\R}^{2n})^m$
with $\hat{x}^\lambda_ { i } ( t ) = \hat{x}^\lambda( t-(i-1)\tau)$, $i=1,\cdots,m$,
\end{enumerate}
  \end{enumerate}
Moreover, if the functions $G_0(t,\cdot), {G}_1(t,\cdot):{\R}^{2n}\to\R$ are also even for all $t$,
(therefore $\nabla_2 G_0(t, \mathbf{0}) = \nabla_2 G_1(t, \mathbf{0}) = 0$ for all $t$),
and the Hessian $\nabla^2_2{G}_1(t, {\bf 0})$ is either positive definite for all $t$
 or negative definite for all $t$, then the corresponding conclusions (as in the first part of Theorem~\ref{th:bif-per2Delay++}) hold.
\end{theorem}
\begin{proof}[\bf Proof]
Almost repeating the proof of the second part in Theorem~\ref{th:bif-per2Delay++}, we only outline it.
Defining $\hat{H}_0, \hat{H}_1 \colon \mathbb{R} \times (\mathbb{R}^{2n})^m \to \mathbb{R}$ by
\begin{equation}\label{e:H_0H_1G}
    \hat{H}_0 \left(t, x_1, \ldots, x_m \right) = \sum^m_{i=1} G_0 \left(t, x_i \right) \quad \text{and} \quad \hat{H}_1 \left(t, x_1, \ldots, x_m \right) = \sum^m_{i=1} G_1 \left(t, x_i \right),
\end{equation}
By (\ref{e:Delay29}),  the assumptions and  Proposition~\ref{prop:Liu12}(iii),
$\mathbb{R}\ni t\mapsto \bar{v}(t):=(\bar{x}_1(t)^\top,\cdots, \bar{x}_m(t)^\top)^\top\in ({\R}^{2n})^m$,
where $\bar{x}_ { i } ( t ) = \bar{x}( t-(i-1)\tau)$ for $i=1,\cdots,m$, satisfies
\begin{eqnarray}\label{e:LinearDelay2+AG}
&&  \dot{\bar{v}}(t)=J_{m,n}\nabla_2 \hat{H}_0(t, \bar{v}(t))\quad\hbox{and}\quad \bar{v}(t+\tau)=P_{2n,m}^{-1}\bar{v}(t)\;\;\forall t\in \R,\\
 &&\nabla_2 \hat{H}_1 (t , \bar{v}(t))=0\;\forall t \label{e:LinearDelay2+BG}
   \end{eqnarray}
  and $\nabla^2_2 \hat{H}_1 (t , \bar{v}(t))$  is either positive definite for all $t$
  or negative definite for all $t$, which can be shown as
 in Step~3 of the proof of  Theorem~\ref{th:bif-per2Delay++}.
Define $\check{H}_i(t, z)=\hat{H}_i(t, {\Gamma}(m, 2n)^{-1}z)$ for all
$(t, z)\in \R\times({\R}^{2n})^m$, $i=0,1$, and
 $\check{H}: \mathbb{R}\times\mathbb{R}\times({\R}^{2n})^m\to\R$  by $\check{H}(\lambda, t, x)=\check{H}_0(t,x)+\lambda \check{H}_1(t,x)$.
Then $\bar{u}(t):=\Gamma(m, 2n)^T\bar{v}(t)$ satisfies $\nabla_2 \check{H}_1 (t , \bar{u}(t))=0\;\forall t$
and the following problem
\begin{equation*}
\dot{u}(t)={J}_{mn} \nabla_3 \check{H}(\lambda, t, u(t))\quad\hbox{and}\quad
 u(t+\tau)={M}_{m,n}u(t)\;\;\forall t\in \R,
\end{equation*}
where $M_{m,n}=\Gamma(m, 2n)^\top P_{n,m}^{-1}(\Gamma(m, 2n)^\top)^{-1}$.
Moreover, $\nabla^2_2 \check{H}_1 (t , \bar{u}(t))$  is either positive definite for all $t$
  or negative definite for all $t$.
As in the proof of the second part in Theorem~\ref{th:bif-per2Delay++},
 by applying Corollary~1.9 in \cite{Lu11} to $(H,M)=(\check{H}, M_{m,n})$ we may deduce the desired conclusions in the first part.

The conclusions in the ``Moreover'' part can be derived in the same way as the proof of the first part in Theorem~\ref{th:bif-per2Delay++}.
\end{proof}

\section{ Bifurcations of the delay system (\ref{e:Delay3}) }\label{sec:delay3}

 Under Assumption~\ref{ass:BasiAss1Delay4}, let $m\ge 2$, and let $v^\lambda(t):=(x^\lambda_1(t)^\top,\cdots, x^\lambda_m(t)^\top)^\top$
 with $x^\lambda_ { i } ( t ) = x^\lambda( t-(i-1)\tau)$ ($i=1,\cdots,m$),
 $u^\lambda(t):=(\mathcal { A } _ { m,n})^{-1}\dot{v}^\lambda(t)$ and  $w^\lambda(t):=(v^\lambda(t)^\top, u^\lambda(t)^\top)^\top$.
 Then  for $\mathcal{H}$ in Proposition~\ref{prop:Liu12}(iv), $w^\lambda$ satisfies
 (\ref{e:Delay22}), and there holds
\begin{equation}\label{e:Delay32}
{\small \nabla^2_3\mathcal{H}(\lambda,t, w^\lambda(t))
= \left( \begin{array} { c c }
\left( \begin{array} { c c c c }
\nabla^2_3U \left(\lambda,t, x^\lambda_1(t) \right) & 0 & \cdots & 0 \\ 0 &
\nabla^2_3U \left(\lambda,t, x^\lambda_{ 2 }(t) \right) & \cdots & 0 \\ \vdots & \vdots & \vdots & \vdots \\ 0 & 0 & \cdots &
\nabla^2_3U \left(\lambda,t, x^\lambda_ { m }(t) \right) \end{array} \right)
 & {\mathbf{0}} \\
{\mathbf{0}} & \mathcal{A}_{m,n}
\end{array} \right).}
\end{equation}
Let $\Gamma_{n, m,x}^\lambda$ be the fundamental matrix solution of the linear system
\begin{equation}\label{e:Delay33}
\dot{Z}(t)=J_{mn}\nabla^2_3\mathcal{H}(\lambda,t, w^\lambda(t))Z(t)
\end{equation}
with $\Gamma^\lambda_{n,m,x}(0)=I_{2mn}$. Using  the Maslov-type index $(i_{\tau, M}, \nu_{\tau, M})$ in (\ref{e:dongIndex}) we define
\begin{eqnarray}\label{e:Delay34}
i_{\tau}(x^\lambda):=i_{\tau, \mathcal{P}_{n,m}^{-1}}(\Gamma_{n, m,x}^\lambda)\quad\text{and}\quad
\nu_\tau(x^\lambda):=\nu_{\tau, \mathcal{P}_{n,m}^{-1}}(\Gamma_{n, m,x}^\lambda).
\end{eqnarray}
By Remark~\ref{rm:Liu12}(iv), $\nu_\tau(x^\lambda)$ equals the dimension of the solution space
 of the following linearization for the problem (\ref{e:Delay3}) along ${x}^\lambda$,
\begin{equation}\label{e:LinearDelay3}
\begin{cases}
 \ddot{x}( t ) = -\sum^{m-1}_{i=1}\nabla^2_3 U (\lambda, t , {x}^\lambda( t - i\tau ) ) x ( t - i\tau ),  \\
  x( t + m \tau ) =  x( t )\;\forall t.
   \end{cases}
 \end{equation}
Since $\mathcal{H}(\lambda,t+\tau, \mathcal { P } _ { n , m }^{-1}w)=\mathcal{H}(\lambda,t, w)$
for all $(\lambda, t, w)$, and $\mathcal { P } _ { n , m }$ is an orthogonal symplectic matrix satisfying
$\mathcal { P } _ { n , m } ^ { m } = I _ { 2 mn} $,
as in Sections~\ref{sec:delay1},~\ref{sec:delay2},  applying Theorem~1.5 in \cite{Lu11} to
$(H, u_\lambda, M)=(\mathcal{H}, w^\lambda, \mathcal{P}_{n,m}^{-1})$ in problem (\ref{e:Delay22}),
  we derive from Proposition~\ref{prop:Liu12}(iv):

\begin{theorem}\label{th:bif-per1Delay3}
Let Assumption~\ref{ass:BasiAss1Delay4} be satisfied, and
let $i_{\tau}(x^\lambda)$ and $\nu_{\tau}(x^\lambda)$
be defined by (\ref{e:Delay34}).
Then the following holds.
\begin{enumerate}
\item[\rm (I)]{\rm (\textsf{Necessary condition}):}
 If $(\mu, x^\mu)$ is a bifurcation point along sequences of the problem (\ref{e:Delay3})
 with respect to the branch $\{(\lambda, x^\lambda)\,|\,\lambda\in\Lambda\}$, then $\nu_{\tau}(x^\mu)\ne 0$.

\item[\rm (II)]{\rm (\textsf{Sufficient condition}):}
Let $\Lambda$ be first countable.
For some $\mu\in\Lambda$, suppose that
there exist two sequences $(\lambda_k^-)$ and
$(\lambda_k^+)$ in  $\Lambda$ such that $\lambda_k^\pm\to\mu$,  and such that
for each $k\in\mathbb{N}$,
$$
[i_{\tau}(x^{\lambda_k^-}), i_{\tau}(x^{\lambda_k^-})+\nu_{\tau}(x^{\lambda_k^-})]
\cap[i_{\tau}(x^{\lambda_k^+}), i_{\tau}(x^{\lambda_k^+})+
\nu_{\tau}(x^{\lambda_k^+})]=\emptyset,
$$
 and moreover, either $\nu_{\tau}(x^{\lambda_k^+})=0$ or $\nu_{\tau}(x^{\lambda_k^-})=0$.
Let $\hat{\Lambda}:=\{\mu,\lambda^+_k, \lambda^-_k\,|\,k\in\mathbb{N}\}$.
Then $(\mu, x^\mu)$ is a bifurcation point of the problem (\ref{e:Delay3}) with respect to the branch
$\{(\lambda, x^\lambda) \mid \lambda \in \hat\Lambda\}$; hence, it is also a bifurcation point with respect to the larger branch
$\{(\lambda, x^\lambda) \mid \lambda \in \Lambda\}$.

\item[\rm (III)]{\rm (\textsf{Existence for bifurcations}):}
 Let $\Lambda$ be path-connected. Suppose that there exist two  points $\lambda^+, \lambda^-\in\Lambda$ such that
  $$
[i_{\tau}(x^{\lambda^-}), i_{\tau}(x^{\lambda^-})+\nu_{\tau}(x^{\lambda^-})]
\cap[i_{\tau}(x^{\lambda^+}), i_{\tau}(x^{\lambda^+})+
\nu_{\tau}(x^{\lambda^+})]=\emptyset,
$$
 and either $\nu_{\tau}(x^{\lambda^+})=0$ or $\nu_{\tau}(x^{\lambda^-})=0$.
    Then for any path $\alpha:[0,1]\to\Lambda$ connecting $\lambda^+$ to $\lambda^-$
   there exists a sequence $(t_k)\subset [0, 1]$ converging to some $\bar{t}$
     and solutions $x^k\ne x^{\alpha(t_k)}$ of the problem (\ref{e:Delay3}) with $\lambda=\alpha(t_k)$ for $k=1,2,\cdots$,
  such that $(x^k)$ converges to  $x^{\alpha(\bar{t})}$ in the $C^1$ norm on any compact interval $I\subset\R$  as $k\to\infty$.
   Moreover,  $\alpha(\bar{t})$ is not equal to $\lambda^+$ (resp. $\lambda^-$) if $\nu_{\tau}(x^{\lambda^+})=0$ (resp. $\nu_{\tau}(x^{\lambda^-})=0$).
  \end{enumerate}
  \end{theorem}

Applying the first part of Theorem~1.8 in \cite{Lu11} to
$(H, u_\lambda, M)=(\mathcal{H}, w^\lambda, \mathcal{P}_{n,m}^{-1})$ in the
  problem (\ref{e:Delay22})  immediately yields the following theorem.

\begin{theorem}[\textsf{Alternative bifurcations of Rabinowitz type}]\label{th:bif-per2Delay3}
Let Assumption~\ref{ass:BasiAss1Delay4}  with $\Lambda$ being a real interval be  satisfied,
and let $i_{\tau}(x^\lambda)$ and $\nu_{\tau}(x^\lambda)$
be defined by (\ref{e:Delay34}). For some interior point $\mu$ of $\Lambda$, suppose that
 $\nu_{\tau}(x^\mu)\ne 0$ and  $\nu_{\tau}(x^\lambda)=0$
 for each $\lambda\in\Lambda\setminus\{\mu\}$ sufficiently close to $\mu$, and that
  $i_{\tau}(x^\lambda)$ takes, respectively, values $i_{\tau}(x^\mu)$ and
  $i_{\tau}(x^\mu)+ \nu_{\tau}(x^\mu)$
 as $\lambda\in\Lambda$ varies in  two deleted half neighborhoods  of $\mu$.
  Then  one of the following assertions holds:
\begin{enumerate}
\item[\rm (i)]
The problem (\ref{e:Delay3}) with $\lambda=\mu$ has a sequence of solutions, $x^{\mu,j}\ne x^\mu$ ($j=1,2,\dots$)
such that  $x^{\mu,j}|_{[0,\tau]}\to x^\mu|_{[0,\tau]}$ in $C^2([0,\tau];\mathbb{R}^{n})$.

\item[\rm (ii)]  For every $\lambda\in\Lambda\setminus\{\mu\}$  sufficiently close to $\mu$,
 there is a  solution $\bar{x}^\lambda\ne x^\lambda$ of the problem
(\ref{e:Delay3}) with parameter value $\lambda$
such that  $\bar{x}^\lambda|_{[0,\tau]}-x^\lambda|_{[0,\tau]}\to 0$ in $C^2([0,\tau];\mathbb{R}^{n})$  as $\lambda\to \mu$.

\item[\rm (iii)]
For a given neighborhood $\mathcal{W}$ of $x^\mu|_{[0,\tau]}$ in $C^2([0,\tau];\mathbb{R}^{n})$,
there is a one-sided  neighborhood $\Lambda^0$ of $\mu$ such that
for any $\lambda\in\Lambda^0\setminus\{\mu\}$, the problem (\ref{e:Delay3}) with parameter value $\lambda$
has at least two distinct solutions $\bar{x}^\lambda\ne x^\lambda$ and $\hat{x}^\lambda\ne x^\lambda$ satisfying
$\bar{x}^\lambda|_{[0,\tau]}\in \mathcal{W}$ and $\hat{x}^\lambda|_{[0,\tau]}\in \mathcal{W}$.
Moreover, if $\nu_{\tau}(x^\mu)>1$ and the problem (\ref{e:Delay3}) with parameter value $\lambda$
has only finitely many  solutions whose restrictions to $[0,\tau]$ belong to $\mathcal{W}$, then the above
$\bar{x}^\lambda$ and $\hat{x}^\lambda$  can, in addition, be chosen to satisfy:
\begin{align}\label{e:LinearDelay4}
&\int^{\tau}_0\left[\frac{1}{2}(\dot{\bar{w}}^\lambda(t), J_{mn}^{-1}\bar{w}^\lambda(t))_{\mathbb{R}^{2mn}}+ \mathcal{H}(\lambda, t,
\bar{w}^\lambda(t))\right]dt\nonumber\\
&\ne \int^{\tau}_0\left[\frac{1}{2}(\dot{\hat{w}}^\lambda(t), J_{mn}^{-1}\hat{w}^\lambda(t))_{\mathbb{R}^{2mn}}+ \mathcal{H}(\lambda, t,
\hat{w}^\lambda(t))\right]dt
\end{align}
for $\mathcal{H}$ in (\ref{e:Delay19}), where
\end{enumerate}
\begin{enumerate}
\item[$\bullet$] $\bar{w}^\lambda(t):=(\bar{v}^\lambda(t)^\top, \bar{u}^\lambda(t)^\top)^\top$
with $\bar{v}^\lambda(t):=(\bar{x}^\lambda_1(t)^\top,\cdots, \bar{x}^\lambda_m(t)^\top)^\top\in ({\R}^{n})^m$
and $\bar{u}^\lambda(t)=(\mathcal { A } _ { m })^{-1}\dot{\bar{v}}^\lambda(t)$, where
 $\bar{x}^\lambda_ { i } ( t ) = \bar{x}^\lambda( t-(i-1)\tau)$ for $i=1,\cdots,m$,  and
\item[$\bullet$] $\hat{w}^\lambda(t):=(\hat{v}^\lambda(t)^\top, \hat{u}^\lambda(t)^\top)^\top$ with
 $\hat{v}^\lambda(t):=(\hat{x}^\lambda_1(t)^\top,\cdots, \hat{x}^\lambda_m(t)^\top)^\top\in ({\R}^{n})^m$
 and $\bar{u}^\lambda(t)=(\mathcal { A } _ { m })^{-1}\dot{\bar{v}}^\lambda(t)$, where
 $\hat{x}^\lambda_ { i } ( t ) = \hat{x}^\lambda( t-(i-1)\tau)$ for $i=1,\cdots,m$.
 \end{enumerate}
 \end{theorem}

As in the proof of (\ref{e:Delay26+++}), we can deduce an equivalent version
of (\ref{e:LinearDelay4}) expressed clearly in terms of $\bar{x}(t)$ and $\hat{x}(t)$.
Since
$\bar{w}^\lambda(\tau)=\mathcal { P } _ { n , m }\bar{w}^\lambda(0)=((P^\top_{n,m}\bar{v}(0))^\top,
(P^\top_{n,m}\bar{u}(0))^\top)^\top$, and
%\hat{w}^\lambda(\tau)=((P^\top_{n,m}\hat{v}(0))^\top, (P^\top_{n,m}\hat{u}(0))^\top)^\top,\\
\begin{eqnarray*}
(\dot{\bar{w}}^\lambda(t), J_{mn}^{-1}\bar{w}^\lambda(t))_{\mathbb{R}^{2mn}}
=(\dot{\bar{v}}^\lambda(t), \bar{u}^\lambda(t))_{\mathbb{R}^{mn}}
-(\dot{\bar{u}}^\lambda(t), \bar{v}^\lambda(t))_{\mathbb{R}^{mn}},
\end{eqnarray*}
via integration by parts, we obtain
$$
\frac{1}{2}\int^{\tau}_0(\dot{\bar{w}}^\lambda(t), J_{mn}^{-1}\bar{w}^\lambda(t))_{\mathbb{R}^{2mn}}dt=
\int^{\tau}_0(\bar{u}^\lambda(t), \dot{\bar{v}}^\lambda(t))_{\mathbb{R}^{mn}}dt=
\int^{\tau}_0(\mathcal { A }_{ m,n}^{-1}\dot{\bar{v}}^\lambda(t), \dot{\bar{v}}^\lambda(t))_{\mathbb{R}^{mn}}dt.
$$
Note that $\bar{u}^\lambda(t)=(\mathcal { A } _ { m })^{-1}\dot{\bar{v}}^\lambda(t)$.
(\ref{e:Delay19}) and (\ref{e:Delay20}) lead to
\begin{align*}
\int^{\tau}_0 \mathcal{H}(\lambda, t, \bar{w}^\lambda(t))dt
&=-\frac{1}{2}\int^{\tau}_0(\mathcal { A }_{ m,n}\bar{u}^\lambda(t), {\bar{u}}^\lambda(t))_{\mathbb{R}^{mn}}dt-\sum^m_{i=1}\int^{\tau}_0U(\lambda,t,\bar{x}^\lambda_i(t))dt\\
&=-\frac{1}{2}\int^{\tau}_0(\mathcal { A }_{ m,n}^{-1}\dot{\bar{v}}^\lambda(t), \dot{\bar{v}}^\lambda(t))_{\mathbb{R}^{mn}}dt-\sum^m_{i=1}\int^{\tau}_0U(\lambda,t,\bar{x}^\lambda_i(t))dt.
\end{align*}
Hence, since  $\bar{x}$ is $m\tau$-periodic and $U(\lambda,t,x)$ is $\tau$-periodic in $t$,
\begin{align}\label{e:LinearDelay5}
&\int^{\tau}_0\left[\frac{1}{2}(\dot{\bar{w}}^\lambda(t), J_{mn}^{-1}\bar{w}^\lambda(t))_{\mathbb{R}^{2mn}}+ \mathcal{H}(\lambda, t,
\bar{w}^\lambda(t))\right]dt\nonumber\\
&=\frac{1}{2}\int^{\tau}_0(\mathcal { A }_{ m,n}^{-1}\dot{\bar{v}}^\lambda(t), \dot{\bar{v}}^\lambda(t))_{\mathbb{R}^{mn}}dt-\sum^m_{i=1}\int^{\tau}_0U(\lambda,t,\bar{x}^\lambda_i(t))dt
\nonumber\\
&=\frac{1}{2}\int^{\tau}_0(\mathcal { A }_{ m,n}^{-1}\dot{\bar{v}}^\lambda(t), \dot{\bar{v}}^\lambda(t))_{\mathbb{R}^{mn}}dt-\int^{m\tau}_0U(\lambda,t,\bar{x}^\lambda(t))dt.
\end{align}
Moreover,   $\dot{\bar{v}}^\lambda(t)=(\dot{\bar{x}}^\lambda_1(t)^\top,\cdots, \dot{\bar{x}}^\lambda_m(t)^\top)^\top$.
From (\ref{e:Delay21}) we have
$$
\mathcal { A }_{ m,n}^{-1}\dot{\bar{v}}^\lambda(t)=\left(\dot{\bar{x}}^\lambda_1(t)^\top-\sum^m_{j=1}
\dot{\bar{x}}^\lambda_j(t)^\top,\cdots,
\dot{\bar{x}}^\lambda_m(t)^\top-\sum^m_{j=1}
\dot{\bar{x}}^\lambda_j(t)^\top\right)^\top
$$
and therefore
\begin{align*}
\frac{1}{2}\int^{\tau}_0(\mathcal { A }_{ m,n}^{-1}\dot{\bar{v}}^\lambda(t), \dot{\bar{v}}^\lambda(t))_{\mathbb{R}^{mn}}dt
&=\sum^m_{j=1}\int^{\tau}_0(\dot{\bar{x}}^\lambda_j(t), \dot{\bar{x}}^\lambda_j(t))_{\mathbb{R}^{n}}dt-
\frac{1}{m-1}\sum^m_{i,j=1}\int^{\tau}_0(\dot{\bar{x}}^\lambda_i(t), \dot{\bar{x}}^\lambda_j(t))_{\mathbb{R}^{n}}dt\\
&=\int^{m\tau}_0(\dot{\bar{x}}^\lambda(t), \dot{\bar{x}}^\lambda(t))_{\mathbb{R}^{n}}dt-
\frac{1}{m-1}\sum^m_{i,j=1}\int^{\tau}_0(\dot{\bar{x}}^\lambda_i(t), \dot{\bar{x}}^\lambda_j(t))_{\mathbb{R}^{n}}dt
\end{align*}
because $\bar{x}$ is $m\tau$-periodic. Combining  with
(\ref{e:LinearDelay5}) yields
\begin{align}\label{e:LinearDelay6}
&\int^{\tau}_0\left[\frac{1}{2}(\dot{\bar{w}}^\lambda(t), J_{mn}^{-1}\bar{w}^\lambda(t))_{\mathbb{R}^{2mn}}+ \mathcal{H}(\lambda, t, \bar{w}^\lambda(t))\right]dt=\int^{m\tau}_0(\dot{\bar{x}}^\lambda(t), \dot{\bar{x}}^\lambda(t))_{\mathbb{R}^{n}}dt\nonumber\\
&-\frac{1}{m-1}\sum^m_{i,j=1}\int^{\tau}_0(\dot{\bar{x}}^\lambda_i(t), \dot{\bar{x}}^\lambda_j(t))_{\mathbb{R}^{n}}dt
-\int^{m\tau}_0U(\lambda,t,\bar{x}^\lambda(t))dt.
\end{align}
Then (\ref{e:LinearDelay4}) becomes
\begin{align}\label{e:LinearDelay7}
&\int^{m\tau}_0(\dot{\bar{x}}^\lambda(t), \dot{\bar{x}}^\lambda(t))_{\mathbb{R}^{n}}dt
-\frac{1}{m-1}\sum^m_{i,j=1}\int^{\tau}_0(\dot{\bar{x}}^\lambda_i(t), \dot{\bar{x}}^\lambda_j(t))_{\mathbb{R}^{n}}dt
-\int^{m\tau}_0U(\lambda,t,\bar{x}^\lambda(t))dt\nonumber\\
&\ne\int^{m\tau}_0(\dot{\hat{x}}^\lambda(t), \dot{\hat{x}}^\lambda(t))_{\mathbb{R}^{n}}dt
-\frac{1}{m-1}\sum^m_{i,j=1}\int^{\tau}_0(\dot{\hat{x}}^\lambda_i(t), \dot{\hat{x}}^\lambda_j(t))_{\mathbb{R}^{n}}dt
-\int^{m\tau}_0U(\lambda,t,\hat{x}^\lambda(t))dt.
\end{align}

By (\ref{e:Delay20}), if all $U(\lambda,t,\cdot)$ are even, then so are all $\mathcal{H}(\lambda,t, \cdot)$.
Therefore, applying the second part of Theorem~1.8 in \cite{Lu11} to
  problem (\ref{e:Delay22}) with $(H, u_\lambda, M)=(\mathcal{H}, {\bf 0}, \mathcal{P}_{n,m}^{-1})$,
  we immediately obtain the following theorem.

\begin{theorem}[\textsf{Alternative bifurcations of  Fadell-Rabinowitz type}]\label{th:bif-per2Delay+3}
Let  Assumption~\ref{ass:BasiAss1Delay4}  with $\Lambda$ being a real interval be  satisfied.
Suppose that all $U(\lambda,t,\cdot)$ are even and therefore that
the problem (\ref{e:Delay3}) with parameter value $\lambda$ has the zero solution ${\mathbf{0}}^\lambda$.
For some interior point $\mu$ of $\Lambda$, suppose that
 $\nu_{\tau}({\mathbf{0}}^\mu)\ne 0$ and  $\nu_{\tau}({\mathbf{0}}^\lambda)=0$
 for each $\lambda\in\Lambda\setminus\{\mu\}$ sufficiently close to $\mu$, and that
  $i_{\tau}({\mathbf{0}}^\lambda)$ takes, respectively, values $i_{\tau}({\mathbf{0}}^\mu)$ and
  $i_{\tau}({\mathbf{0}}^\mu)+ \nu_{\tau}({\mathbf{0}}^\mu)$
 as $\lambda\in\Lambda$ varies in
 two deleted half neighborhoods  of $\mu$.
 Then  one of the following assertions holds:
\begin{enumerate}
\item[\rm (I)] The problem (\ref{e:Delay3}) with $\lambda=\mu$ has a sequence of solutions, $x^{\mu,j}\ne {\mathbf{0}}^\mu$ ($j=1,2,\cdots$)
such that  $x^{\mu,j}|_{[0,\tau]}$ converges to the zero  in $C^2([0,\tau];\mathbb{R}^{n})$.
\item[\rm (II)] There exist left and right  neighborhoods $\Lambda^-$ and $\Lambda^+$ of $\mu$ within $\Lambda$
and nonnegative integers $n^+$ and $n^-$ satisfying $n^++n^-\ge \nu_{\tau}({\mathbf{0}}^\mu)$,
such that for $\lambda\in\Lambda^-\setminus\{\mu\}$ (resp. $\lambda\in\Lambda^+\setminus\{\mu\}$),
the problem (\ref{e:Delay3}) with parameter value $\lambda$  has at least $n^-$ (resp. $n^+$) distinct pairs of nontrivial solutions,
$\{v^{\lambda,i}, -v^{\lambda,i}\}$ ($i = 1, \dots, n^-$ resp. $n^+$),
whose restrictions to $[0,\tau]$  converge to zero in $C^2([0,\tau];\mathbb{R}^{n})$  as $\lambda\to\mu$.
\end{enumerate}
 \end{theorem}

\section{Bifurcations of  solutions for problems
	 (\ref{e:crm1}) and  (\ref{e:crm1Ad})}\label{sec:delay4}

%\section{Bifurcations of anti-periodic solutions in delay/advanced systems: $\dot{x}(t) = -\nabla_y H(\lambda, t, x(t), x(t\pm\tau))$}\label{sec:Remark}
%
%
%\section{Concluding remarks}\label{sec:Remark}

% Suppose
%  \begin{equation}\label{e:crm00}
%H(\lambda,t+\tau, x, y)=H(\lambda,t, x, y)=H(\lambda,t, y, -x),\quad\forall (\lambda, t, x,y).
%\end{equation}
Clearly, the first equality in Assumption~\ref{ass:Crm1}(ii) shows that
 $H(\lambda,t, x, y)$ is $\tau$-periodic in $t$,
  and the second equality in Assumption~\ref{ass:Crm1}(ii) implies that for each fixed pair
  $(\lambda,t)$, the function $H(\lambda, t, \cdot)$ is  invariant under the transformation
   $(x, y) \mapsto (y, -x)$. The latter transformation
 corresponds to the action of $-J_n$ on $\mathbb{R}^{2n} \equiv \mathbb{R}^n \times \mathbb{R}^n$ defined by
\[
(x_1, \cdots, x_n, y_1, \cdots, y_n)^\top \mapsto -J_n(x_1, \cdots, x_n, y_1, \cdots, y_n)^\top = (y_1, \cdots, y_n, -x_1, \cdots, -x_n)^\top.
\]
 (Since $-J_n=J_n^{-1}$, $H(\lambda,t, J_n^{-1}\cdot z)=H(\lambda,t, (-J_n)\cdot z)=H(\lambda,t,  z)$
 if and only if $H(\lambda,t, J_n\cdot z)=H(\lambda,t,  z)$, that is,
 each $H(\lambda,t, \cdot)$ is $J_n$-invariant.
 Combining this with the first equality in Assumption~\ref{ass:Crm1}(ii),
 we arrive at $H(\lambda, t+\tau, J_nz)=H(\lambda, t+\tau, z)=H(\lambda, t, z)$.)
 %%%%%%%%%%%%%%%%%%%%%%%%%%%%%%%%%%%%%%%%%%%%%%%%%%%%%%%%%%%%%%%%%%%%%%%%%%%%%%%%%%%%%%
% $H(\lambda, t, x, y)=H(\lambda, t, y, -x)$. Combining this with the second equality in Assumption~\ref{ass:Crm1}(ii),
% we arrive at $H(\lambda, t, x, y)=H(\lambda, t, -x, -y)$, i.e.,
%  \textsf{the function $H(\lambda, t, \cdot)$ is even}.)
 %We have assumed that all vectors in the Euclidean spaces are column vectors.)

 Denote by $\nabla_zH(\lambda,t,\cdot)$  the euclidian gradient of
 $H(\lambda,t,\cdot)$ with respect to the $\mathbb{R}^n\times\mathbb{R}^n\equiv\mathbb{R}^{2n}$-variable, and by $\nabla^2_zH(\lambda,t,\cdot)=D_z(\nabla_zH(\lambda,t,\cdot))\in\mathcal{L}_s(\mathbb{R}^{2n})$.
 Let $\nabla_3H(\lambda,t, x, y)$ and $\nabla_4H(\lambda,t, x, y)$ denote
  the euclidian gradient of  $H(\lambda,t, x, y)$ with respect to the $x$-variable and $y$-variable, respectively.
 The second equality in Assumption~\ref{ass:Crm1}(ii) yields
  \begin{equation}\label{e:crm0}
  \nabla_3H(\lambda,t, x, y)=-\nabla_4H(\lambda,t, y, -x),\quad\forall (\lambda, t, x,y).
\end{equation}

\begin{remark}\label{rm:crm1}
{\rm
 There exists a close relation between systems of the form (\ref{e:crm1}) and (\ref{e:crm1Ad}).
If $x(t)$ satisfies (\ref{e:crm1}), then from $x(t+2\tau) = -x(t)$ we have $x(t+\tau) = x(t-\tau+2\tau) = -x(t-\tau)$.
Using (\ref{e:crm0}), we deduce
\begin{align}\label{e:crm1Ad+}
\dot{x}(t)&=-\nabla_4H(\lambda, t,x(t),x(t-\tau))\nonumber\\
&=-\nabla_4H(\lambda, t,x(t), -x(t+\tau))\nonumber\\
&=\nabla_3H(\lambda, t, x(t+\tau), x(t)).
\end{align}

Now define $G:\Lambda\times \mathbb{R}\times{\mathbb{R}}^{n}\times{\mathbb{R}}^{n}\to\mathbb{R}$ by $G(\lambda,t, u, v)=- H(\lambda,t, v, u)$.
Assumption~\ref{ass:Crm1}(ii) implies that $G$ satisfies
\begin{eqnarray}\label{e:crm1Ad++}
G(\lambda, t+\tau, u, v)=G(\lambda,t, u, v)= G(\lambda,t, -v, u),\quad\forall (\lambda, t, u, v).
\end{eqnarray}
Note that $\nabla_3H(\lambda,t, u, v)=- \nabla_4G(\lambda,t, v, u)$. Combining this identity with (\ref{e:crm1}) and (\ref{e:crm1Ad+}), we obtain
\begin{equation}\label{e:crm1Ad+++}
    \dot{x}(t) = -\nabla_4 G(\lambda, t, x(t), x(t + \tau)) \quad \text{and} \quad x(t + 2\tau) = -x(t);
\end{equation}
that is, $x$ satisfies system (\ref{e:crm1Ad}) with the Hamiltonian $H$ replaced by $G$.

Conversely, if $x(t)$ satisfies (\ref{e:crm1Ad+++}), then using (\ref{e:crm1Ad++}), one can verify that $x(t)$ indeed satisfies (\ref{e:crm1}) by reversing the above steps.
}
\end{remark}

Remark~\ref{rm:crm1} shows that the results of systems (\ref{e:crm1}) and (\ref{e:crm1Ad}) can be converted into each other. Thus, it suffices to study either one carefully.

The following relation %between solutions of (\ref{e:crm1}) and (\ref{e:crm2})
was stated in \cite[Section~4]{WaLiu14}. % Qi Wang,  Chungen Liu,

\begin{claim}\label{cl:crm1}
Under Assumption~\ref{ass:Crm1}(i), if  $z(t)=(x(t)^\top, y(t)^\top)^{\top}$ satisfies
\begin{equation}\label{e:crm2}
\dot{z}(t)=J_n\nabla_z H({\lambda,t}, z(t))\quad\text{and}\quad z(t+\tau)=J_nz(t),\;\forall t\in \mathbb{R},
\end{equation}
then  $x(t)$  satisfies (\ref{e:crm1}).
 Conversely, under Assumption~\ref{ass:Crm1}, if $x(t)$ satisfies (\ref{e:crm1}),
 then $z(t)=(x(t)^\top, y(t)^\top)^{\top}$, where $y(t):=x(t-\tau)$,
  satisfies  (\ref{e:crm2}).
%\begin{equation}\label{e:crm2}
%\dot{z}(t)=J_n\nabla_z H({\lambda,t}, z(t))\;\forall t\in \mathbb{R}\quad\text{and}\quad z(t+\tau)=J_nz(t)
%\end{equation}
 \end{claim}
\begin{proof}[\bf Proof]
Let $z(t)=(x(t)^\top, y(t)^\top)^{\top}$ be a solution of (\ref{e:crm2}). Then the second equality in (\ref{e:crm2}) implies
 $x(t)=-y(t-\tau)$, $y(t)=x(t-\tau)$ and thus
$x(t+2\tau)=-y(t+\tau)=-x(t)$. Combining  with the first equality in (\ref{e:crm2}) yields
$$
\dot{x}(t)=-\nabla_4H(\lambda, t,x(t), y(t))=-\nabla_4H(\lambda, t,x(t), x(t-\tau)).
$$
Conversely, if $x(t)$ satisfies (\ref{e:crm1}), define
 $z(t)=(x(t)^\top, y(t)^\top)^{\top}$ with $y(t)=x(t-\tau)$. Then
$$
z(t+\tau)=(x(t+\tau)^\top, x(t)^\top)^{\top}=(-x(t-\tau)^\top, x(t)^\top)^{\top}=(-y(t)^\top, x(t)^\top)^{\top}=J_nz(t)
$$
and
\begin{align*}
\left( \begin{array} { c c }
\dot{x}(t) \\
\dot{y}(t)
\end{array} \right)&=
\left( \begin{array} { c c }
-\nabla_4H(\lambda, t,x(t),x(t-\tau)) \\
\dot{x}(t-\tau)
\end{array} \right)=\left( \begin{array} { c c }
-\nabla_4H(\lambda, t,x(t),x(t-\tau)) \\
-\nabla_4H(\lambda, t-\tau,x(t-\tau),x(t-2\tau))
\end{array} \right)\\
&=\left( \begin{array} { c c }
-\nabla_4H(\lambda, t,x(t), y(t)) \\
-\nabla_4H(\lambda, t, y(t), -x(t))
\end{array} \right)=\left( \begin{array} { c c }
-\nabla_4H(\lambda, t,x(t), y(t)) \\
\nabla_3H(\lambda, t, x(t), y(t))
\end{array} \right)
\end{align*}
due to Assumption~\ref{ass:Crm1}(ii) and (\ref{e:crm0}). %Conversely,
%if $z(t)=(x(t)^T, y(t)^T)^{T}$ is a solution of (\ref{e:crm1}), then $x(t)=-y(t-\tau)$, $y(t)=x(t-\tau)$, and
%\begin{eqnarray*}
%\frac{d}{dt}(y(t)-x(t-\tau))&=&\dot{y}(t)-\dot{x}(t-\tau)\\
%&=&\nabla_3H(\lambda, t, x(t), y(t))+\nabla_4H(\lambda, t-\tau, x(t-\tau), y(t-\tau))\\
%&=&\nabla_3H(\lambda, t, x(t), y(t))+\nabla_4H(\lambda, t, y(t), -x(t))\\
%&=&\nabla_3H(\lambda, t, x(t), y(t))-\nabla_3H(\lambda, t, x(t), y(t))\equiv 0
%\end{eqnarray*}
%due to (\ref{e:crm0}). Combining this with $x(0)=y(\tau)$, we obtain $y(t)=x(t-\tau)$ and thus
%$$
%\dot{x}(t)=-\nabla_4H(\lambda, t, x(t), y(t))=-\nabla_4H(\lambda, t, x(t), x(t-\tau))\quad\text{and}\quad
%x(t+2\tau)=-y(t+\tau)=-x(t).
%$$
\end{proof}

By a similar argument, we have:

\begin{claim}\label{cl:crm2}
Under Assumption~\ref{ass:Crm1}(i), if $ z(t) = (x(t)^\top, y(t)^\top)^\top $ satisfies
\begin{equation}\label{e:crm2Ad}
    \dot{z}(t) = J_n \nabla_z H(\lambda, t, z(t))\quad \text{and} \quad z(t + \tau) = -J_n z(t),\;
     \quad \forall t \in \mathbb{R},
\end{equation}
then $ x(t) $ satisfies (\eqref{e:crm1Ad}).
Conversely, under Assumption~\ref{ass:Crm1}, if $ x(t) $ satisfies \eqref{e:crm1Ad},
then $ z(t) = (x(t)^\top, y(t)^\top)^\top $ with $ y(t) := x(t+ \tau) $, satisfies \eqref{e:crm2Ad}.
\end{claim}

% and $\Lambda\times [0,\tau]\ni (\lambda,t)\mapsto u_\lambda(t)\in\mathbb{R}^{2n}$ is also continuous.
Because of relations in Claim~\ref{cl:crm1} (resp. Claim~\ref{cl:crm2}),
 results in \cite{Lu11} can be applied to system (\ref{e:crm2}) (resp. (\ref{e:crm2Ad})) and thus yield bifurcation conclusions for (\ref{e:crm1}) (resp. (\ref{e:crm1Ad})).
%Owing to this relation, numerous results from \cite{Lu11} can be applied to system (\ref{e:crm2}), thereby yielding bifurcation conclusions for (\ref{e:crm1}).
%This relation enables the application of results from \cite{Lu11} to system (\ref{e:crm2}), which in turn provides bifurcation conclusions for (\ref{e:crm1}). µ«ÄúÔ­À´µÄ¾ä×ÓÒѾ­ºÜºÃ, ²»ÐèÒªÐÞ¸Ä.
%We present only  a few of them.
%Only a select few are presented here.
A select few are provided below.

\begin{assumption}\label{ass:BasiAss1Delay2Rm}
{\rm Under Assumption~\ref{ass:Crm1}, for each $\lambda\in\Lambda$, let $x^\lambda:\mathbb{R}\to\mathbb{R}^{n}$
be a solution of (\ref{e:crm1}).
Moreover, assume that the mapping $(\lambda,t)\mapsto x^\lambda(t)$ from $\Lambda\times\mathbb{R}$ to $\mathbb{R}^{n}$ is continuous.}
\end{assumption}

% ({\it Note}: the map $\mathbb{R}\ni t\mapsto 0\in\mathbb{R}^n$,
%  denoted by ${\mathbf{0}}^\lambda$ without occurrence of confusions,
%   satisfies (\ref{e:Delay1}) because $V(\lambda,x)$ is also even in variable $x\in\mathbb{R}^n$.)}

For each $\lambda\in\Lambda$, let $z^\lambda(t)=(x^\lambda(t)^\top, x^\lambda(t-\tau)^\top)^{\top}$,
and let $\gamma_\lambda:[0,\tau]\to {\rm Sp}(2n,\mathbb{R})$ be the fundamental matrix solution of
\begin{equation}\label{e:crm3}
\dot{Z}(t)=J_n\nabla^2_zH(\lambda,t, z^\lambda(t))Z(t).
\end{equation}
Let $(i_{\tau, J_n}(\gamma_\lambda), \nu_{\tau, J_n}(\gamma_\lambda))$ denote the Maslov-type index of $\gamma_\lambda$ defined by (\ref{e:dongIndex}). Then $\nu_{\tau, J_n}(\gamma_\lambda)$ equals the dimension of the solution space of the linearized system of (\ref{e:crm1}) along $x^\lambda$:
 \begin{equation}\label{e:crm4}
 \begin{cases}
\dot{x}(t)=-\nabla_4\nabla_3H(\lambda, t,x^\lambda(t), x^\lambda(t-\tau))x(t)- \nabla^2_4H(\lambda, t, x^\lambda(t), x^\lambda(t-\tau))x(t-\tau),\\
 x(t + 2\tau) = -x(t), \quad \forall t.
   \end{cases}
  \end{equation}

%To find $4\tau$-periodic solution $x(t)$ with $x(t+2\tau)=-x(t)$, let
% then the solutions of (DDS.2) is corresponding to the solutions of the following Hamiltonian systems
%$$
%\begin{cases}
%\dot{z}(t)=J_{N}\nabla_{z} V(t,z), \
%z(t)=J_{N}^{-1}z(t+\tau),
%\end{cases}
%$$

As in Section~\ref{sec:delay1}, we can apply  Theorem~1.5 in \cite{Lu11} to problem (\ref{e:crm2})
 to obtain:

\begin{theorem}\label{th:crm1}
Under Assumption~\ref{ass:Crm1} and Assumption~\ref{ass:BasiAss1Delay2Rm}, the following holds.
 \begin{enumerate}
\item[\rm (I)]{\rm (\textsf{Necessary condition}):}
 If  $(\mu, x^\mu)$ is a bifurcation point along sequences of  (\ref{e:crm1})
 with respect to the branch $\{(\lambda, x^\lambda)\,|\,\lambda\in\Lambda\}$,
 then $\nu_{\tau, J_n}(\gamma_\mu)\ne 0$.

\item[\rm (II)]{\rm (\textsf{Sufficient condition}):}
Assume that $\Lambda$ is first countable. Let $\mu\in\Lambda$, and suppose there exist two sequences $(\lambda_k^-)$ and $(\lambda_k^+)$ in $\Lambda$ such that $\lambda_k^\pm\to\mu$ and, for each $k\in\mathbb{N}$,
\[
[i_{\tau, J_n}(\gamma_{\lambda_k^-}), i_{\tau,J_n}(\gamma_{\lambda_k^-})+\nu_{\tau,J_n}(\gamma_{\lambda_k^-})]
\cap
[i_{\tau,J_n}(\gamma_{\lambda_k^+}), i_{\tau,J_n}(\gamma_{\lambda_k^+})+\nu_{\tau,J_n}(\gamma_{\lambda_k^+})]=\emptyset,
\]
with either $\nu_{\tau,J_n}(\gamma_{\lambda_k^+})=0$ or $\nu_{\tau,J_n}(\gamma_{\lambda_k^-})=0$.
Set $\hat{\Lambda}:=\{\mu,\lambda^+_k, \lambda^-_k \mid k\in\mathbb{N}\}$. Then $(\mu, x^\mu)$ is a bifurcation point of (\ref{e:crm1}) with respect to the sub-branch $\{(\lambda, x^\lambda) \mid \lambda \in \hat\Lambda\}$; consequently, it is also a bifurcation point with respect to the full branch $\{(\lambda, x^\lambda) \mid \lambda \in \Lambda\}$.

\item[\rm (III)]{\rm (\textsf{Existence for bifurcations}):}
Assume that $\Lambda$ is path-connected, and that there exist two points $\lambda^+, \lambda^-\in\Lambda$ such that
\[
[i_{\tau,J_n}(\gamma_{\lambda^-}), i_{\tau,J_n}(\gamma_{\lambda^-})+\nu_{\tau,J_n}(\gamma_{\lambda^-})]
\cap
[i_{\tau,J_n}(\gamma_{\lambda^+}), i_{\tau,J_n}(\gamma_{\lambda^+})+\nu_{\tau,J_n}(\gamma_{\lambda^+})]=\emptyset,
\]
with either $\nu_{\tau,J_n}(\gamma_{\lambda^+})=0$ or $\nu_{\tau,J_n}(\gamma_{\lambda^-})=0$.
Then for any continuous path $\alpha:[0,1]\to\Lambda$ connecting $\lambda^-$ to $\lambda^+$, there exist a sequence $(t_k)\subset [0, 1]$ converging to some $\bar{t}\in[0,1]$ and solutions $x^k$ of (\ref{e:crm1}) with $\lambda=\alpha(t_k)$ (where $x^k\neq x^{\alpha(t_k)}$ for all $k$) such that $x^k\to x^{\alpha(\bar{t})}$ in $C^1_{\rm loc}(\mathbb{R},\mathbb{R}^n)$ as $k\to\infty$.
Moreover, if $\nu_{\tau,J_n}(\gamma_{\lambda^+})=0$ then $\alpha(\bar{t})\neq\lambda^+$, and if $\nu_{\tau,J_n}(\gamma_{\lambda^-})=0$ then $\alpha(\bar{t})\neq\lambda^-$.
  \end{enumerate}
\end{theorem}

 Similarly, applying the first part of Theorem~1.8 in \cite{Lu11} to system (\ref{e:crm1}) yields:

\begin{theorem}[\textsf{Alternative bifurcations of Rabinowitz type}]\label{th:crm2}
Under Assumption~\ref{ass:Crm1} and Assumption~\ref{ass:BasiAss1Delay2Rm},
let $\mu$ be an interior point of $\Lambda$ such that
 $\nu_{\tau, J_n}(\gamma_\mu)\ne 0$, $\nu_{\tau, J_n}(\gamma_\lambda)=0$
  for all $\lambda$ in some punctured neighborhood of $\mu$ in $\Lambda$,
and the index  $i_{\tau, J_n}(\gamma_\lambda)$ takes, respectively, values $i_{\tau, J_n}(\gamma_\mu)$ and
  $i_{\tau, J_n}(\gamma_\mu)+ \nu_{\tau, J_n}(\gamma_\mu)$
 as $\lambda\in\Lambda$ varies in
 two deleted half neighborhoods  of $\mu$.
 Then at least one of the following assertions holds:
 \begin{enumerate}
\item[\rm (i)] The system (\ref{e:crm1}) with $\lambda=\mu$ has a sequence of
solutions $x^{\mu,j}\ne x^\mu$ ($j=1,2,\cdots$)
such that  $x^{\mu,j}\to x^\mu$ in $C^1_{\rm loc}(\mathbb{R},\mathbb{R}^n)$ as $j\to\infty$.

\item[\rm (ii)] For every $\lambda\in\Lambda\setminus\{\mu\}$ sufficiently close to $\mu$, there exists a solution $\bar{x}^\lambda\ne x^\lambda$ of (\ref{e:crm1}) with parameter $\lambda$
such that $\bar{x}^\lambda\to x^\mu$ in $C^1_{\rm loc}(\mathbb{R},\mathbb{R}^n)$ as $\lambda\to \mu$.

\item[\rm (iii)]
For a given neighborhood $\mathcal{W}$ of $x^\mu|_{[0,\tau]}$ in $C^1([0,\tau];\mathbb{R}^{n})$,
there exists a one-sided neighborhood $\Lambda^0$ of $\mu$ (either $\Lambda^0\subset(\mu-\delta,\mu]$ or $\Lambda^0\subset[\mu,\mu+\delta)$) such that
for any $\lambda\in\Lambda^0\setminus\{\mu\}$, system (\ref{e:crm1}) with parameter $\lambda$
has at least two distinct solutions $\bar{x}^\lambda\neq x^\lambda$ and $\hat{x}^\lambda\neq x^\lambda$ satisfying
$\bar{x}^\lambda|_{[0,\tau]}\in \mathcal{W}$ and $\hat{x}^\lambda|_{[0,\tau]}\in \mathcal{W}$.
Moreover, if $\nu_{\tau, J_n}(\gamma_\mu)>1$ and (\ref{e:crm1}) with parameter $\lambda$
has only finitely many solutions whose restrictions to $[0,\tau]$ belong to $\mathcal{W}$, then $\bar{x}^\lambda$ and $\hat{x}^\lambda$ can be chosen such that
\begin{equation*}
\int^{\tau}_0\left[\frac{1}{2}(J_n\dot{\bar{v}}^\lambda(t), \bar{v}^\lambda(t))_{\mathbb{R}^{2n}}+ H(\lambda, t, \bar{v}^\lambda(t))\right]dt
\neq \int^{\tau}_0\left[\frac{1}{2}(J_n\dot{\hat{v}}^\lambda(t), \hat{v}^\lambda(t))_{\mathbb{R}^{2n}}+ H(\lambda, t, \hat{v}^\lambda(t))\right]dt,
\end{equation*}
where $\bar{v}^\lambda(t):=(\bar{x}^\lambda(t)^\top, \bar{x}^\lambda(t-\tau)^\top)^\top$ and $\hat{v}^\lambda(t):=(\hat{x}^\lambda(t)^\top, \hat{x}^\lambda(t-\tau)^\top)^\top$.
\end{enumerate}
Furthermore, if $H(\lambda, t, z)$ is even in $z$ and $x^\lambda\equiv 0$ for every $\lambda\in\Lambda$,
then at least one of  (i) and (iv) occurs, where (i) is as stated
previously, and
\begin{enumerate}
\item[\rm (iv)] There exist left and right neighborhoods $\Lambda^-$ and $\Lambda^+$ of $\mu$ in $\Lambda$
and integers $n^+, n^-\ge 0$ with $n^++n^-\ge \nu_{\tau,J_n}(\gamma_\mu)$,
such that for $\lambda\in\Lambda^-\setminus\{\mu\}$ (resp. $\lambda\in\Lambda^+\setminus\{\mu\}$),
system (\ref{e:crm1}) with parameter $\lambda$ has at least $n^-$ (resp. $n^+$) distinct pairs of nontrivial solutions $\{v_\lambda^i, -v_\lambda^i\}$, $i=1,\cdots,n^-$ (resp. $i=1,\cdots,n^+$),
which converge to zero in $C^1_{\rm loc}(\mathbb{R},\mathbb{R}^{n})$ as $\lambda\to\mu$.
\end{enumerate}
\end{theorem}
As before, it can be readily computed that
$$
\frac{1}{2}\int^{\tau}_0(J_n\dot{\bar{v}}^\lambda(t), \bar{v}^\lambda(t))_{\mathbb{R}^{2n}}=
-({\bar{x}}^\lambda(\tau), \bar{x}^\lambda(0))_{\mathbb{R}^{n}}+
\int^{\tau}_0(\dot{\bar{x}}^\lambda(t), \bar{x}^\lambda(t-\tau))_{\mathbb{R}^{n}}dt.
$$

Corollaries~1.9 and 1.10 in \cite{Lu11} can also yield corresponding results.

Now we proceed to the study of bifurcations in system (\ref{e:crm1Ad}).

\begin{assumption}\label{ass:BasiAss1Delay2Rm+}
{\rm Under Assumption~\ref{ass:Crm1}, for each $\lambda\in\Lambda$, let $x^\lambda:\R\to\mathbb{R}^{n}$
be a solution of (\ref{e:crm1Ad}). Moreover, assume that the mapping $(\lambda,t)\mapsto x^\lambda(t)$ from $\Lambda\times\mathbb{R}$ to $\mathbb{R}^{n}$ is continuous.
 }
\end{assumption}

For each $\lambda\in\Lambda$, define $w^\lambda(t)=(x^\lambda(t)^\top, x^\lambda(t+\tau)^\top)^{\top}$.
Let $\Upsilon_\lambda:[0,\tau]\to {\rm Sp}(2n,\mathbb{R})$ be the fundamental matrix solution of
\begin{equation}\label{e:crm5}
\dot{Z}(t)=J_n\nabla^2_zH(\lambda,t, w^\lambda(t))Z(t).
\end{equation}
Denote by $(i_{\tau, -J_n}(\Upsilon_\lambda), \nu_{\tau, -J_n}(\Upsilon_\lambda))$ the Maslov-type index of $\Upsilon_\lambda$ as defined in (\ref{e:dongIndex}). Then $\nu_{\tau, -J_n}(\Upsilon_\lambda)$ equals the dimension of the solution space of the linearized system of (\ref{e:crm1Ad}) along the solution $x^\lambda$:
\begin{equation}\label{e:crm6}
\begin{cases}
\dot{x}(t) = -\nabla_4\nabla_3 H(\lambda, t, x^\lambda(t), x^\lambda(t+\tau))x(t)-\nabla^2_4 H(\lambda, t, x^\lambda(t), x^\lambda(t+\tau))x(t+\tau),\\[4pt]
x(t + 2\tau) = -x(t), \quad \forall t\in\mathbb{R}.
\end{cases}
\end{equation}

\begin{theorem}\label{th:crm3}
Theorem~\ref{th:crm1} and Theorem~\ref{th:crm2} hold true
after replacing (\ref{e:crm1}) with (\ref{e:crm1Ad}) and making the following substitutions:
\begin{enumerate}
\item[$\bullet$] $J_n$ with $-J_n$,
\item[$\bullet$] $\gamma_\lambda$ and $\gamma_\mu$ with $\Upsilon_\lambda$ and $\Upsilon_\mu$, respectively,
\item[$\bullet$] $\gamma_{\lambda^\pm}$ and $\gamma_{\lambda^\pm_k}$  with $\Upsilon_{\lambda^\pm}$ and $\Upsilon_{\lambda^\pm_k}$, respectively,
\item[$\bullet$] $x^\lambda(t-\tau)$ with $x^\lambda(t+\tau)$.
\end{enumerate}
\end{theorem}

\section{ Bifurcations of the distributed  delay differential system  (\ref{e:Bifu-distributedDelay1}) }\label{sec:delay5}

%The following two lemmas follow from Lemma~3.1 and Lemma~3.2 in \cite{ZhongWL25}, respectively (with similar arguments applying to their proofs).
%This is contained in   Since our assumptions are weaker,
%we include its proof here for clarity.

The proof of Lemma~3.1 in \cite{ZhongWL25} actually shows:

\begin{lemma}\label{lem:distri1}
Let $\Lambda$ be a  topological space $\Lambda$. Let  $F: \Lambda \times \mathbb{R} \times \mathbb{R}^n \to \mathbb{R}$
be such that  for each fixed $(\lambda,t) \in \Lambda \times \mathbb{R}$, the map
$F(\lambda,t,\cdot): \mathbb{R}^n \to \mathbb{R}$ is  $C^1$.  Furthermore,
 let $G: \Lambda \times  \mathbb{R}^{n} \to \mathbb{R}$ be such that
for each fixed $\lambda\in \Lambda$, the map
$G(\lambda,\cdot): \mathbb{R}^n \to \mathbb{R}$ is even and $C^1$.
%%%%%%%%%%%%%%%%%%%%%%%%%%%%%%%%%%%%%%%%%%%%%%%%%%%%%%%%%%%%%
%Let $\Lambda$ be a topological space. Consider $F: \Lambda \times \mathbb{R} \times \mathbb{R}^n \to \mathbb{R}$
%such that $F(\lambda,t,\cdot)\in C^1(\mathbb{R}^n)$  for all $(\lambda,t) \in \Lambda \times \mathbb{R}$.
% Also consider $G: \Lambda \times  \mathbb{R}^{n} \to \mathbb{R}$ be such that
% for each $\lambda\in \Lambda$, $G(\lambda,\cdot)$  is an even $C^1$ function on $\mathbb{R}^n$.
%Under Assumption~\ref{ass:Distri1Delay1}, but with "$C^2$" replaced by "$C^1$",
Let $H(\lambda,t,z) = 2G(\lambda,x) + F(\lambda, t,y)$
for $z=(x^\top,y^\top)^\top\in\mathbb{R}^{2n}$.
[Write $\nabla_3F(\lambda,t,\cdot)$ and $\nabla_3H(\lambda,t,\cdot)$ for the Euclidean gradients of $F(\lambda,t,\cdot)$
(in $\mathbb{R}^{n}$) and $H(\lambda,t,\cdot)$ (in $\mathbb{R}^{2n}$), respectively,
and $\nabla_2G(\lambda,\cdot)$ for that of $G(\lambda,\cdot)$ (in $\mathbb{R}^{n}$.)]
  If a differentiable curve
$x:\mathbb{R}\to \mathbb{R}^n$ satisfies
\begin{equation}\label{e:distributedDelay1}
\dot{x}(t) = -\nabla_3 F\left(\lambda, t, \int_0^\tau
\nabla_2 G(\lambda, x(t-s)) \mathrm{d}s\right)\quad\hbox{and}\quad
x(t+\tau) = - x(t)\;\forall t,
\end{equation}
setting $y(t):=\int_0^\tau
\nabla_2 G(\lambda, x(t-s)) \mathrm{d}s$, then  $z(t)=(x(t)^\top, y(t)^\top)^\top$
 satisfies
\begin{equation}\label{e:distributedDelay2}
\dot{z}(t) = J_n\nabla_3 H(\lambda, t, z(t))\quad\hbox{and}\quad
z(t+\tau) = - z(t)\;\forall t.
\end{equation}
Conversely, if $z(t)=(x(t)^\top, y(t)^\top)^\top$ satisfies (\ref{e:distributedDelay2}), then
 $y(t)=\int_0^\tau\nabla_2 G(\lambda, x(t-s)) \mathrm{d}s$ and $x(t)$ satisfies
(\ref{e:distributedDelay1}).
\end{lemma}

%For the sake of completeness, we

\begin{proof}[\bf Proof]
Substituting $u=t-s$ in  the definition of $y(t)$ yields, for all \(t\in\mathbb{R}\),
\begin{equation}\label{e:distributedDelay3}
y(t)=\int_0^\tau
\nabla_2 G(\lambda, x(t-s)) \mathrm{d}s=\int_0^t
\nabla_2 G(\lambda, x(u)) \mathrm{d}u- \int_0^{t-\tau}
\nabla_2 G(\lambda, x(u)) \mathrm{d}u.
\end{equation}
Assume  $x(t)$ satisfies (\ref{e:distributedDelay1}). Then we have
\begin{equation}\label{e:distributedDelay4}
y(t+\tau)=\int_0^\tau
\nabla_2 G(\lambda, x(t+\tau-s)) \mathrm{d}s=\int_0^\tau
\nabla_2 G(\lambda, -x(t-s)) \mathrm{d}s
\end{equation}
by the first representation in (\ref{e:distributedDelay3}) and the anti-periodicity $x(t+\tau)=-x(t)$, and
\begin{equation}\label{e:distributedDelay5}
\dot{y}(t)=\nabla_2 G(\lambda, x(t))-\nabla_2 G(\lambda, x(t-\tau))=
\nabla_2 G(\lambda, x(t))-\nabla_2 G(\lambda, -x(t))
\end{equation}
by the second representation in (\ref{e:distributedDelay3}) and  $x(t-\tau)=-x(t)$.
By the oddness of $\nabla_2 G(\lambda, \cdot)$,
(\ref{e:distributedDelay4}) and (\ref{e:distributedDelay5}) lead to
 $y(t+\tau)=-y(t)$ and $\dot{y}(t)=2\nabla_2 G(\lambda, x(t))$, respectively.

Conversely, suppose $z(t)=(x(t)^\top, y(t)^\top)^\top$ satisfies \eqref{e:distributedDelay2}.
Then $\dot{x}(t) = -\nabla_3 F(\lambda, t, y(t))$ and $\dot{y}(t)=2\nabla_2 G(\lambda, x(t))$.
Now set $y^*(t):=\int_{0}^{\tau}\nabla_2 G(\lambda, x(t-s))\,\mathrm{d}s$.
Since $\nabla_2 G(\lambda, \cdot)$ is odd, as above
using \eqref{e:distributedDelay3}, \eqref{e:distributedDelay3} and \eqref{e:distributedDelay5}
we deduce that $y^*$ satisfies $y^*(t+\tau)=-y^*(t)$ and $\dot{y}^*(t)=2\nabla_2 G(\lambda, x(t))$.
Hence $\frac{\mathrm{d}}{\mathrm{d}t}(y-y^*)\equiv 0$, so $y(t)-y^*(t)\equiv \mathbf{c}$ for some constant vector $\mathbf{c}$.
Evaluating at $t+\tau$ yields $\mathbf{c}=y(t+\tau)-y^*(t+\tau)=-y(t)+y^*(t)\equiv -\mathbf{c}$, whence $\mathbf{c}=0$.
Thus $y(t)\equiv y^*(t)=\int_{0}^{\tau}\nabla_2 G(\lambda, x(t-s))\,\mathrm{d}s$, and $x(t)$ satisfies \eqref{e:distributedDelay1}.
\end{proof}

 \begin{assumption}\label{ass:Distri1Delay2}
{\rm Under Assumption~\ref{ass:Distri1Delay1}, suppose also that $F(\lambda,t+\tau,x)=F(\lambda,t,-x)$ for all $(\lambda,t,x)\in
\Lambda \times \mathbb{R} \times \mathbb{R}^n$.
For each $\lambda\in\Lambda$ let $x^\lambda:\R\to\mathbb{R}^{n}$
be a  solution of (\ref{e:Bifu-distributedDelay1}), such that the map $\Lambda\times \R\ni (\lambda,t)\mapsto x^\lambda(t)\in\mathbb{R}^{n}$
 is also continuous.}
\end{assumption}

Under Assumption~\ref{ass:Distri1Delay2}, define
$y_\lambda(t):=\int_0^\tau \nabla_2 G(\lambda, x_\lambda(t-s)) \,\mathrm{d}s$
and $z_\lambda(t)=(x_\lambda(t)^\top, y_\lambda(t)^\top)^\top$.
By Lemma~\ref{lem:distri1}, each $z_\lambda$ satisfies \eqref{e:distributedDelay2}.
The linearized system of \eqref{e:distributedDelay2} at $z_\lambda$ is given by
\begin{equation}\label{e:LinedistributedDelay3}
\left\{
\begin{aligned}
&\dot{z}(t) = J_n\nabla^2_{3}H(\lambda, t, z_\lambda(t))\, z(t), \\
&z(t+\tau) = -z(t), \quad \forall t,
\end{aligned}
\right.
\end{equation}
where $\nabla^2_{3}H$ denotes the Hessian matrix of $H$ with respect to its third variable $z$.
Let ${\gamma}_\lambda$ be the fundamental matrix solution of the linear system in the first line of (\ref{e:LinedistributedDelay3}),
and let $(i_{\tau, -I_{2n}}(\gamma_\lambda), \nu_{\tau, -I_{2n}}(\gamma_\lambda))$ denote the Maslov-type index of $\gamma_\lambda$ defined by (\ref{e:dongIndex}).
%Then $\nu_{\tau, J_n}(\gamma_\lambda)$ equals the dimension of the solution space of the linearized system of (\ref{e:crm1}) along $x^\lambda$:
%\begin{equation}\label{e:Delay30}
%\dot{Z}(t)=J_{n}\nabla^2_3{H}\left(\lambda, t , z_\lambda(t)\right)Z(t)
%\end{equation}
%with ${\gamma}^\lambda(0)=I_{2n}$. Then
%$$
%\nu_{\tau,-I_{2n}}(\gamma^\lambda)=\dim{\rm Ker}(\gamma^\lambda(\tau)+I_{2n})
%$$
%is equal to
The conditions that $G(\lambda, -x)=G(\lambda, x)$ and
$F(\lambda,t+\tau, x)=F(\lambda,t,-x)$ ensure that the Hamiltonian
 $H(\lambda,t,z) = 2G(\lambda,x) + F(\lambda, t,y)$ satisfies
 $$
 H(\lambda, t+\tau, -I_{2n}z)=2G(\lambda, -x)+ F(\lambda, t+\tau, -y)=2G(\lambda, x)+ F(\lambda, t, y)=H(\lambda, t, z).
 $$
 Hence, $H$ satisfies  Assumption~1.2 in \cite{Lu11} with $M=-I_{2n}$.
  Consequently, by applying the method of Section~\ref{sec:delay1} and combining Theorem~1.5 of \cite{Lu11} with Lemma~\ref{lem:distri1},
   we obtain the following result:
% following the approach of Section~\ref{sec:delay1} and combining Theorem~1.5 in \cite{Lu11} with Lemma~\ref{lem:distri1}, we derive:

\begin{theorem}\label{th:Distri1}
Under Assumption~\ref{ass:Distri1Delay2}, the following holds.
 \begin{enumerate}
\item[\rm (I)]{\rm (\textsf{Necessary condition}):}
 If  $(\mu, x^\mu)$ is a bifurcation point along sequences of  (\ref{e:Bifu-distributedDelay1})
 with respect to the branch $\{(\lambda, x^\lambda)\,|\,\lambda\in\Lambda\}$,
 then $\nu_{\tau, -I_{2n}}(\gamma_\mu)\ne 0$.

\item[\rm (II)]{\rm (\textsf{Sufficient condition}):}
Assume that $\Lambda$ is first countable. Let $\mu\in\Lambda$, and suppose there exist two sequences $(\lambda_k^-)$ and $(\lambda_k^+)$ in $\Lambda$ such that $\lambda_k^\pm\to\mu$ and, for each $k\in\mathbb{N}$,
\[
[i_{\tau, -I_{2n}}(\gamma_{\lambda_k^-}), i_{\tau,-I_{2n}}(\gamma_{\lambda_k^-})+\nu_{\tau,-I_{2n}}(\gamma_{\lambda_k^-})]
\cap
[i_{\tau,-I_{2n}}(\gamma_{\lambda_k^+}), i_{\tau,-I_{2n}}(\gamma_{\lambda_k^+})+\nu_{\tau,-I_{2n}}(\gamma_{\lambda_k^+})]=\emptyset,
\]
with either $\nu_{\tau,-I_{2n}}(\gamma_{\lambda_k^+})=0$ or $\nu_{\tau,-I_{2n}}(\gamma_{\lambda_k^-})=0$.
Set $\hat{\Lambda}:=\{\mu,\lambda^+_k, \lambda^-_k \mid k\in\mathbb{N}\}$. Then $(\mu, x^\mu)$ is a bifurcation point of
(\ref{e:Bifu-distributedDelay1}) with respect to the sub-branch $\{(\lambda, x^\lambda) \mid \lambda \in \hat\Lambda\}$; consequently, it is also a bifurcation point with respect to the full branch $\{(\lambda, x^\lambda) \mid \lambda \in \Lambda\}$.

\item[\rm (III)]{\rm (\textsf{Existence for bifurcations}):}
Assume that $\Lambda$ is path-connected, and that there exist two points $\lambda^+, \lambda^-\in\Lambda$ such that
\[
[i_{\tau,-I_{2n}}(\gamma_{\lambda^-}), i_{\tau,-I_{2n}}(\gamma_{\lambda^-})+\nu_{\tau,-I_{2n}}(\gamma_{\lambda^-})]
\cap
[i_{\tau,-I_{2n}}(\gamma_{\lambda^+}), i_{\tau,-I_{2n}}(\gamma_{\lambda^+})+\nu_{\tau,-I_{2n}}(\gamma_{\lambda^+})]=\emptyset,
\]
with either $\nu_{\tau,-I_{2n}}(\gamma_{\lambda^+})=0$ or $\nu_{\tau,-I_{2n}}(\gamma_{\lambda^-})=0$.
Then for any continuous path $\alpha:[0,1]\to\Lambda$ connecting $\lambda^-$ to $\lambda^+$, there exist a sequence $(t_k)\subset [0, 1]$ converging to some $\bar{t}\in[0,1]$ and solutions $x^k$ of (\ref{e:Bifu-distributedDelay1}) with $\lambda=\alpha(t_k)$ (where $x^k\neq x^{\alpha(t_k)}$ for all $k$) such that $x^k\to x^{\alpha(\bar{t})}$ in $C^1_{\rm loc}(\mathbb{R},\mathbb{R}^n)$ as $k\to\infty$.
Moreover, if $\nu_{\tau,-I_{2n}}(\gamma_{\lambda^+})=0$ then $\alpha(\bar{t})\neq\lambda^+$, and if $\nu_{\tau,-I_{2n}}(\gamma_{\lambda^-})=0$ then $\alpha(\bar{t})\neq\lambda^-$.
  \end{enumerate}
\end{theorem}

 Similarly, applying the first part of Theorem~1.8 
 in \cite{Lu11} to system (\ref{e:distributedDelay2}) yields:

\begin{theorem}[\textsf{Alternative bifurcations of Rabinowitz type}]\label{th:Distri2}
Under Assumption~\ref{ass:Distri1Delay2},
let $\mu$ be an interior point of $\Lambda$ such that
 $\nu_{\tau, -I_{2n}}(\gamma_\mu)\ne 0$, $\nu_{\tau, -I_{2n}}(\gamma_\lambda)=0$
  for all $\lambda$ in some punctured neighborhood of $\mu$ in $\Lambda$,
and the index 
$i_{\tau, -I_{2n}}(\gamma_\lambda)$ takes, respectively, values $i_{\tau, -I_{2n}}(\gamma_\mu)$ and
  $i_{\tau, -I_{2n}}(\gamma_\mu)+ \nu_{\tau, -I_{2n}}(\gamma_\mu)$
 as $\lambda\in\Lambda$ varies in
 two deleted half neighborhoods  of $\mu$.
%$i_{\tau, -I_{2n}}(\gamma_\lambda)$ jumps by $\nu_{\tau, -I_{2n}}(\gamma_\mu)$ across $\mu$.$i_{\tau, -I_{2n}}(\gamma_\lambda)$ takes, respectively, values $i_{\tau, -I_{2n}}(\gamma_\mu)$ and
 Then at least one of the following assertions holds:
 \begin{enumerate}
\item[\rm (i)] The system (\ref{e:Bifu-distributedDelay1}) with $\lambda=\mu$ has a sequence of
solutions $x^{\mu,j}\ne x^\mu$ ($j=1,2,\cdots$)
such that  $x^{\mu,j}\to x^\mu$ in $C^1_{\rm loc}(\mathbb{R},\mathbb{R}^n)$ as $j\to\infty$.

\item[\rm (ii)] For every $\lambda\in\Lambda\setminus\{\mu\}$ sufficiently close to $\mu$, there exists a solution $\bar{x}^\lambda\ne x^\lambda$ of
(\ref{e:Bifu-distributedDelay1}) with parameter $\lambda$
such that $\bar{x}^\lambda\to x^\mu$ in $C^1_{\rm loc}(\mathbb{R},\mathbb{R}^n)$ as $\lambda\to \mu$.

\item[\rm (iii)]
For a given neighborhood $\mathcal{W}$ of $x^\mu|_{[0,\tau]}$ in $C^1([0,\tau];\mathbb{R}^{n})$,
there exists a one-sided neighborhood $\Lambda^0$ of $\mu$ (either $\Lambda^0\subset(\mu-\delta,\mu]$ or $\Lambda^0\subset[\mu,\mu+\delta)$) such that
for any $\lambda\in\Lambda^0\setminus\{\mu\}$, system (\ref{e:Bifu-distributedDelay1}) with parameter $\lambda$
has at least two distinct solutions $\bar{x}^\lambda\neq x^\lambda$ and $\hat{x}^\lambda\neq x^\lambda$ satisfying
$\bar{x}^\lambda|_{[0,\tau]}\in \mathcal{W}$ and $\hat{x}^\lambda|_{[0,\tau]}\in \mathcal{W}$.
Moreover, if $\nu_{\tau, -I_{2n}}(\gamma_\mu)>1$ and (\ref{e:Bifu-distributedDelay1}) with parameter $\lambda$
has only finitely many solutions whose restrictions to $[0,\tau]$ belong to $\mathcal{W}$, then $\bar{x}^\lambda$ and $\hat{x}^\lambda$ can be chosen such that
\begin{eqnarray*}
&&\int^{\tau}_0\left[\frac{1}{2}(J_n\dot{\bar{v}}^\lambda(t), \bar{v}^\lambda(t))_{\mathbb{R}^{2n}}+ 2G(\lambda, \bar{x}^\lambda(t))+ F(\lambda, t, \bar{y}^\lambda(t))\right]dt\\
&&\neq \int^{\tau}_0\left[\frac{1}{2}(J_n\dot{\hat{v}}^\lambda(t), \hat{v}^\lambda(t))_{\mathbb{R}^{2n}}+
G(\lambda, \hat{x}^\lambda(t))+ F(\lambda, t, \hat{y}^\lambda(t))\right]dt,
\end{eqnarray*}
where $\bar{v}^\lambda(t):=(\bar{x}^\lambda(t)^\top, \bar{y}^\lambda(t)^\top)^\top$, $\hat{v}^\lambda(t):=(\hat{x}^\lambda(t)^\top,
\hat{y}^\lambda(t)^\top)^\top$ and
$$
\bar{y}^\lambda(t):=\int_0^\tau\nabla_2 G(\lambda, \bar{x}^\lambda(t-s)) \mathrm{d}s,\quad
\hat{y}^\lambda(t):=\int_0^\tau\nabla_2 G(\lambda, \hat{x}^\lambda(t-s)) \mathrm{d}s.
$$
\end{enumerate}
Furthermore, if, for every $\lambda\in\Lambda$, $F(\lambda,t, y)$ is even in $y$ and $x^\lambda\equiv 0$,
then at least one of  (i) and (iv) occurs, where (i) is as stated
previously, and
\begin{enumerate}
\item[\rm (iv)] There exist left and right neighborhoods $\Lambda^-$ and $\Lambda^+$ of $\mu$ in $\Lambda$
and integers $n^+, n^-\ge 0$ with $n^++n^-\ge \nu_{\tau,-I_{2n}}(\gamma_\mu)$,
such that for $\lambda\in\Lambda^-\setminus\{\mu\}$ (resp. $\lambda\in\Lambda^+\setminus\{\mu\}$),
system (\ref{e:Bifu-distributedDelay1}) with parameter $\lambda$ has at least $n^-$ (resp. $n^+$) distinct pairs of nontrivial solutions $\{v_\lambda^i, -v_\lambda^i\}$, $i=1,\cdots,n^-$ (resp. $i=1,\cdots,n^+$),
which converge to zero in $C^1_{\rm loc}(\mathbb{R},\mathbb{R}^{n})$ as $\lambda\to\mu$.
\end{enumerate}
\end{theorem}

Corollary~1.12 in \cite{Lu11} can also yield corresponding results.

\begin{corollary}\label{cor:bif-per2}
 For a real $\tau>0$, consider continuous functions
$F_0, \hat{F}:  \mathbb{R} \times \mathbb{R}^n \to \mathbb{R}$ and $C^2$ even functions
$G_0, \hat{G}:  \mathbb{R}^{n} \to \mathbb{R}$ satisfying:
\begin{itemize}
\item[\rm (i)] $F_0(t+\tau,x)=F_0(t,-x)$ and $\hat{F}(t+\tau,x)=\hat{F}(t,-x)$ for all $(t,x)\in
\mathbb{R} \times \mathbb{R}^n$.

\item[\rm (ii)] For each $t\in\mathbb{R}$, $F_0(t,\cdot)$ and $\hat{F}(t,\cdot)$ are
  $C^2$ on ${\R}^{n}$.
\item[\rm (iii)] Both the Euclidean gradients, $\nabla_2F_0(t, x)$ and $\nabla_2\hat{F}(t, x)$,
 and the Hessian matrices, $\nabla^2_2F_0(t, x)$ and $\nabla^2_2\hat{F}(t, x)$,
  depend continuously on $(t, x)\in \mathbb{R}\times\mathbb{R}^{n}$.
 \end{itemize}
Let $\bar{x}:\R\to\R^{n}$ satisfy
\begin{enumerate}
\item[\rm (c)]  $\bar{x}(t+\tau)=-\bar{x}(t)$ and
$\dot{\bar{x}}(t) = -\nabla_2 F_0\left(t, \displaystyle\int_0^\tau
\nabla G_0(\bar{x}(t-s)) \,\mathrm{d}s\right)$.

\item[\rm (d)]
$\nabla\hat{G}(\bar{x}(t))=0$ and $\nabla_2\hat{F}(t,\bar{y}(t))=0$
for all $t\in \R$, where $\bar{y}(t):=\int_0^\tau
\nabla G_0(\bar{x}(t-s)) \,\mathrm{d}s$;
moreover either $\nabla^2\hat{G}(\bar{x}(t))>0$ and $\nabla^2_2\hat{F}(t, \bar{y}(t))>0$ for all $t\in \R$,
 or $\nabla^2\hat{G}(\bar{x}(t))<0$ and $\nabla^2_2\hat{F}(t, \bar{y}(t))<0$ for all $t\in \R$.
\end{enumerate}
Take $F(\lambda,t, x)=F_0(t,x)+\lambda\hat{F}(t,x)$ and $G(\lambda, x)=G_0(x)+\lambda\hat{G}(x)$ in (\ref{e:Bifu-distributedDelay1}).
Then
\begin{enumerate}
\item[\rm (A)] $\Sigma:=\{\lambda\in\R\,|\, \nu_{\tau,M}(\gamma_\lambda)>0\}$
is a discrete set in $\R$.
\item[\rm (B)] $(\mu, \bar{v})$ with $\mu\in\R$ is a bifurcation point  for (\ref{e:Bifu-distributedDelay1})
if and only if $\nu_{\tau,-I_{2n}}(\gamma_\mu)>0$.
\item[\rm (C)] For each $\mu\in\Sigma$, %and a small enough $\rho>0$, (\ref{e:case1}) and (\ref{e:case2})
%hold, and therefore
the conclusions in the first part of Theorem~\ref{th:Distri2}
holds, and the second one of Theorem~\ref{th:Distri2} is also true provided that
$\bar{x}=0$ and all $F_0(t,\cdot), \hat{F}(t,\cdot)$ are even.
\end{enumerate}
\end{corollary}

Indeed, let $H_0(t,z) = 2G_0(x) + F_0(t,y)$ and $\hat{H}(t,z) = 2\hat{G}(x) + \hat{F}(t,y)$
for $z=(x^\top,y^\top)^\top\in\mathbb{R}^{2n}$. Then  condition (c) shows that
$\bar{v}(t):=(\bar{x}(t)^\top, \bar{y}(t)^\top)^\top$
satisfies $\bar{v}(t+\tau)=-\bar{v}(t)$ and $\dot{\bar{v}}(t) = J_n\nabla_2 H_0(t, \bar{v}(t))$.
Moreover, from condition (d) we deduce that
 $\nabla_2\hat{H}(t,\bar{v}(t))=0$ for all $t\in \R$, and that
 either $\nabla^2_2\hat{H}(t,\bar{v}(t))>0$  for all $t\in \mathbb{R}$,
 or  $\nabla^2_2\hat{H}(t, \bar{v}(t))<0$ for all $t\in \mathbb{R}$.
Thus the conclusion of Corollary~\ref{cor:bif-per2} follow by
combining Corollary~1.12 of \cite{Lu11} with Lemma~\ref{lem:distri1} immediately.

Similarly, by studying bifurcations of symmetric brake orbits of Hamiltonian systems (\cite{ZhongWL25}),
 we obtain bifurcation results for the following problem:

\begin{equation}\label{e:distributedDelay6}
\begin{cases}
\dot{x}(t) = -\nabla_3 F\!\left(\lambda, t, \int_0^\tau
\nabla_2 G(\lambda, x(t-s)) \,\mathrm{d}s\right),\\[4pt]
x(t+\tau) = -x(t)=x(-t)\quad\forall t.
\end{cases}
\end{equation}
%\begin{equation}\label{e:distributedDelay6}
%\left\{\begin{array}{ll}
%\dot{x}(t) = -\nabla_3 F\left(\lambda, t, \int_0^\tau
%\nabla_2 G(\lambda, x(t-s)) \mathrm{d}s\right),\\
%x(t+\tau) = -x(t)=x(-t)\quad\forall t.
%\end{array}\right.
%\end{equation}

\vspace{5mm}
\noindent{\bf Acknowledgements}\quad 
I would like to thank the anonymous referee for pointing out several errors and for many valuable comments and suggestions.
%Professors Yujun Dong and Chungen Liu for the  helpful discussion, and Professor Yiming Long for helping me to correct an error in the first version.
%I am very grateful to 
%the anonymous referee for his/her many valuable comments and suggestions regarding the mathematical content and the writing, which have greatly improved the revised version.
\vspace{2mm}

\noindent{\bf  Funding}\quad
This work was supported by National Natural Science Foundation of China (Grant No. 12371108).\vspace{2mm}

\noindent{\bf  Author Contribution}\quad
The author confirms sole responsibility for the following: study conception and design, data collection, analysis and interpretation of results, and manuscript preparation.

\noindent{\bf  Conflict of Interest}\quad
The author declares that there is no conflict of interest regarding the publication of this paper.
\vspace{2mm}

%{\bf Funding}\\
%This work was supported by National Natural Science Foundation of China [12371108].

\renewcommand{\refname}{REFERENCES}

\medskip

\begin{tabular}{l}
 School of Mathematical Sciences, Beijing Normal University\\
 Laboratory of Mathematics and Complex Systems, Ministry of Education\\
 Beijing 100875, The People's Republic of China\\
 E-mail address: gclu@bnu.edu.cn\\
\end{tabular}

\end{document}